\documentclass[a4paper]{amsart}
\numberwithin{equation}{section}

\usepackage{amssymb}
\usepackage{mathrsfs}
\usepackage{enumerate, xspace}
\usepackage{marvosym} 
\usepackage[sans]{dsfont}        
\usepackage{upgreek}
 \usepackage[backgroundcolor=blue!20!white%
             ,linecolor=blue!20!white,%
             ,textsize=footnotesize%
             ]{todonotes}
\usepackage[colorlinks]{hyperref}
\usepackage{verbatim}
\usepackage{bm}
\usepackage[all]{xy}
\usepackage{enumitem}

\ifx\XeTeXversion\undefined
\usepackage[english]{babel}
\else
\usepackage{fontspec}
\usepackage{polyglossia}
\setmainlanguage{english}
\usepackage{microtype}
\fi

\makeatletter

\newcommand\mergedsub[2]{#1\sc@sub{#2}}
\newcommand\newsubcommand[3]{\newcommand#1{\mergedsub{#2}{#3}}}
\def\sc@sub#1{\def\sc@thesub{#1}\@ifnextchar_{\sc@mergesubs}{_{\sc@thesub}}}
\def\sc@mergesubs_#1{_{\sc@thesub#1}}

\newcommand\mergedsubC[2]{#1\sc@subC{#2}}

\def\sc@subC#1{\def\sc@thesub{#1}\@ifnextchar_{\sc@mergesubsC}{_{\sc@thesub}}}
\def\sc@mergesubsC_#1{_{#1,\sc@thesub}}

\newcommand\mergedsubCR[2]{#1\sc@subCR{#2}}
\newcommand\newsubcommandCR[3]{\newcommand#1{\mergedsubCR{#2}{#3}}}
\def\sc@subCR#1{\def\sc@thesub{#1}\@ifnextchar_{\sc@mergesubsCR}{_{\sc@thesub}}}
\def\sc@mergesubsCR_#1{_{\sc@thesub,#1}}

\newcommand\mergedsup[2]{#1\sc@sup{#2}}
\newcommand\newsupcommand[3]{\newcommand#1{\mergedsup{#2}{#3}}}
\def\sc@sup#1{\def\sc@thesup{#1}\@ifnextchar^{\sc@mergesups}{^{\sc@thesup}}}
\def\sc@mergesups^#1{^{\sc@thesup,#1}}

\newcommand\mergedsupC[2]{#1\sc@supC{#2}}
\newcommand\newsupcommandC[3]{\newcommand#1{\mergedsupC{#2}{#3}}}
\def\sc@supC#1{\def\sc@thesup{#1}\@ifnextchar^{\sc@mergesupsC}{^{\sc@thesup}}}
\def\sc@mergesupsC^#1{^{#1,\sc@thesup}}

\newcommand\mergedpar[2]{\sc@parC{#1}{#2}}

\def\sc@parC#1#2{\def\sc@thebase{#1}\def\sc@thepar{#2}\@ifnextchar_{\sc@mergeparsub}{\@ifnextchar'{\sc@mergeparprime}{\@ifnextchar^{\sc@mergeparsup}{\sc@thebase(\sc@thepar)}}}}
\def\sc@mergeparsub_#1{\sc@thebase_{#1}(\sc@thepar)}
\def\sc@mergeparprime#1{\sc@thebase #1(\sc@thepar)}
\def\sc@mergeparsup^#1{\sc@thebase^{#1}(\sc@thepar)}
\makeatother

\newtheorem*{thmm}{Main Theorem}

\newtheorem{thm}{Theorem}[section]
\newtheorem{lem}[thm]{Lemma}
\newtheorem{cor}[thm]{Corollary}
\newtheorem{sublem}[thm]{Sub-lemma}
\newtheorem{prop}[thm]{Proposition}

\newtheorem{rem}[thm]{Remark}
\newtheorem{nrem}[thm]{Notational remark}

\newcommand\cC{{\mathcal C}}
\newcommand\cD{{\mathcal D}}

\newcommand\cF{{\mathcal F}}
\newcommand\cG{{\mathcal G}}
\newcommand\cH{{\mathcal H}}

\newcommand\cJ{{\mathcal J}}
\newcommand\cK{{\mathcal K}}

\newcommand\cO{{\mathcal O}}

\newcommand\cR{{\mathcal R}}

\newcommand\cT{{\mathcal T}}
\newcommand\cW{{\mathcal W}}
\newcommand\cZ{{\mathcal Z}}

\newcommand\bE{{\mathbb E}}
\newcommand\bG{{\mathbb G}}
\newcommand\bH{{\mathbb H}}
\newcommand\bL{{\mathbb L}}
\newcommand\bN{{\mathbb N}}
\newcommand\bP{{\mathbb P}}

\newcommand\bR{{\mathbb R}}
\newcommand\bS{{\mathbb S}}
\newcommand\bT{{\mathbb T}}

\newcommand\bZ{{\mathbb Z}}

\newcommand\fkL{{\mathfrak{L}}}

\newcommand\ve{\varepsilon}
\newcommand\eps{\epsilon}
\newcommand\vf{\varphi}
\newcommand\vt{\vartheta}
\newcommand\lt{\varrho}

\newcommand\nc[1]{_{\cC^{#1}}}
\newcommand\nTV{_\text{TV}}

\newcommand\bRp{\bR_+}
\newcommand\intr{\textup{int}\,}
\newcommand\cl{\textup{cl}\,}
\newcommand{\sink}[1]{\theta_{#1,-}}

\makeatletter
\newcommand{\correctPeriodSpacing}[1]{#1\@ifnextchar,{}{\@ifnextchar:{}{\ }}}
\makeatother
\newcommand{\eg}{\correctPeriodSpacing{e.g.}}
\renewcommand{\ae}{\correctPeriodSpacing{a.e.}} 
\newcommand{\ie}{\correctPeriodSpacing{i.e.}}
\newcommand{\resp}{\correctPeriodSpacing{resp.}}
\newcommand{\as}{\correctPeriodSpacing{a.s.}}
\setlist[enumerate,1]{label=\textup{(\alph*)}}

\newcommand{\ispinnedto}{is \pinnedto}
\newcommand{\bepinnedto}{be \pinnedto}

\newcommand{\arepinnedto}{are \pinnedto}

\newcommand{\pinnedto}{located at}
\newcommand{\notpinnedto}{not located at}

\newcommand\nin{\not\in}
\newcommand\Id{{\mathds{1}}}

\newcommand\unifsmall{\bar\varrho}

\newcommand\Lip{\textup{Lip}}
\newcommand{\bigtimes}{{\sf X}}
\newcommand\stable{{s}}  

\newcommand{\cond}{\big|}
\newcommand{\Const}{{C_\#}}
\newcommand{\const}{{c_\#}}
\newcommand{\Kcone}{\gamma}
\newcommand{\Kconeu}{\Kcone^\un}
\newcommand{\Kconec}{\Kcone^\nt}
\newcommand{\cone}{\mathfrak{C}}
\newcommand{\coneu}{\cone^\un}
\newcommand{\conec}{\cone^\nt}
\newcommand{\sve}{\sqrt{\ve}}
\newcommand{\vei}{\ve^{-1}}
\newcommand{\veh}{\ve^{1/2}}

\newcommand{\fa}{\forall\,}
\newcommand{\deh}{{d}}

\newcommand{\st}{\,:\,}
\newcommand{\prob}[1]{\bP\left(#1\right)}

\newcommand{\invr}{^{-1}}
\renewcommand{\ln}{\textup{log}\,}
\newcommand{\dist}{\textup{dist}}
\newcommand{\vdist}{\textup{dist}_{\textup{V}}}
\newcommand{\wDist}{\textup{d}_\text{W}}
\newcommand{\supp}{\textup{supp}\,}
\newcommand{\pint}[1]{{\lfloor#1\rfloor}}
\newcommand{\npi}{\boldsymbol\pi}
\newsubcommand{\npiA}{\boldsymbol\pi}{\alphaset}
\newcommand{\pres}{\gamma}

\newcommand{\alphaset}{\mathcal{A}}
\newcommand{\holexp}{\beta}
\newcommand{\holexpa}{\alpha}
\newcommand{\holexpb}{\beta}
\newcommand{\prodSet}[2]{#1\times#2}
\newcommand{\prodElement}[2]{(#1,#2)}
\newcommand{\malpha}{\mathscr{A}}

\newcommand{\invariant}{X}      
\newcommand{\trapa}{\cT}
\newsubcommandCR{\trap}{\trapa}{\aeps}
\newsubcommandCR{\ttrap}{\tilde\trapa}{\aeps}
\newsubcommandCR{\trapp}{\trapa}{\aeps'}

\newcommand{\avgtheta}[1]{\theta^*_{#1}}
\newcommand{\avgthetaa}[1]{\theta^{*,\privatea}_{#1}}

\newcommand{\fpath}{h}
\newcommand{\dfpath}{\boldsymbol\partial \fpath}
\newcommand{\reachPrivate}{{A}}
\newcommand\freach[1]{\reachPrivate^+_{\aeps,#1}}
\newcommand\freachp[1]{\reachPrivate^+_{\aeps',#1}}
\newcommand\freacha[1]{\reachPrivate^+_{#1}}
\newcommand\fbreach[1]{\reachPrivate^\pm_{\aeps,#1}}

\newcommand\breach[1]{\reachPrivate^-_{\aeps,#1}}
\newcommand\breachp[1]{\reachPrivate^-_{\aeps',#1}}
\newcommand\breacha[1]{\reachPrivate^-_{#1}}

\newcommand\Omp{\Ompx\aeps}
\newcommand\Omm{\Ommx\aeps}
\newcommand{\Ompx}[1]{\Omega^+_{#1}}
\newcommand{\Ommx}[1]{\Omega^-_{#1}}
\newcommand{\ho}{\hat\omega}

\newcommand{\ksp}{{k_*}}

\newcommand{\fbas}[1]{I_{#1,-}}
\newcommand{\bbas}[1]{I_{#1,+}}

\newcommand\numToPrime[1]{%
  \ifnum0=#1\relax%
  \else%
  \ifnum1=#1\relax%
  '\else%
  \ifnum2=#1\relax%
  ''\else%
  \ifnum3p=#1\relax%
  '''\else%
  ^{(#1)}%
  \fi\fi\fi\fi} 
\newcommand\dtt[1]{D_1\numToPrime{#1}}
\newcommand\dttrho[1]{D_2\numToPrime{#1}}

\newcommand\zeros{\textup{Z}}
\newsubcommandCR\nz{n}{\zeros}
\newcommand\nTraps{n_\trapa}

\newcommand{\iid}{\xi}
\newcommand{\iidb}{\eta}
\newcommand{\siid}{\mathbf{\Xi}}

\newcommand{\siidb}{\mathbf{H}}
\newcommand{\counter}{\cG}
\newcommand{\butime}{U} 

\newcommand{\clyapPrivate}{\psi}
\newcommand{\clyap}{\clyapPrivate_*}
\newcommand{\bclyap}{\bar\clyapPrivate_*}
\newcommand{\clyapReg}{\clyapPrivate}
\newcommand{\bclyapReg}{\bar\clyapPrivate}
\newcommand{\clyapn}[1]{\clyapPrivate_{#1}}
\newcommand{\slim}{s_*}
\newcommand{\leftoverm}{\varkappa}
\newcommand{\NCoup}{N}

\newcommand{\qh}{R}

\newcommand{\happy}{H}
\newcommand{\hhappy}{\hat\happy}

\newcommand{\sad}{S}
\newcommand{\uhappy}{\bH}
\newcommand{\usad}{\bS}
\newcommand{\uhhappy}{\hat\uhappy}

\newcommand{\lyapFunc}{\mathcal V}

\newcommand{\un}{\textup{u}}
\newcommand{\nt}{\textup{c}}

\newcommand{\couplingTime}{\cK}
\newcommand{\Bs}{\textup{B}}  
\newcommand{\RBs}{\cR_\Bs}
\newcommand{\Cp}{\textup{C}} 
\newcommand{\TCp}{\cR_\Cp}
 
\newcommand{\Rmd}{\textup{R}}
\newcommand{\Uc}{\textup{U}}
\newcommand{\Rd}{{\textup{U}*}}
\newcommand{\Rg}{\textup{R}} \newcommand{\NRg}{{N_{\Rg}}}

\newcommand{\aeps}{\eps}
\newcommand{\tmod}{\textup{mod}} 

\newsubcommand{\TBas}{T}{I}

\newsubcommand{\TTrap}{T}{\trapa}
\newsubcommand{\TForb}{T}{\text{F}}
\newcommand{\Cn}{\textup{S}}
\newcommand{\NCn}{{N_{\Cn}}} \newsubcommand{\TCn}{T}{\Cn}   
\newcommand{\Sl}{\textup{A}}
\newcommand{\TSl}{\couplingTime_\Sl} \newcommand{\RSl}{\cR_\Sl} 

\newcommand{\Rdp}{\cR_\text{D}}
\newcommand{\Rgen}{\cR}
\newcommand{\Tgen}{\couplingTime}
\newcommand{\RCp}{\couplingTime_{\Cp}}
\newcommand{\Mhr}{N_0}\newcommand{\Thr}{T_0}
\newcommand{\Nesc}[1]{N_{1}} \newcommand{\Tesc}[1]{T_{1}}
\newcommand{\Leb}{\textup{Leb}}
\newcommand{\fm}{\nu}
\newcommand{\fmm}{Z}

\newcommand{\expb}{\Psi} 
\newcommand{\expoCloseness}{\tau}

\newcommand{\smallness}{\gamma}
\newcommand{\smallnesst}{\hat\gamma}

\newcommand{\closeness}{\varrho}

\newcommand\admissiblep[2]{admissible \tpath{#1}{#2}}
\newcommand\eadmissiblep[2]{$\aeps$-admissible \tpath{#1}{#2}}
\newcommand\tpath[2]{$(#1,#2)$-\hspace{0pt}path}
\newcommand\couple{couple}\newcommand\couples{couples}%

\newcommand{\expo}[1]{\exp(#1)}
\newcommand\shiftPar{\kappa}
\newcommand\deviation{\Delta}
\newcommand\Var{\boldsymbol{\upsigma}} 
\newcommand\bVar{\boldsymbol{\hat{\upsigma}}}

\newcommand{\couplingmu}{\boldsymbol\mu}

\newcommand{\privateFell}{\ell} 
\newcommand{\privateFellC}{\ellC}
\newcommand{\fell}{\privateFell}

\newcommand{\fellC}{\privateFellC}
\newcommand{\fellCi}{\privateFellC^{\privateIndex{i}}}
\newcommand{\fellCa}{\mergedsupC\privateFellC{\privatea}}
\newcommand{\fellCb}{\mergedsupC\privateFellC{\privateb}}
\newcommand{\fellf}[1]{\mergedpar{\privateFell}{#1}}

\newcommand{\fellaf}[1]{\mergedpar{\privateFell^{\privatea}}{#1}} 

\newcommand{\fellbf}[1]{\mergedpar{\privateFell^{\privateb}}{#1}} 

\newcommand{\fellCf}[1]{\mergedpar{\privateFellC}{#1}}
\newcommand{\fellCfa}[1]{\mergedpar{\mergedsupC{\privateFellC}{\privatea}}{#1}}
\newcommand{\fellCfi}[1]{\mergedpar{\privateFellC^{i}}{#1}}
\newcommand{\fellCfb}[1]{\mergedpar{\privateFellC^{\privateb}}{#1}}

\newcommand{\SFF}[2]{(#2,#1)}
\newsubcommand{\liftMap}{\hat F}{\ve}
\newcommand{\alphaMapPrivate}{\boldsymbol\alpha}
\newcommand{\alphaMap}[1]{\alphaMapPrivate_{#1}}
\newcommand{\inj}{\boldsymbol\iota}
\newcommand{\privateIndex}[1]{#1}

\newcommand{\privatea}{\privateIndex{0}}
\newcommand{\privateb}{\privateIndex{1}}
\newcommand{\privatei}{\privateIndex{i}}

\newcommand{\pFve}{\mergedsub{F}{\ve}}

\newcommand{\spc}[1]{c_#1}

\newcommand{\eqc}[1]{\left[#1\right]}
\newcommand{\stdf}{\fkL}
\newcommand{\subfamily}{\subset}

\newcommand{\stdfSet}[2]{\bL_{#1}\ifx&#2&\else(#2)\fi}
\newcommand{\stdpSet}[1]{L_{#1}}
\newsupcommandC{\stdfa}{\stdf}{\privatea}
\newsupcommandC{\stdfb}{\stdf}{\privateb}
\newsupcommandC{\stdfi}{\stdf}{\privateIndex{i}}
\newsupcommand{\stdfpa}{\stdf}{\prime\privatea}
\newsupcommand{\stdfpb}{\stdf}{\prime\privateb}
\newsupcommand{\stdfpi}{\stdf}{\prime\privateIndex{i}}

\newsupcommandC{\tstdfa}{\tilde\stdf}{\privatea}
\newsupcommandC{\tstdfb}{\tilde\stdf}{\privateb}
\newsupcommandC{\tstdfi}{\tilde\stdf}{\privateIndex{i}}
\newcommand{\stdfC}{\duoPrivate{\stdf}{2}}
\newsupcommandC{\stdfCa}{\stdfC}{\privatea}
\newsupcommandC{\stdfCb}{\stdfC}{\privateb}
\newsupcommandC{\stdfCi}{\stdfC}{\privatei}
\newcommand{\tstdfC}{\widetilde\stdfC{}}

\newsupcommandC{\bstdfa}{\bar\stdf}{\privatea}

\newcommand{\privateStepSeq}[3]{#1^#3_{#2}}
\newcommand{\coupledStepSeq}[2]{\privateStepSeq{#1}{#2}{\Cp}}
\newcommand{\uncoupledStepSeq}[2]{\privateStepSeq{#1}{#2}{\Uc}}

\newcommand\prepC[1]{\left(#1\right){}}
\newcommand{\bootstrapped}[2]{#1^\Bs}
\newcommand{\notBootstrapped}[2]{#1^\Rmd}

\newcommand{\datingC}[2]{{#1}^{[#2]}}
\newcommand{\availableC}[2]{{#1}^{[#2]*}}

\newcommand{\coupledStepIniPrivate}[2]
{#1|_{#2}}

\newcommand{\auxFamily}[2]{#1_{[#2]}}

\newcommand{\MCS}[1]{M_{[#1]}}
\newcommand{\MUS}[1]{M_{[#1]*}}

\newcommand{\bGbb}{\bG^\privateb}
\newcommand{\bGs}[1]{\bG^{\privateIndex{#1}}}
\newcommand{\Ga}{G^\privatea}
\newcommand{\Gb}{G^\privateb}
\newcommand{\Gi}{\bG^\privateIndex{i}}
\newcommand{\GRi}{\bG^\privateIndex{i}_\textup{R}}
\newcommand{\GLi}{\bG^\privateIndex{i}_\textup{L}}

\newcommand{\bGa}{\bG^\privatea}
\newcommand{\bGb}{\bG^\privateb}
\newcommand{\bGi}{\bG^\privateIndex{i}}

\newcommand{\Gs}[1]{G^{\privateIndex{#1}}}
\newcommand{\dGs}[1]{G^{\privateIndex{#1}}{}'}

\newcommand{\duoPrivate}[2]{\underline{#1\mkern-#2mu}\mkern#2mu}

\newsupcommandC{\ella}{\ell}{\privatea}
\newsupcommandC{\ellb}{\ell}{\privateb}
\newcommand{\ellC}{\duoPrivate{\ell\,}{4}}
\newsupcommandC{\ellCa}{\ellC}{\privatea}
\newsupcommandC{\ellCb}{\ellC}{\privateb}
\newsupcommandC{\ellCi}{\ellC}{\privatei}

\newsupcommandC{\elli}{\ell}{\privateIndex{i}}
\newsupcommandC{\gapi}{\gap}{\privateIndex{i}}

\newcommand{\ellLi}{\ell_\textup{L}^\privateIndex{i}}
\newcommand{\ellRi}{\ell_\textup{R}^\privateIndex{i}}

\newcommand{\rhoi}{\rho^\privateIndex{i}}
\newcommand{\rhoLi}{\rho^\privateIndex{i}_\textup{L}}
\newcommand{\rhoRi}{\rho^\privateIndex{i}_\textup{R}}
\newcommand{\rhoa}{\rho^\privatea}
\newcommand{\rhob}{\rho^\privateb}
\newcommand{\ua}{u^\privatea}
\newcommand{\ub}{u^\privateb}
\newcommand{\ui}{u^\privateIndex{i}}
\newcommand{\pa}{p^\privatea}
\newcommand{\pb}{p^\privateb}

\newcommand{\pii}{p^\privateIndex{i}} 
\newcommand{\alphaseti}{\alphaset^\privateIndex{i}}
\newsupcommandC{\alphaseta}{\alphaset}{\privatea}
\newsupcommandC{\alphasetb}{\alphaset}{\privateb}
\newsupcommandC{\alphasetg}{\alphaset}{\text{g}}
\newsupcommandC{\balphasetg}{\bar\alphaset}{\text{g}}
\newsupcommandC{\balphasetng}{\bar\alphaset}{\text{b}}
\newsupcommandC{\alphasetng}{\alphaset}{\text{b}}

\newcommand{\gap}{U}

\newcommand{\fmi}{\fm^\privateIndex{i}}

\newsupcommandC{\fma}{\fm}{\privatea} 
\newsupcommandC{\fmb}{\fm}{\privateb}

\newcommand{\mi}{m^\privateIndex{i}}
\newcommand{\mLi}{m_\textup{L}^\privateIndex{i}}
\newcommand{\mLa}{m_\textup{L}^\privatea}
\newcommand{\mRi}{m_\textup{R}^\privateIndex{i}}
\newcommand{\mRa}{m_\textup{R}^\privatea}
\newcommand{\ma}{m^\privatea}
\newcommand{\mb}{m^\privateb}

\newcommand{\consth}{\cD}
\newcommand{\caa}{c_*}
\newcommand{\aaa}{a_*^\privatea} 
\newcommand{\ab}{a_*^\privateb}
\newcommand{\ba}{b_*^\privatea}
\newcommand{\bb}{b_*^\privateb}
\newcommand{\ai}{a_*^\privateIndex{i}}
\newcommand{\bi}{b_*^\privateIndex{i}}
\newcommand{\xii}{x^\privateIndex{i}}
\newcommand{\xa}{x^\privatea}
\newcommand{\xb}{x^\privateb}

\newcommand{\thetaa}{\theta^\privatea}
\newcommand{\thetab}{\theta^\privateb}

\newcommand{\VH}[1]{W_{#1,-}}
\newcommand{\VS}[1]{W_{#1,+}}
\newcommand{\VSshort}{W_{+}}
\newcommand{\fuzz}{\varkappa}
\newcommand{\fuzziness}{(1/2,2)}
\newcommand{\soRad}{c_S}

\newcommand{\holoDiff}{{h}}
\newcommand{\holoMap}{\cH}

\newcommand{\mPrivate}{m}
\newsubcommand{\mC}{\mPrivate}{\Cp}
\newsubcommand{\msC}{\tilde{\mPrivate}}{\Cp}

\newcommand{\mB}{m_\Bs}

 \newcommand{\MCp}[1]{M_{\Cp_{#1}}}

\newcommand{\tI}{\textrm{I}} \newcommand{\tII}{\textrm{I}\textrm{I}}\newcommand{\tIII}{\textrm{I}\textrm{I}\textrm{I}}
\newcommand{\largestart}{\bar k}
\newcommand\slo{z}

\begin{document}
\title[Statistical properties]{Statistical properties of mostly contracting fast-slow
  partially hyperbolic systems.}
  \author{Jacopo De Simoi}
\address{Jacopo De Simoi\\
  Department of Mathematics\\
  University of Toronto\\
  40 St George St. Toronto, ON M5S 2E4} \email{{\tt jacopods@math.utoronto.ca}}
\urladdr{\href{http://www.math.utoronto.ca/jacopods/}{http://www.math.utoronto.ca/jacopods}}
\author{Carlangelo Liverani}
\address{Carlangelo Liverani\\
  Dipartimento di Matematica\\
  II Universit\`{a} di Roma (Tor Vergata)\\
  Via della Ricerca Scientifica, 00133 Roma, Italy.}  \email{{\tt liverani@mat.uniroma2.it}}
\urladdr{\href{http://www.mat.uniroma2.it/~liverani/}{http://www.mat.uniroma2.it/~liverani}}

\thanks{This work would not exist without the many and fruitful discussions which both
  authors had with Dmitry Dolgopyat, who would very well deserve to be listed among the
  authors.  The authors are also glad to thank Piermarco Cannarsa, Bastien Fernandez, Ian
  Morris, Christophe Poquet and Ke Zhang for their very useful comments. We also thank the
  anonymous referees for many very helpful suggestions among which to include the examples
  and discussion in Section~\ref{sec:examples} and for pointing out several imprecisions
  in a previous version.  This work was supported by the European Advanced Grant
  Macroscopic Laws and Dynamical Systems (MALADY) (ERC AdG 246953) and by NSERC. Both
  authors are pleased to thank the Fields Institute in Toronto, Canada, (where this work
  started) for the excellent hospitality and working conditions provided during the spring
  semester 2011.}

\begin{abstract}
  {\ifx\XeTeXversion\undefined \relax \else \normalsize \small 
    \fi%
    We consider a class of $\cC^{4}$ partially hyperbolic systems on $\bT^2$ described by
    maps $F_\ve(x,\theta)=(f(x,\theta),\theta+\ve\omega(x,\theta))$ where
    $f(\cdot,\theta)$ are expanding maps of the circle.  For sufficiently small $\ve$ and
    $\omega$ generic in an open set, we precisely classify the SRB measures for $F_\ve$
    and their statistical properties, including exponential decay of correlation for
    H\"older observables with explicit and nearly optimal bounds on the decay rate.}
\end{abstract}
\keywords{Averaging, partially hyperbolic, decay of correlations, metastability}
\subjclass[2000]{37A25, 37C30, 37D30, 37A50, 60F17}
\maketitle
\section{Introduction}
There has been a lot of attention lately to the properties of {\em partially hyperbolic}
systems and their perturbations. The main emphasis has been on geometric properties and on
stable ergodicity. In the latter field many deep results have been obtained starting
with~\cite{MPS94, pugh-shub97}. Nevertheless, it is well known, at least since the work of
Krylov~\cite{Krylov79}, that for many applications ergodicity is not sufficient and some
type of mixing (usually in the form of effective quantitative estimates) is of paramount
importance. Unfortunately, very few results are known regarding stronger statistical
properties for partially hyperbolic systems.  More precisely, we have some results in the
case of mostly expanding central direction~\cite{ALP05}, and mostly contracting central
direction~\cite{Dimacontract, deCastro02}.  For central direction with zero Lyapunov
exponents (or close to zero) there exist quantitative results on exponential decay of
correlations only for group extensions of Anosov maps and Anosov flows~\cite{Dima-group,
  Che98, Dima98, Liverani04, Tsujii10}, but none of them apply to an open class (although
some form of rapid mixing is known to be typical for large classes of
flows~\cite{FMT07, Melbourne}). It would then be of great interest in the field of Dynamical
Systems, but also, e.g., for {\em non-equilibrium Statistical Mechanics}, to extend the
class of systems for which statistical properties are well understood. See~\cite{pesin10,
  liverani10} for a discussion of some aspects of these issues and~\cite{ruelle} for an
interesting application to non-equilibrium Statistical Mechanics.

Another argument of renewed interest is {\em averaging theory}, due to new powerful
results~\cite{DimaAveraging, Dima12} and, among others, new applications of clear
relevance for non-equilibrium Statistical Mechanics~\cite{Dimaliverani}. Yet, averaging
theory only provides information on a given time scale; a natural and very relevant
question is {\em what happens at longer, possibly infinitely long, time scales}. Such
information would be encoded in the SRB measure and its statistical properties.  Hence we
have a natural connection with the above mentioned open problem in partially hyperbolic
systems (as indeed the slow variables can often be considered as central directions).

The above program can be carried out for deterministic systems subject to small random
perturbations (\eg stochastic differential equations with vanishing diffusion
coefficient), where the fast variable (modeled by Brownian Motion), is (in some sense)
infinitely fast~\cite{WeFr}.  In our setting, on the one hand the motion of the fast
variable is deterministic, although chaotic; on the other hand, the motion of the slow
variable is not hyperbolic (hence one cannot implement strategies based on the strong
chaoticity of the unperturbed motion and the essential irrelevance of the perturbation,
where many powerful technical tool are available, starting with~\cite{Kifer88}). It is
then not so surprising that a preliminary step needed to carry out the above program is to
establish, in a very precise technical sense, to which extent the motion of the fast
variable can be confused with the motion of a random variable. In particular, this
requires to go well beyond the known results on averaging and deviations from the average
that can be found in~\cite{Kifer92, DimaAveraging}.  One needs the analog of a Local
Central Limit Theorem for the process of the fluctuations around the average. This is in
itself a non trivial endeavor which has been first accomplished, for a simple but relevant class of
systems, in~\cite{DeL1}.

Finally, in analogy with the stochastic case, see~\cite{WeFr}, one can expect {\em
  metastable} behavior.\footnote{\label{foot:meta} By a {\em metastable} system here we
  mean a situation in which two time scales are present: one, the short one, in which the
  system seems to have several invariant measures, and hence to lack ergodicity, and a
  longer time scale in which it turns out that the system has indeed only one relevant,
  mixing, invariant measure. This can be seen experimentally by the presence of two time
  scales in the decay of correlations.} Indeed, metastability is a phenomenon that has
been widely investigated in the stochastic setting, see~\cite{Olivieri-Vares05} for a
detailed account. Yet, to our knowledge, no results whatsoever exist in the deterministic
setting. The strongest results in such a direction can be found in~\cite{Kifer09} where it is
proven, for a fairly large class of systems, that the transition between basins of
attractions takes place only at exponentially long time scales, thus one has a clear
indication of the existence of, at least, two time scales in such systems. Yet, the
results in~\cite{Kifer09} are not sufficient to investigate the longer times needed to
establish a full metastability scenario (in the sense of Footnote~\ref{foot:meta}). It is
then natural to ask if metastability results hold in the present deterministic setting.
Of course, to answer to such questions, one needs to combine good {\em Large Deviation
Estimates}\footnote{ These were first derived in~\cite{Kifer09}. But a much more refined
  and quantitative version can be found in~\cite{DeL1}.} with a
precise quantitative understanding of the mixing properties of the local dynamics. This is
the topic of this paper and it clarifies the connection of metastability with the above
mentioned general problems. Accordingly, metastability (together with partial
hyperbolicity, non-equilibrium statistical mechanics and averaging) constitutes a fourth
natural and important line of research among the ones that motivate and converge in this
paper.

To carry out the above research program it turns out that a preliminary understanding of
the long time properties of the averaged motion is necessary. In general, this is an
impossible task, since the averaged system can be essentially any ordinary differential
equation (ODE). To simplify matters, as a first step, we consider the simplest possible
averaged dynamics: a one dimensional ODE on the circle with finitely many, non degenerate,
equilibrium points.
\newpage
\section{The model and the results}\label{sec:results}
\subsection{Our model}
Let us now introduce the class of systems we are going to investigate.  For $\ve>0$ we
consider the maps $F_\ve\in\cC^{4}(\bT^2,\bT^2)$ defined by
\begin{equation}\label{eq:map}
  F_\ve(x,\theta)=(f(x,\theta),\theta+\ve \omega(x,\theta) \mod 1)
\end{equation}
where $f$ and $\omega$ are both $\cC^{4}$ functions.  We assume that
$f(\cdot,\theta) = f_\theta:\bT\to\bT$ is an orientation-preserving expanding map for each
$\theta\in\bT$; moreover, by possibly replacing $F_\ve$ by a suitable iterate, we will always
assume that $\partial_x f\ge\lambda>2$.
\begin{rem}
  In the sequel, we will take $\ve$ to be fixed and sufficiently small depending on
  $f$ and $\omega$.  We could indeed regard~\eqref{eq:map} as an
  arbitrary perturbation of the map $\widetilde F(x,\theta) = (\tilde f(x,\theta),\theta)$ by
  $\ve(g(x,\theta),\omega(x,\theta))$; in fact, since a sufficiently small perturbation of
  a family of expanding maps is still a family of expanding maps, one could always
  set $f=\tilde f+\ve g$.  However, in order to treat this slightly more general case, we would need to show
  that our conditions on the smallness of $\ve$ depend on $\omega$
  uniformly in a neighborhood of $\tilde f$.  We do not pursue this for sake of simplicity.
\end{rem}
Since $f_\theta$ are expanding maps of the circle, there exists a unique family of
absolutely continuous (SRB) $f_\theta$-invariant probability measures whose densities we
denote by $\rho_\theta$.  By our regularity assumptions on $F_{\ve}$ it follows (see
e.g.~\cite[Section~8]{GL06}) that $\rho_\theta$ is a $\cC^3$-smooth family of
$\cC^3$-densities.

Let us recall a few well-known definitions: a function$\phi\in\cC^0(\bT)$ is said to
be a \emph{(continuous) coboundary (with respect to a map $f:\bT\to\bT$)} if there exists
$\beta\in\cC^0(\bT)$ so that
\begin{align*}
\phi=\beta-\beta\circ f.
\end{align*}
Two functions $\phi_1,\phi_2\in\cC^0(\bT)$ are said to be \emph{cohomologous (with respect
  to $f$)} if their difference $\phi_2-\phi_1$ is a coboundary (with respect to $f$).

Our first assumption on $F_\ve$ is:
\begin{enumerate}[start=0,label=\textup{(A\arabic*)},ref=(A\arabic*)]
\item\label{a_noCobo} for each $\theta\in\bT$, the function $\omega(\cdot,\theta)$ is not
  cohomologous to a constant function with respect to $f_\theta$.
\end{enumerate}
Let us now define $\bar\omega(\theta)=\int_{\bT}\omega(x,\theta)\rho_\theta(x) \deh x$.
Observe that our earlier considerations concerning the smoothness of the family $\rho_\theta$
imply that $\bar\omega\in\cC^3(\bT)$.

Our second standing assumption reads
\begin{enumerate}[label=\textup{(A\arabic*)},,ref=(A\arabic*),resume]
  \item\label{a_discreteZeros} $\bar\omega$ has a non-empty discrete set of non-degenerate zeros.
\end{enumerate}
In particular, we assume the set of zeros to be given by
${\{\theta_{i,\pm}\}}_{i\in\bZ/\nz\bZ}$ with $\nz\in\bN$, $\bar\omega'(\theta_{i,+})>0$ and
$\bar\omega'(\theta_{i,-})<0$; we assume, having fixed an orientation of $\bT$, that the
indexing is so that for any $k$, $\theta_{k,+}<\theta_{k,-}<\theta_{k+1,+}$, where all
indices $k$ are taken $\bmod\,\nz$.

In Section~\ref{s_geometry} we will see that the map $F_\ve$ has an invariant center
distribution $(\slim(x,\theta),1)$.  Let us now introduce the function
\begin{align}\label{e_definitionPsi}
  \clyap(x,\theta)=\partial_\theta\omega(x,\theta)+\partial_x
  \omega(x,\theta) \slim(x,\theta)
\end{align}
together with its average
$\bclyap(\theta)=\int_{\bT}\clyap(x,\theta)\rho_{\theta}(x)\deh x$.  As it is made clear
by~\eqref{eq:central-lyap} and subsequent discussion, $1+\ve\clyap $ is the one
step-contraction (or expansion) in the center direction.  Remark that the system is
non-uniformly hyperbolic and it is far from obvious how to compute the central Lyapunov
exponent for Lebesgue-\ae point.  Our third assumption will, eventually, allow us
to prove that the center Lyapunov exponent is Lebesgue-\as negative:
\begin{enumerate}[label=\textup{(A\arabic*)},,ref=(A\arabic*),start = 2]
  \item\label{a_almostTrivial}\label{a_almostTrivialp} $\max_{k\in\{1,\cdots,\nz\}}\bclyap(\theta_{k,-}) = -1$.
\end{enumerate}
\begin{rem}
  Observe that if $\max_{k\in\{1,\cdots,\nz\}}\bclyap(\theta_{k,-}) < 0$, it is always
  possible to rescale\footnote{ That is, for $\varrho_r\ne0$, we let
    $\omega\mapsto\varrho_r\omega$ and $\ve\mapsto\varrho_r\invr\ve$ so that the product
    $\ve\omega$ is left unchanged, together with all other dynamically defined quantities,
    e.g. $\slim(\theta)$ (see~\eqref{eq:central-lyap-s} and following equations).  Observe
    that under this rescaling,~\eqref{e_definitionPsi} gives
    $\clyap\mapsto\varrho_r\clyap$.} $\omega$ and $\ve$ so that~\ref{a_almostTrivial}
  holds.  In other words, the $-1$ on the right hand side is just a normalization which
  can be achieved without loss of generality.
\end{rem}
\begin{rem}
  Observe moreover that one can explicitly compute an arbitrarily precise approximation of
   $\clyap$ (see~\eqref{eq:regularizedpsi} and Remark~\ref{rem:newpsi}).  Thus (in view of
  the above remark) our condition~\ref{a_almostTrivial} is in principle explicitly
  checkable for a given map.
\end{rem}
The above condition is not optimal, it implies that the center direction is mostly
contracting on average (see Lemma~\ref{l_goodSet}) in a neighborhood of \emph{every} sink.
Of course, a negative Lyapunov exponent in the center direction could also emerge from the
interaction between different sinks but this would be much harder to investigate.
\begin{rem} It is quite possible that~\ref{a_almostTrivial} is not necessary and our Main
  Theorem holds also in the case of zero or positive central Lyapunov exponent. Yet, its
  proof clearly would require a different approach and it remains the subject of further
  studies, see Section~\ref{sec:conclusion} for more details.
\end{rem}
\begin{rem}
  Observe that if $\partial_\theta f = 0$ identically, then~\ref{a_almostTrivial} follows
  by~\ref{a_discreteZeros} since $\slim=0$, $\partial_\theta\rho_\theta=0$ and hence
  $\bar\omega'(\theta)=\bclyap(\theta)$. Thus, in such cases, the center Lyapunov exponent
  turns out to be determined by the averaged system and it is always negative.  This is
  not true in general if $\partial_\theta f\neq 0$, as it is clearly shown in
  Section~\ref{sec:no-ex}: hence the need of assumption~\ref{a_almostTrivial}.
\end{rem}

In order to state our Main Theorem, it is necessary to introduce a few more definitions,
which force us to take a (very minor) detour through non-smooth analysis (see \eg \cite{Clarke}).
A Lipschitz function $\fpath\in\cC^\Lip([0,T],\bT)$ is said to be a
\emph{\tpath{\thetaa}{\thetab}} if it satisfies the boundary conditions
$\fpath(0) = \thetaa$, $\fpath(T) = \thetab$; $T$ will be referred to as the \emph{length}
of $\fpath$.  Recall that Rademacher's Theorem implies that a Lipschitz function $\fpath$
is differentiable everywhere except on a set of zero Lebesgue measure which we denote with
$E_\fpath$.  For each $s\in[0,T]$ let us define the \emph{Clarke generalized derivative} of
$\fpath$ as the set-valued function:
\begin{align*}
  \dfpath(s) = \text{hull}\{\lim_{k\to\infty} \fpath'(s_k) \st s_k\to s \text{ and } \{s_k\}\subset
  [0,T]\setminus E_\fpath
  \}.
\end{align*}
The set $\dfpath(s)$ is compact and non-empty for any $s\in[0,T]$ (see \eg~\cite[Chapter
2, Proposition 1.5]{Clarke}) and so is its
graph, \ie the set $\bigcup_{s\in[0,T]} \{s\}\times\dfpath(s)\subset[0,T]\times\bR$ (this
follows from the definition and from a standard Cantor diagonal argument).  Moreover if
$s\nin E_\fpath$, then $h'(s)\in\dfpath(s)$ and if $h'$ is continuous at $s$ we have
$\dfpath(s) = \{h'(s)\}$ (see~\cite[Chapter 2, Proposition 3.1]{Clarke}).

We say that a path $\fpath$ of length $T$ is \emph{admissible} if for any $s\in[0,T]$,
$\dfpath(s)\subset\intr\Omega(h(s))$, where for any $\theta\in\bT$, we define the (convex
and compact) set
\begin{align}\label{e_definitionOmega}
  \Omega(\theta)=\{\mu(\omega(\cdot,\theta))\,| \,\mu\text{ is a
  }f_\theta\text{-invariant probability}\}.
\end{align}

We can now state two more conditions:
\begin{enumerate}[label=\textup{(A\arabic*)},ref=(A\arabic*),resume]
  \item\label{a_fluctuation} there exists $i\in\{1,\cdots,\nz\}$ so that for any
  $\theta\in\bT$, there exists an \admissiblep{\theta}{\theta_{i,-}}{}.  We can always
  assume, without loss of generality, that $i=1$.
\end{enumerate}
Observe that, under conditions~\ref{a_noCobo} and~\ref{a_discreteZeros}, condition
\ref{a_fluctuation} is trivially satisfied if $\nz=1$ (see
Section~\ref{ss_furtherProperties}).  In cases where~\ref{a_fluctuation} does not hold, we
can still obtain interesting results under the following additional non-degeneracy
condition:
\begin{enumerate}[label=\textup{(A\arabic*)},ref=(A\arabic*),start = 4]
\item \label{a_fluctuationGap} the set of zeros $\{\theta_{i,\pm}\}_{i = 1,\cdots,\nz}$ of
  $\bar\omega$ cuts $\bT$ in $2\nz$ open intervals: any such
  interval $J$ satisfies one of the following two properties
  \begin{enumerate}[label = \textup{\roman*.},ref = \textup{\roman*}]
  \item \label{p_twoway} for any $\theta\in J$, $0\in\intr\Omega(\theta)$
  \item \label{p_oneway} there exists $\theta\in J$ so that $0\nin\Omega(\theta)$.
  \end{enumerate}
\end{enumerate}
\subsection{Our results} We are now finally ready to state our main results.
\begin{thmm}
  Under assumptions~\ref{a_noCobo},~\ref{a_discreteZeros},~\ref{a_almostTrivialp}
  and~\ref{a_fluctuationGap}, if $\ve>0$ is sufficiently small, $F_\ve$ admits at most
  $\nz$ SRB measures.

  Under assumptions~\ref{a_noCobo},~\ref{a_discreteZeros},~\ref{a_almostTrivialp}
  and~\ref{a_fluctuation}, if $\ve > 0$ is sufficiently small, $F_\ve$ admits a unique SRB
  measure $\mu_\ve$; this measure enjoys exponential decay of correlations for H\"older
  observables.  
  More precisely: there exist $C_1,C_2,C_3,C_4>0$ (independent of $\ve$) such that, for
  any $\holexpa\in (0,3]$ and $\holexp\in(0,1]$, any two functions
  $A\in\cC^\holexpa(\bT^2)$ and $B\in \cC^\holexpb(\bT^2)$:
  \begin{align*}
    \left| \Leb(A\cdot B\circ F_\ve^n) - \Leb(A)\mu_\ve(B)\right|\leq
    C_1\sup_\theta\|A(\cdot, \theta)\|\nc{\holexpa}\sup_x\|B(x,\cdot)\|\nc{\holexpb}
    e^{-{\holexpa\holexpb} c_\ve n},
  \end{align*}
  where
  \begin{equation}\label{e_lowerBoundRate}
    c_\ve=
    \begin{cases}
      C_2\ve/\log\vei&\text{if }\nz=1,\\
      C_3\exp(-C_4\vei)&\text{otherwise.}
    \end{cases}
  \end{equation}
\end{thmm}
The proof of our Main Theorem will be given in Section~\ref{ss_proofMainTheorem} for a
slightly stronger version (detailing the exact number and properties of the SRB measures
in the case~\ref{a_fluctuation} does not hold, but~\ref{a_fluctuationGap} does) which, to
be properly stated, needs the introduction of several extra notations (see
Theorem~\ref{t_mainTheoremForTraps} and Corollary~\ref{c_mainTheoremCorollary} for more
details).  Now, we provide a number of remarks to clarify and put into context the above
result.
\begin{rem}\label{r_endomorphismSRB}
 There is a long lasting controversy regarding the definition of SRB measures
    (see e.g.~\cite{young02}).  Since endomorphisms do not have an unstable foliation, we
    will follow common practice~(see e.g.~\cite[Corollary 2]{DimaPH}) and say that
    $\mu_\ve$ is an SRB measure if it is $F_\ve$-invariant and its \emph{ergodic basin}
  \begin{align*}
    B(\mu_\ve) = \left\{p\in\bT^{2}\st\frac1n\sum_{k = 0}^{n-1}\delta_{F_\ve^k(p)}\to\mu_\ve\text{ weakly
    as }n\to\infty\right\}
  \end{align*}
  has positive Lebesgue measure. These measures are also called {\em physical measures}.
\end{rem}
\begin{rem}\label{r_completeness}
  As a particular case of~\ref{a_fluctuationGap} (\ie every interval $J$ satisfies
  property~\ref{p_twoway}) let us introduce the following condition\footnote{
    In~\cite{Kifer09}, $\omega$ is said to be \emph{complete} at $\theta$ if this
    condition holds at $\theta$.}
  \begin{enumerate}[label=\textup{(A\arabic*$^*$)},,ref=(A\arabic*$^*$),start=4]
    \item\label{a_completeness} for any $\theta\in\bT$, $0\in\intr\Omega(\theta)$;
  \end{enumerate}
  Condition~\ref{a_completeness} immediately implies~\ref{a_fluctuation} (it is actually
  stronger and assuming it in our Main Theorem would imply the existence of a unique SRB measure with $\ve$-dense support).
  Most importantly, it can in principle be checked in concrete examples as it
  suffices\footnote{ The equivalence holds since the measures supported on periodic orbits
    are weakly dense in the set of the invariant measures~\cite{Parthasarathy61}.}  to
  find, for every $\theta\in\bT$, two periodic orbits of $f_\theta$ so that the average of
  $\omega(\cdot,\theta)$ is positive on one of them and negative on the other one.
  Moreover, it is obvious to observe that, for any given $F_0$, the set
  $\{\omega\st \text{\ref{a_completeness} holds}\}$ contains an open set in the
  $\cC^4$-topology.  Finally, it is immediate to check that Condition~\ref{a_completeness}
  also implies Condition~\ref{a_noCobo}.
\end{rem}
\begin{rem}\label{r_optimality}
  A natural question is whether the values of $c_\ve$ in~\eqref{e_lowerBoundRate} are
  optimal or not.  The answer is ``essentially yes". To clarify this, in
  Section~\ref{sec:examples} we give some explicit examples to which our Theorem
  applies and we provide a lower bound for the decay of correlations in such
  examples. Also we take the opportunity to compare our situation with the case of small
  stochastic perturbations discussed by Wentzell-Freidlin~\cite{WeFr} uncovering both
  strong similarities and fundamental differences.  We summarize our findings in Remarks~\ref{rem:diff0-fw},~\ref{rem:diff1-fw},~\ref{rem:no-FW}
  and~\ref{rem:no-lyap-fw}.
\end{rem}
\begin{rem}\label{rem:lebtd}
  For simplicity our Main Theorem is stated for the Lebesgue measure.  In fact it holds,
  for a much wider class of measures, \ie measures that can be obtained as weak limit of
  standard families (see Section~\ref{s_standardPairs} for details).  Such measures
  include, in particular, SRB measures as a special example.  Also, note that for the
  SRB measure it is certainly possible for the decay of correlations to be much faster
  also in the case $\nz>1$ since, the mass being already distributed in equilibrium, one
  may not have metastable states (see later discussion in the following subsection).
\end{rem}
\begin{rem} Note that several related results are present in the literature. First of all
  Tsujii in~\cite{Tsujii05} proves that a generic\footnote{ The exact meaning of
    \emph{generic} is a bit technical and we refer to~\cite{Tsujii05} for the
    details.}  family of type~\eqref{eq:map} has a finite number of SRB measures
  absolutely continuous with respect to Lebesgue. We believe that such a result applies to
  the present context, as Tsujii's genericity condition should reduce to our
  hypotheses~\ref{a_noCobo} and~\ref{a_fluctuationGap}, but this is not obvious to check.
  Next, exponential decay of correlations has been proven in the case of {\em mostly expanding}
  and {\em mostly contracting} center foliations. The mostly contracting case is studied
  in~\cite{BV00, deCastro02, deCastro04, Dimacontract}, but see~\cite{VY} for a recent
  overview on the subject; the mostly expanding case in~\cite{ABV00, ALP05, Gou06}.
  Unfortunately, to apply such results it is necessary to either show that the central
  Lyapunov exponent is negative or to estimate the Lebesgue measure of the points that
  have not expanded up to time $n$. On the one hand this is rather tricky to do,\footnote{
    We essentially prove that the Lyapunov exponents are negative, but this takes a good
    part of this paper.}  on the other hand the estimates on the rate of correlation decay
  provided by these papers are not quantitative. In particular, such results do not
  provide any information on how the rate of decay depends on $\ve$, hence they completely
  miss the issue of metastability. On the contrary, Kifer's papers~\cite{Kifer04,
    Kifer09} address very clearly the metastability issue, but, as already remarked, the
  results there do not allow to investigate the longer time scales, that is the SRB
  measures and their statistical properties.

  Finally let us remark that, contrary to most of the current literature (which discusses
  ``generic" systems), our conditions are explicit and, often, checkable by just studying
  few periodic orbits of the system.
\end{rem}
\subsection{Overview and structure of the paper} Let us now sketch the strategy of our
proof and outline the structure of this paper.  In Section~\ref{sec:examples} we present
some explicit class of examples to which our Main Theorem applies and an interesting case
to which it does not.  For some simple situations we compute how the SRB measure looks
like, we investigate metastability and compare it with the Wentzell-Freidlin case.

Our system is an example of fast-slow system (see Section~\ref{s_geometry}): averaging
theory (see Section~\ref{s_averaging}) implies that the slow variable $\theta$ undergoes a
diffusion around the dynamics of the averaged system, which is described by the ODE
$\dot\theta=\bar\omega(\theta)$. Assumption~\ref{a_noCobo} implies,\footnote{
  In~\cite{DeL1} it is shown that assumption~\ref{a_noCobo} is in fact generic in $\cC^2$.
  Observe moreover that the condition can be easily checked on periodic orbits.} by the
results of~\cite{DeL1}, that the diffusion is non-degenerate and indeed satisfies precise
Large Deviation Estimates and a Local Central Limit Theorem. In turn, Assumption
\ref{a_discreteZeros} implies that the averaged system has $\nz$ pairs of sinks and
sources: the set $\{\theta_{k,+}\}$ partitions $(\tmod~0)$ the torus $\bT$ in $\nz$
intervals $\fbas{k}=[\theta_{k,+},\theta_{k+1,+}] $, whose interiors are the basins of
attraction of $\theta_{k,-}$, \ie the averaged dynamics pushes every point in
$\intr\fbas{k}$ to $\theta_{k,-}$ exponentially fast.

In particular, if $\nz=1$, then the averaged dynamics pushes almost every initial
condition to the unique sink $\theta_-$.  Introducing a suitable notion of standard pairs
(see Section~\ref{s_standardPairs}), we can prove that the true dynamics closely follows
the averaged one with high probability (Section~\ref{s_averaged2real}). Thanks to this
fact we can establish a coupling argument (see Section~\ref{sec:coupling-def} for the
basic facts on coupling, Section~\ref{s_coupling} for the setup of the argument and
Section~\ref{sec:coup-proof} for proofs and the details) among sufficiently close standard
pairs: this implies exponential decay of correlations with a rate that is essentially
given by the time-scale of the averaged motion.

On the other hand, if $\nz>1$, then the averaged dynamics will push initial conditions
belonging to different basins to the corresponding sink; we thus need to rely on large
deviations to prove that standard pairs (\ie mass) are allowed to move from one basin to
another, although with very small probability.  Such events are called \emph{adiabatic
  transitions} and their typical time-scale is exponentially small in $\vei$.  If the
diffusion were purely stochastic and unbounded (\ie, in a Wentzell-Freidlin system
~\cite{WeFr}), then all transitions between different basins would be allowed.  On the
contrary, in our deterministic realization, some of the transitions might not be actually
possible (see Section~\ref{subsec:noerg} for an explicit example of this phenomenon and
Section~\ref{ss_furtherProperties} for an accurate description),
since the ``noise" is bounded, hence some sinks could act as traps for the real dynamics:
this constitutes an obstruction to ergodicity.  We need assumption~\ref{a_fluctuation} to
guarantee that no such obstructions occur.  In case~\ref{a_fluctuation} does not
  hold we need~\ref{a_fluctuationGap} to exclude borderline situations in which the number
  of different SRB measures could change in an arbitrarily small neighborhood of $F_\ve$.
  It is simple to check (see Lemma~\ref{l_genericFluctuationGap})
  that~\ref{a_fluctuationGap} is an open and dense property among maps
  enjoying~\ref{a_noCobo},\ref{a_discreteZeros} and~\ref{a_almostTrivial}. Finally, in
Section~\ref{sec:conclusion} we discuss the strengths and shortcomings of our approach and
we illustrate several open problems that must be addressed to push forward the research
program started by this paper.

\begin{nrem} We will henceforth fix $f$ and $\omega$ to satisfy all properties enumerated
  before; all values that we declare to be \emph{constant} below will depend on this
  choice.  We will often use $\Const,\const$ to designate some constants (again possibly
  depending on $f$ and $\omega$), whose actual value is irrelevant and can thus change
  form one instance to the next.
\end{nrem}
\section{Examples}\label{sec:examples}
To provide a better understanding of the results obtained in the present paper we first
discuss in some detail a few examples to which our theory applies. Then we briefly mention
an example that does not satisfy our conditions since the central direction seems to be,
unexpectedly, mostly expanding. Along the way, we take the occasion to carry our a precise
comparison with the case of small random perturbations of a dynamical system (the
so-called Wentzell--Freidlin systems). The conclusions of such a comparison are summarized in  Remarks~\ref{rem:diff0-fw},~\ref{rem:diff1-fw},~\ref{rem:no-FW}  and~\ref{rem:no-lyap-fw}.

Carrying out explicit computations in a specified example can be rather laborious.  We
will thus start by considering a particularly simple class of systems in which such
explicit computations can be done fairly easily: skew-products over the doubling map:
\begin{equation}\label{eq:skew}
  F_{\ve}(x,\theta)=(2x,\theta+\ve\bar\omega(\theta)+\ve \hat\omega(x))\mod 1.
\end{equation}
Also, to further simplify matters, we assume that
$\int_\bT\bar\omega(\theta) \deh\theta=0$.

Note that in this case the fast dynamics does not depend on $\theta$, hence making the
example very simple, although far from trivial. Thus the SRB measure for the fast dynamics
is always the Lebesgue measure $m$ for every $\theta\in \bT$.  Recall that our Main
Theorem discusses statistical properties of the process $\theta_n$ where
$(x_n,\theta_n)=F_\ve^n(x_0,\theta_0)$ and $x_0$ is distributed according to a measure
with smooth density w.r.t. Lebesgue.  The corresponding Wentzell--Freidlin scenario, which
should resemble our rescaled process $\theta_\ve(t)=\theta_{\ve^{-1}n}$, reads\footnote{
  It is possible to make this correspondence quantitatively precise for times of order
  $\ve^{-\alpha}$ for some $\alpha>0$.  We refrain from doing it to keep the length of the
  paper under control and we postpone it to further work.}
\begin{equation}\label{eq:f-w}
\deh w=\bar\omega(w)dt+\sqrt\ve\bVar(w) \deh B
\end{equation}
where $B$ is the standard Brownian motion. Note that, due to the fact that we have a skew
product and the simple form of $\omega$, in this particular case $\bVar(\theta)=\bVar$ is
independent of $\theta$.

Later, we will also mention less trivial examples without entering in too many details.

\subsection{Skew-products over the doubling map--one sink}\label{subsec:onesink}
Consistently with our notation we assume $m(\hat\omega)=0$. Also, to further simplify
matters, we assume that $\hat\omega(0)=3$ and $\|\bar\omega\|_{\infty}\leq 1$. Since the
Dirac measure $\delta_0$ is an invariant measure for the doubling map, we have that
$\hat\omega$ cannot be a coboundary, hence~\ref{a_noCobo} is satisfied. We
assume~\ref{a_discreteZeros}.  Since $\partial_\theta f=0$, we have
$\clyap(x,\theta)=\bar\omega'(\theta)=\bclyap(\theta)$.  We can then assume
that~\ref{a_almostTrivial} is satisfied.

Let us start with the case in which $\bar\omega$ has only one sink $\theta_-$, hence
$\bar\omega'(\theta_-)=-1$ by~\ref{a_almostTrivialp}.  Observe moreover that
assumption~\ref{a_fluctuation} is automatically verified since we have only one sink.
First of all let us understand the SRB measure $\mu_\ve$. Let $\hat H$ be a suitable
neighborhood of $\theta_-$ (see Section~\ref{subsec:averagedDynDescription} for more
details) and let $1-p=\mu_\ve(\hat H)$; setting
$B_\ve=[\theta_--\Const\sqrt\ve\ln\ve^{-1},\theta_-+\Const\sqrt\ve\ln\ve^{-1}]$,
$q=\mu_\ve(\hat H\backslash B_\ve)$.  Then Lemma~\ref{l_escapeFromAlcatraz} implies that
$p\leq \ve^\beta$.  Lemma~\ref{l_deepPurple} imply that at least $\frac 23$ of the mass of
a standard pair in $\hat H$ moves to $B_\ve$ is a time of order $\ve^{-1}\ln\ve^{-1}$.
While~\cite[Theorem~2.7]{DeL1} implies that at time $T\ve^{-1}$ the mass on any standard
pair $\ell$ in $B_\ve$ will be distributed according to a Gaussian centered in
$\theta_-+e^{-2T}(\theta_\ell-\theta_-)$ and with variance $\frac{1-e^{-4T}}2\ve\bVar^2$
apart from a mass $\ve^{2\beta}$, eventually making $\beta$ smaller. Iterating this
$\ln\ve^{-1}$ times we have that the mass is distributed according to a Gaussian centered
in $\theta_-$ and with variance $\ve\bVar^2$ apart from a mass $\ve^{\beta}$. This implies
that,
\[
1-p-q\geq  \frac 23 q-(1-p-q)(1-\ve^\beta)-\Const \ve^{\beta}
\]
hence $q\leq \Const\ve^\beta$.  It follows that $\mu_\ve$ consists of a Gaussian of
variance $\ve\bVar^2$ centered at $\theta_-$, a part from a mass of order
$\ve^\beta$.\footnote{ Of course, we mean this in the sense of~\cite[Theorem~2.7]{DeL1},
  on a scale smaller than $\ve$ the SRB could have some complicated fine structure. This
  issue is here left open.}

Now that we have a good understanding of the SRB measure $\mu_\ve$, we can address the
issue of the decay of correlations, in particular it is natural to wonder if the results
of our Main Theorem is optimal or not. Let us choose $A$ supported on a $\delta$
neighborhood of $\{\theta_+\}\times\bT$, where $\bar\omega(\theta_+)=0$ and
$\bar\omega'(\theta_+)>0$, and $B$ supported in a $\delta$ neighborhood of
$\{\theta_-\}\times \bT$, then we ask what is the maximal $c_\ve$ such that
\[
\begin{split}
  \left| \Leb(A\cdot B\circ F_\ve^n) - \Leb(A)\mu_\ve(B)\right|&\leq
  C_1\sup_\theta\|A(\cdot, \theta)\|_{\cC^2}\sup_x\|B(x,\cdot)\|_{\cC^2} \exp(-c_\ve n)\\
  &\leq \Const \exp(-c_\ve n).
\end{split}
\]
If we choose $\delta$ small enough, there will be a distance bigger than
$\frac 12|\theta_--\theta_+|$ between the support of $A$ and $B$. Moreover,
$\Leb(A)=\delta$, while $\mu_\ve(B)\geq 1-\Const\ve^\beta$. In addition, since $\theta$
can move only of steps of order $\ve$, we will have that $A\cdot B\circ F_\ve^n=0$ for all
$n\leq \Const\ve^{-1}$. Hence, $\frac 12\leq \Const \exp(-c_\ve \ve^{-1})$ which implies
that it must be $c_\ve\leq \Const\ve$.

We have thus seen that in the present case our Main Theorem is, at least, close to
optimal. Whether or not the $\ln\ve^{-1}$ is really there, or it is an artifact of our
method of proof, it remains to be seen.

To gain some more insight, let us compare the above situation with the Wentzell--Freidlin
system~\eqref{eq:f-w}. First of all, note that, taking advantage of the fact that we are
in dimension one, we can write the generator associated to the process as\footnote{ Note that, by hypotheses, $\rho_\ve(1)=Z\bVar^{-2}e^{2\ve^{-1}\int_0^1\frac{\bar\omega}{\bVar^2}}=Z\bVar^{-2}=\rho_\ve(0)$, hence $\rho_\ve$ is a smooth function on $\bT$.}
\begin{equation}\label{eq:generator}
\begin{split}
&L_\ve\vf=\bar\omega\vf'+\frac \ve 2\bVar^2\vf''=\frac{\ve}{2 \rho_\ve}(\bVar^2\rho_\ve\vf')',\\
&\rho_\ve(\theta)=Z\bVar^{-2}e^{2\ve^{-1}\int_0^\theta\frac{\bar\omega}{\bVar^2}}\;;\quad \int_{\bT}\rho_\ve=1.
\end{split}
\end{equation}
One can then easily check that $L_\ve$ is reversible with respect to the probability
measure $d\nu_\ve=\rho_\ve dx$. Thus $\nu_\ve$ is the invariant probability measure
of~\eqref{eq:f-w} and it is a direct computation to see that its Wasserstein distance from
$\mu_\ve$ is less that $\Const\ve^\beta$. Next, we need an estimate of the spectral gap of
$L_\ve$ in $L^2(\bT,\nu_\ve)$. As we were unable to locate it in the literature, we
provide it here (and since we are at it, we do it for the general case in which $\bVar$ is
not constant, but still strictly positive).

\begin{lem}
  If $\bar\omega,\bVar\in\cC^2(\bT,\bR)$, $\bar\omega$ has only two non degenerate zeroes
  $\theta_-,\theta_+$ and $\inf\bVar>0$, then there exists $c_0>0$ such that, for each
  $\ve$ small enough, $L_\ve$ has a spectral gap larger than $c_0$.
\end{lem}
\begin{proof}
We follow the logic used in~\cite[Theorem~1.2, Appendix A.19]{Villani09}. Setting $V_\ve(\theta)=- 2\int_0^\theta\frac{\bar\omega}{\ve\bVar^2}+2\ln\bVar -\ln Z$, we have $\rho_\ve=e^{-V_\ve}$.  Next, let
\[
W_\ve=\bVar^2(V_\ve')^2/2-(\bVar^2V_\ve')',
\]
then for each $\vf\in\cC^2$ we have (see~\cite[Proof of Theorem~6.2.21]{DS89})\footnote{ For the reader convenience, here is how to argue: compute using
\[
0\leq \int_{\bT}\left[\bVar\left(e^{-V/2}\vf\right)'\right]^2 \textrm{ and } \int_{\bT} \bVar^2V'\vf'\vf e^{-V}=\frac 12\int_{\bT}\left[ (\bVar V')^2-(\bVar^2V')'\right]\vf^2 e^{-V}.
\]
}
\begin{equation}\label{eq:strook-d}
\int_\bT W_\ve\vf^2\rho_\ve\leq 2\int_\bT\bVar^2 (\vf')^2\rho_\ve.
\end{equation}
Note that the right hand side of~\eqref{eq:strook-d} is nothing else than the Dirichlet form associated to $L_\ve$. In addition, $W_\ve=\frac{2}{\ve^2\bVar^2}\bar\omega^2-4\frac{\bar\omega\bVar'}{\ve\bVar}+\frac{2}{\ve}\bar\omega'+2(\bVar')^2-2(\bVar'\bVar)'$ is always positive apart from a neighborhood of $\theta_-$. Indeed, let $A_\ve=[\theta_--a\sqrt\ve,\theta_--a\sqrt\ve]$, then, if $a$ is chosen large enough, $\inf_{A_\ve^c}W_\ve\geq \ve^{-1}$. Moreover, if $\theta,\theta'\in A_\ve$, then
\begin{equation}\label{eq:rho-constant}
\frac{\rho_\ve(\theta)}{\rho_\ve(\theta')}\leq e^{\const \ve^{-1}|\theta-\theta'|^2}\leq e^{\const a^2}.
\end{equation}
Let $K$ to be chosen later (large enough) and choose  $a$ so that $\int_{A_\ve^c}\rho_\ve\leq K^{-1}$. Note that there exist a constant $C_a>0$ such that $C_a^{-1}\ve^{-\frac 12}\leq\rho_\ve(\theta)\leq C_a\ve^{-\frac 12}$, for all $\theta\in A_\ve$.

The last needed ingredient is the standard Poincar\`e inequality in $A_\ve$: there exists $b>0$ such that, for all $\ve$ small enough,\footnote{ Again, for the reader convenience, here is how to argue: first of all note that~\eqref{eq:rho-constant} implies that, on $A_\ve$ the ratio between the sup and inf to $\rho_\ve$ is bounded by $e^{\const a^2}$, and remember that $\nu_\ve(A_\ve)\geq 1/2$. Then
\[
\int_{A_\ve}\vf^2\rho_\ve=\frac{1}{2\nu_\ve(A_\ve)}\int_{A_\ve^2}[\vf(x)-\vf(y)]^2\rho_\ve(x)\rho_\ve(y) dxdy+\frac{1}{\nu_\ve(A_\ve)}\left(\int_{A_\ve}\vf\rho_\ve\right)^2.
\]
While $[\vf(x)-\vf(y)]^2\leq (\int_{A_\ve}|\vf'|)^2\leq |A_\ve|\int_{A_\ve}(\vf')^2\leq \Const\ve\int_{\bT}(\vf')^2\rho_\ve$.
}
\[
\int_{A_\ve}\vf^2\rho_\ve\leq b\ve\int_{\bT}(\vf')^2\rho_\ve+2\left(\int_{A_\ve}\vf\rho_\ve\right)^2.
\]
Thus, for each $\vf\in\cC^2$, we have, for $K$ large enough,
\[
\begin{split}
\int_{\bT}\vf^{2}\rho_\ve&\leq \Const \ve\int_{\bT}W_\ve\vf^2e^{-V_\ve}+\frac K 4 \int_{A_\ve}\vf^2\rho_\ve\\
&\leq \Const\ve\int_{\bT}\bVar^2(\vf')^2\rho_\ve+\frac K2 \left(\int_{A_\ve}\vf\rho_\ve\right)^2.
\end{split}
\]
To conclude, assume that $\int_{\bT}\vf\rho_\ve=0$, hence
$\int_{A_\ve}\vf\rho_\ve=-\int_{A_\ve^c}\vf\rho_\ve$. Thus
\[
 \left(\int_{A_\ve}\vf\rho_\ve\right)^2\leq\nu_\ve(A_\ve^c)\int_{A_\ve^c}\vf^2\rho_\ve\leq K^{-1}\int_\bT \vf^2\rho_\ve.
\]
We are thus ready to compute the spectral gap: let $\vf\in\cC^2$ such that $\nu_\ve(\vf)=0$, then
\[
\int_\bT\vf(-L_\ve\vf)\rho_\ve=\frac{\ve}2\int_\bT\bVar^2(\vf')\rho_\ve\geq \Const\int_\bT\vf^2\rho_\ve.
\]
\end{proof}
 The above Lemma implies that,
 \[
 \left|\nu_\ve(A(w(0))B(w(t))-\nu_\ve(A)\nu_\ve(B)\right|\leq \Const \|A\|_{L^2(\nu_\ve)}\|B\|_{L^2(\nu_\ve)}e^{-c_0t}.
 \]
 If the Wentzell--Freidlin process~\eqref{eq:f-w} is a good predictor of what happens for
 the process $\theta_{\ve^{-1}t}$ (as we, in the case, conjecture), then the factor
 $\ln\ve^{-1}$, in our estimate for the rate decay of correlations~\eqref{e_lowerBoundRate}, should be absent if one starts from the SRB measure rather then
 the Lebesgue measure.  Note however that if we start from a measure non absolutely
 continuous with respect to $\nu_\ve$ (e.g. a $\delta$ at some $\theta_0$) or with an
 exponentially large Radon-Nikodyn derivative (e.g. Lebesgue), then it will take a time at
 least $\Const \ln\vei$ before the measure becomes uniformly absolutely continuous with
 respect to $\nu_\ve$. Hence, the question if such a factor is present or not when
 starting from a more general measure remains unclear, see also Remark~\ref{rem:lebtd}.
\begin{rem}\label{rem:diff0-fw} We have thus seen that, in this simple case, our
   deterministic process and the Wentzell--Freidlin process are remarkably similar. In fact,
   we conjectured that they have the same exact statistical properties.
\end{rem}
\subsection{Skew-products over the doubling map--two sinks (non ergodic case)}\label{subsec:noerg}
Next, let us consider the case in which $n_Z=2$.  To start with, we assume that
$|\bar\omega|$ reaches the value one in each of the intervals
$(\theta_{1,-},\theta_{1,+})$, $(\theta_{1,+},\theta_{2,-})$,
$(\theta_{2,-},\theta_{2,+})$, $(\theta_{2,+},\theta_{1,-})$. Also we assume that
$|\hat\omega|\leq \frac 12$. Note that now assumption~\ref{a_fluctuation} is {\bf not}
satisfied while it is satisfied hypotheses~\ref{a_fluctuationGap} where
alternative~\ref{p_oneway} holds for any interval $J$.  In the language of Subsection
\ref{ss_furtherProperties} this implies that there are two trapping sets
$(\theta_{2,+},\theta_{1,+})\supset\trap_1\ni\theta_{1,-}$ and
$(\theta_{1,+},\theta_{2,+})\supset\trap_2\ni\theta_{2,-}$.\footnote{ The choice of
  $\aeps$ is rather arbitrary, it suffices that it is small enough so that
  $\trap_{i}\neq \emptyset$. In the present case $\aeps=1/4$ will do.} Hence, the dynamics
has two attractors with basins that contain the respective trapping sets and there are two
SRB measures $\mu_{i,\ve}$ supported in $\{\trap_i\}$, respectively.  In
  fact, the SRB measure $\mu_{1,\ve}$ charges any $\ve$-ball in a fixed (\ie independent
  of $\ve$) neighborhood\footnote{ Again, in the language of
    Section~\ref{ss_furtherProperties}, one can take the neighborhood to be
    $\bigcap_{\aeps>0}\freach{\sink 1}$} of $\sink 1$.  Yet, by
arguments similar to the ones used above, such measures are $\ve^\beta$-close to two
Gaussians, with variance of order $\ve$ and centered at $\{\theta_{1,-},\theta_{2,-}\}$,
respectively.  That is, the system looks superficially like it has two attractors
contained in a $\sqrt \ve$ neighborhood of $\{\theta_{1,-},\theta_{2,-}\}$.
\begin{rem} \label{rem:diff1-fw} Note that in this case we have a drastic difference with
  the Wentzell--Freidlin process which, on the contrary, is ergodic. For the
  Wentzell--Freidlin process the measures $\mu_{i,\ve}$ are essentially the {\em
    metastable states}. On the contrary, in the deterministic case they are stable (\ie
  invariant). As already remarked, this is due to the substantial difference in the large
  deviations rate function of the two processes.
\end{rem}

\subsection{Skew-products over the doubling map--two sinks (ergodic case)}\label{subsec:twoerg}
In this case we choose $\bar\omega$ as in our previous example , but with
$-\frac 13<\hat\omega\leq 3$ with $\hat \omega(0)=3$.  Next, remember that the set of
invariant probability measures is a closed convex set and so
$\Omega(\theta)=\{\mu(\omega(\cdot,\theta))\;|\; \mu \textrm{ invariant for }2x\mod 1\}$
is a closed interval for each $\theta\in\bT$. Thus
$\Omega(\theta)\supset [\bar\omega(\theta),\bar\omega(\theta)+3]\supset [1,2]$.
Accordingly, the path $h(t)=\frac 32t$ is admissible and visits all the circle, hence also
assumption~\ref{a_fluctuation} is satisfied.  We are thus in a setting to which our
results apply, hence the map has a unique SRB measure that can be represented as a
standard family.

As in the previous case the invariant measure will essentially consist of two Gaussians of
variance of order $\ve$ centered at $\{\theta_{1,-},\theta_{2,-}\}$, respectively. Yet, to
really understand how the SRB looks like we must know the mass of the two Gaussians, let
us call them $\{p_i\}$, respectively. Of course, $1-p_1-p_2\leq\Const\ve^\beta$.

In general, to figure out the latter we can use~\cite[Theorem~2.2]{DeL1} to compute the
probability for a trajectory to go from a neighborhood of $\theta_{1,-}$ to a neighborhood
of $\theta_{2,-}$ and viceversa.  Since the ratio of $p_1$ and $p_2$ depends on the
probability of going from one sink to the other.  This entails some work and more
information on $\bar\omega$.

In order to keep things as simple as possible we choose
$\bar\omega(\theta)=\sin(4\pi\theta)$. Note that if we set
$R(x,\theta)=(x,\theta+\frac 12\mod 1)$, then $F_\ve\circ R=R\circ F_\ve$. But such a
symmetry implies that, for each continuous function $\vf$,
\[
R_*\mu_\ve(\vf)=\mu_\ve(\vf\circ R\circ F_\ve)=\mu_\ve(\vf\circ F\circ
R)=R_*\mu_\ve(\vf\circ F_\ve).
\]
Thus $R_*\mu_\ve$ is invariant and, as $\mu_\ve$, can be written in terms of a standard
family. Since such a measure is unique, it must be $\mu_\ve=R_*\mu_\ve$, that is
$p_1=p_2=\frac 12+\cO(\ve^\beta)$.

Now that we have identified the SRB measure, we can discuss the issue of the decay of
correlations. We choose $A$ to be supported in a $\delta$ neighborhood of $\theta_{1,-}$,
for $\delta$ small enough, such that $\Leb(A)=1$ and $B$ to be supported on a $\delta$
neighborhood of $\theta_{2,-}$ such that $\mu_\ve(B)=1$. We then ask what is the maximal
$c_\ve$ such that
\[
\begin{split}
\left| \Leb(A\cdot B\circ F_\ve^n) - \Leb(A)\mu_\ve(B)\right|&\leq
  C_1\sup_\theta\|A(\cdot, \theta)\|_{\cC^2}\sup_x\|B(x,\cdot)\|_{\cC^2} \exp(-c_\ve n)\\
  &\leq \Const \exp(-c_\ve n).
\end{split}
\]
We know that $\Leb(A)\mu_\ve(B)=1$. It remains to compute
$\Leb(A\cdot B\circ F_\ve^n)$. Note that at $1/8$ and $7/8$ we have $\bar\omega=1$, on the
other hand, by hypotheses $\hat\omega\geq -1/2$, thus
$\bar\omega+\hat\omega\geq \frac 12$. That is, around $1/8$ and $7/8$ the motion can take
place only from left to right.

Let $I_{i,-}=\{\theta\in\bT\;:\; |\theta-\theta_{i,-}|\leq \delta\}$.  Consider a process
starting from a standard pair $\ell$ with $\theta_\ell\in I_{1,-}$. Our aim is to compute
the probability of the event $Q_T=\{\theta_\ve(T)\in I_{2,-}\}$.  Note that, by choosing
$T_0$ large enough, we have $|\bar\theta (T_0)-\theta_{1,-}|\leq \frac 14\delta$.  Thus if
$\gamma(T_0)\not\in I_-$, then $\sup_{t\in [0,T_0]}|\gamma(t)-\bar\theta(t)|>
\delta/2$. Then, using Theorem~\ref{t_largeDevz}, we have
\[
\bP_{\ve}(\gamma(T_0)\in I_{2,-})\le\bP_\ve(\{\gamma(T_0)\not\in I_{1,-}\})\leq
\mu_\ell(Q(\delta/2,1))\leq e^{-\const\ve^{-1}}.
\]
Since the distribution at time $T_0$ is still made of standard pairs we can apply the same argument and obtain that
\[
\bP_\ve(\{\gamma(T_0 n)\nin I_{1,-}\})\leq ne^{-\const\ve^{-1}}
\]
It follows that, for each $n\leq e^{\const\ve^{-1}}$ we have
\[
\Leb(A\cdot B\circ F_\ve^n)\leq \|A\|_\infty\|B\|_\infty\bP_\ve(\{\gamma(\ve n)\not\in I_{1,-}\})\leq e^{-\const\ve^{-1}}.
\]
It follows that
\[
c_\ve\leq e^{-\const\ve^{-1}}.
\]
Thus the estimate in our Main Theorem has the right dependence on $\ve$, even though, of
course, the value of the constants are very hard to determine.

This means that, if we start with an initial distribution with mass, say, $1/4$ in a
neighborhood of $\theta_{1,-}$ and $3/4$ in a neighborhood of $\theta_{2,-}$, then for
an exponentially long time we will see a situation very similar to what we have seen in
Section~\ref{subsec:noerg}: it looks like the system is distributed according to an
invariant measure. Yet, if we look at a longer exponential time, we will see the ratio of
the masses of the two Gaussian change till it reaches the values approximately $1/2, 1/2$
which characterize the true invariant measure. Hence the metastability phenomena we have
announced.

\begin{rem}\label{rem:no-FW} We have seen that, in this case, we have metastable states as
  in the Wentzell--Freidlin case. Only, the attentive reader has certainly noticed that, in
  absence of a symmetry, there is no reason for the masses in the two sinks to be the
  same. Also we have seen by our large deviation computations that the probabilities to
  transit from one sink to the other are always exponentially small in $\ve^{-1}$. It is
  thus to be expected that in a non symmetric (generic) case one of the two masses will be
  exponentially smaller than the other. Thus the SRB measure will look very much like a
  single Gaussian centered at the ``winning" sink. Of course, the same occurs for
  Wentzell--Freidlin, yet who is the winning sink is decided by the large deviation
  functionals which are very different. So, again, we should expect cases in which the
  invariant measures of the two processes looks completely different, being essentially
  centered at different sinks.
\end{rem}
\subsection{Not a skew-product} To conclude our discussion, here is another, slightly less simple, example:
  \[
  F_\ve(x,\theta)=(\ell x+ a\theta, \theta +\ve b \cos 2\pi x-\frac{\ve}{2\pi}\sin 2\pi\theta) \mod 1.
  \]
  Note that $\rho_\theta=1$ and $\bar \omega(\theta)= -\frac{1}{2\pi} \cos 2\pi \theta$.
  Note that the fixed points of $\ell x+a\theta$ are $x=\frac{p-a\theta}{\ell-1}$,
  $p\in\bN$, thus $\cos 2\pi \frac{p-a\theta}{\ell-1}-\frac{\ve}{2\pi}\sin 2\pi\theta$
  cannot be zero for all $p$ if $\ell>2$, hence condition (A0) holds. A direct computation
  shows that such maps satisfy (A1-3) as well, provided $\ell, b$ are chosen large and $a$
  small enough.  The reader can check that the analysis carried out in
  Section~\ref{subsec:onesink} can be replicated almost verbatim for the present
  case. Also it is fairly easy to produce examples with several sinks (like in
  Sections~\ref{subsec:twoerg},~\ref{subsec:noerg}). Thus all above consideration apply
  here as well.

  Of course, it is also possible to consider examples where the fast
  dynamics has a $\theta$ dependent invariant measure and to which our results apply. Yet,
  the latter case can hold in storage interesting surprises, as we are going to see.

\subsection{An interesting non-example}\label{sec:no-ex}
Let $\ell\in\bN$, $\ell > 1$, and consider the family
\begin{align*}
  F_\ve(x,\theta)=(\ell x + \sin(2\pi\theta)\left[\alpha\sin(2\pi
    x)+\beta\sin(2\ell\pi x) \right], \theta + \ve\cos(2\pi x))\mod \bZ^2.
\end{align*}
In the above example, assuming $\ell-2\pi(\alpha+\ell\beta) > 2$,
  Assumption~\ref{a_noCobo},~\ref{a_discreteZeros} are satisfied; moreover if $\ell$ is
  odd, we have that $x = 0$ and $x = 1/2$ are fixed points of $f_\theta$ for any
  $\theta\in\bT$; since $\omega(0) = 1$ and $\omega(1/2) = -1$, we have
  $\Omega(\theta)\supset[-1,1]$, hence Assumption~\ref{a_completeness} is satisfied (see
  Remark~\ref{r_completeness}) and in particular~\ref{a_fluctuation} holds.  However,
Assumption~\ref{a_almostTrivial} is not obvious to verify.  This whole subsection is
devoted to the discussion of this issue. In doing so we will uncover the possibility of a
most surprising feature: an ``attractor" with all Lyapunov exponents almost surely
positive (with respect to the SRB measure).\footnote{ Of course, technically speaking,
  there is no attractor as condition~\ref{a_completeness} guarantees that the dynamics
  will visit an $\ve$-dense set in configuration space.  Yet, for small $\ve$ and each
  $\beta\in (0,1/2)$ a portion $1-e^{-\const \ve^{-1+2\beta}}$ of the mass is concentrated
  in a $\cO(\ve^\beta)$-neighborhood of $\theta = 0$.  So the situation differs indeed
  very little from an attractor. In passing, this example shows that a purely topological
  description of the dynamics can fail miserably in capturing the relevant properties of
  the motion.} To actually prove this would take some non trivial work; here we content
ourselves by showing that for $\alpha,\beta>0$ the average dynamics has a sink and yet the
true dynamics near such a sink has center vectors that are mostly expanding.

Observe that if $\theta=0$ or $\theta=1/2$ (so that $\sin(2\pi\theta)=0$), then
$f_\theta(x)=\ell x$, thus $\rho_\theta=1$, and $\bar\omega(\theta)=0$.  Let us
now compute $\bar\omega'(\theta)$ at $\theta=0$:
\begin{align*}
  \bar\omega'(\theta)&=\frac{\deh}{\deh\theta}\int_{\bT}\omega(x)\rho_\theta(x)\deh x=
 \sum_{k=1}^\infty\int_{\bT}(\omega\circ f_\theta^k(x))'\frac{\partial_\theta
    f(x,\theta)}{f'_\theta(x)}\rho_\theta(x)\deh x\\
   &=-{(2\pi)^2}\sum_{k=1}^\infty\ell^{k-1}\int_\bT\sin(2\ell^k\pi x)[\alpha\sin(2\pi
  x)+\beta\sin(2\ell\pi x)]\\&=%
  -2\pi^2\beta,
\end{align*}
where we have used the perturbative results detailed in~\cite[Appendix A.3]{DeL1}.  Thus
if $\beta>0$, then $\theta=0$ is a sink.

On the other hand, as will be shown in~\eqref{eq:central-lyap}, the expansion of center
vectors at time $n$ is
\begin{equation}\label{eq:mun-formula}
  \ln\mu_n(p)=\ve\sum_{k=0}^{n-1}\left[\partial_\theta\omega(p_k)+\partial_x\omega(p_k)\stable_{n-k}(p_k)\right]+\cO(\ve^2n),
\end{equation}
where $\stable_n$ is defined in~\eqref{e_defineSlopes}; we thus have the formula:
\begin{align*}
  \stable_n(p_0)&= \stable_n(x_0,\theta_0)=-\sum_{k=0}^{n-1}\frac{\partial_\theta
    f(x_k,\theta_k)}{\prod_{j=0}^{k}\partial_x f(x_j,\theta_j)}+\cO(\ve)\\&=
  -\sum_{k=0}^{n-1}\frac{\partial_\theta
    f(f_{\theta_0}^k(x_0),\theta_0)}{(f_{\theta_0}^{k+1})'(x_0)}+\cO(\ve\log\vei)
\end{align*}
where we have used~\cite[Lemma~4.1]{DeL1}.  Substituting this is~\eqref{eq:mun-formula}
yields
\[
\begin{split}
  \ln\mu_n(p)&=\ve\sum_{k=0}^{n-1}\left[\partial_\theta\omega(p_k)+\partial_x\omega(p_k)\slim(p_k)\right]+\cO(n\ve^2\ln\vei+\ve)\\
  &=\ve\sum_{k=0}^{n-1}\clyap(p_k)+\cO(n\ve^2\ln\vei+\ve).
\end{split}
\]
The above formula illustrates the announced relation between the function $\clyap$ and the central Lyapunov exponent.
Also we have seen that
\[
\clyap(p_0)=-\sum_{k=0}^{\infty}\frac{\partial_\theta
  f(f_{\theta_0}^k(x_0),\theta_0)}{(f_{\theta_0}^{k+1})'(x_0)}+\cO(\ve\log\vei).
\]
The average of the logarithm of the expansion of center vectors, at $\theta_0=0$, is
\begin{align*}
  \bclyap(\theta_0)=&-\sum_{k=0}^{\infty}\int_\bT\partial_x\omega(x,\theta_0)\frac{\partial_\theta
    f(f_{\theta_0}^k(x),\theta_0)}{(f_{\theta_0}^{k+1})'(x)} \deh x+\cO(\ve\log\vei)=\\&=
  2\pi\sum_{k=0}^{\infty}\int_\bT\sin(2\pi x)\frac{\alpha\sin(2\ell^k\pi
    x)+\beta\sin(2\ell^{k+1}\pi x)}{\ell^{k+1}}\deh x+\cO(\ve\log\vei)=\\&=
  2\pi^2\frac{\alpha}\ell+\cO(\ve\log\vei)
\end{align*}
Consequently, if $\alpha<0$, then assumption~\ref{a_almostTrivial} is satisfied and our
Main Theorem applies. On the contrary, if $\alpha>0$, then
assumption~\ref{a_almostTrivial} is violated and, as announced, we have a map that we
expect to have positive central Lyapunov exponent.

\begin{rem}\label{rem:no-lyap-fw}
  We have just seen another drastic difference between the Wentzell--Freidlin process and
  the deterministic process: in the Wentzell--Freidlin process the Lyapunov exponent
  associated to the slow variable is always negative\footnote{ This follows from a direct
    computation.} while we have seen that for the deterministic process it can be
  positive.  This depends on the fact that the stochastic process does not reflect
  completely the interplay between the slow and the fast variable which can be much more
  subtle in the deterministic case.
\end{rem}
\section{Geometry}\label{s_geometry}
Throughout this article, $\pi:\bT^2\to\bT$ denotes the projection on the
$x$-coordinate. We denote a point in $\bT^2$ by $p=(x,\theta)$; we use the notation
$p_n=(x_n,\theta_n)=F_\ve^n\,p$.  Our first task is to find invariant cones for the
dynamics: for $\Kconeu,\Kconec>0$ to be specified later, let us define the \emph{unstable cone}
and the \emph{center cone} as, respectively:
\begin{align}\label{e_definitionCones}
  \coneu&=\{(\xi,\eta)\in\bR^2\;:\; |\eta|\leq\ve \Kconeu|\xi|\}&
  \conec&=\{(\xi,\eta)\in\bR^2\;:\; |\xi|\leq \Kconec|\eta|\}.
\end{align}
We claim that there exist $\Kconeu,\Kconec$ such that, if $\ve$ is small enough, $\deh
F_{\ve}\coneu\subset\coneu$ and $\deh F_{\ve}\invr\conec\subset\conec$.  In fact, let us
compute the differential of $F_\ve$:
\begin{equation}\label{e_formulaDF}
  \deh F_\ve=\begin{pmatrix} \partial_{x}f &\partial_\theta f\\ \ve\partial_x\omega &1+\ve\partial_\theta\omega\end{pmatrix};
\end{equation}
consequently, if we consider the vector $(1,\ve u)$
\begin{align}
  \deh_pF_\ve(1,\ve u) &= (\partial_{x}f(p) +\ve u\partial_\theta f(p),\ve\partial_x\omega(p) +\ve u+\ve^2 u\partial_\theta\omega(p))\notag\\
  &=\partial_{x}f(p)\left(1 +\ve \frac{\partial_\theta
      f(p)}{\partial_{x}f(p)}u\right)\cdot (1,\ve \Xi_p(u))\label{e_formulaExpansionX}
\end{align}
where
\begin{equation}\label{e_evolutionXi}
  \Xi_p(u)=\frac{\partial_x\omega(p) + (1+\ve\partial_\theta\omega(p))u}{\partial_{x}f(p) +\ve \partial_\theta f(p)u},
\end{equation}
from which we obtain our claim, choosing for instance
\begin{align}\label{e_definitionKconeuc}
  \Kconeu&=2\|\partial_x\omega\|_\infty& \textrm{ and }& &
  \Kconec&=2\|\partial_\theta f\|_\infty.
\end{align}
From the above computations it is easy to see that $F_\ve$ is a partially hyperbolic map
with expanding direction in $\coneu$ and central direction in $\conec$.

It follows that, for any $p\in\bT^2$ and $n\in\bN$, we can define the real quantities
$\mu_n$, $\nu_n$, $u_n$ and $\stable_n$ as follows:
\begin{align}\label{e_defineSlopes}
  \deh_pF_{\ve}^n(1,0)&=\nu_{n}(1,\ve u_n)&\deh_p F_{\ve}^n(\stable_n,1)&=\mu_n(0,1)
\end{align}
with $|u_n|\leq \Kconeu$ and $|\stable_n|\leq \Kconec$. For each $n$ the slope field $\stable_n$ is
smooth, therefore integrable; given any (small) $h>0$ and $p_*=(x_*,\theta_*)\in\bT^2$, define
$\cW_n^\nt(p_*,h)$ the \emph{local $n$-step center manifold of size $h$} as the
connected component containing $p_*$ of the intersection with the strip
$\{\theta\in B(\theta_*,h)\}$ of the integral curve of $(\stable_n,1)$ passing through $p_*$.
Observe that, by definition, any vector tangent to a local $n$-step center manifold
belongs to the center cone.

Moreover, notice that, by definition,
\begin{align*}
\deh_p F_\ve(\stable_n(p),1) = \mu_n(p)/\mu_{n-1}(F_\ve p)(\stable_{n-1}(F_\ve p),1);
\end{align*}
a direct application of~\eqref{e_formulaDF} yields
\begin{equation}\label{eq:central-lyap}
  \frac{\mu_n(p)}{\mu_{n-1}(F_\ve p)}=1+\ve\left[\partial_\theta\omega(p)
    + \partial_x\omega(p)\stable_n(p)\right]
\end{equation}
Observe that the above expression implies
\begin{equation}\label{e_trivialBound}
  \log{\mu_m(p)}-\log{\mu_{m-n}(F_\ve^np)} \le\expb n\ve,
 \text{ where } \expb = \|\partial_\theta\omega\|+\Kconec\|\partial_x\omega\|.
\end{equation}
Moreover,
\begin{align}\label{eq:central-lyap-s}
  s_n(p)&=\frac{(1+\ve\partial_\theta\omega(p))s_{n-1}(F_\ve(p))-\partial_\theta f(p)}{\partial_x f(p)-\ve\partial_x\omega(p) s_{n-1}(F_\ve(p))}=:\Xi_p^-(s_{n-1}(F_\ve(p)).
 \end{align}
Note that
\[
\frac{\deh}{\deh s}\Xi^-_p(s)=\frac{(1+\ve\partial_\theta\omega(p))\partial_x
  f(p)-\ve\partial_x\omega(p)\partial_\theta
   f(p)}{\left[\partial_xf(p)-\ve\partial_x\omega(p) s\right]^2}.
\]
Accordingly, for each $|s|\leq \Kconec$ and $\ve$ small enough, we have that there exists
$\sigma_c\in (0,1)$ such that
\[
\left|\frac{\deh}{\deh s}\Xi^-_p(s)\right|\leq \sigma_c.
\]
This implies that $s_n$ is a converging sequence: let $\slim$ be its limit. Then, for all
$p\in\bT^2$, $|s_n(p)-\slim(p)|\leq \Const \sigma_c^n$.  We have thus a formula for the
center slope.  Yet, it is well known that, in general, $\slim$ is not a very regular
function of the point.

This could create trouble while using our assumption~\ref{a_almostTrivial} since typically
we will need to apply to it formulae that require some regularity.  To overcome this
problem we define a regularized function $\clyapReg$ that approximates $\clyap$:
\begin{align}\label{eq:regularizedpsi}
  \clyapReg(p)&=\partial_\theta\omega(p)+\partial_x \omega(p) s_{\bar n}(p)&
  \bclyapReg(\theta)&=\int_{\bT}\clyapReg(x,\theta)\rho_{\theta}(x)\deh x
\end{align}
where $\bar n$ is such that for any $p\in\bT^2$
\begin{align*}
  \|\clyapReg(p)-\clyap(p)\| < \closeness < 1/8,
\end{align*}
for some $\closeness$ small to be specified in due course.
\begin{rem}\label{rem:newpsi}
  Note that $\bar n\sim|\log\closeness|$ is independent of $\ve$; hence, due to the
  uniformity in $p$ of all estimates involved, we have uniform bounds on the norms of
  $\clyapReg$, \ie: $\|\clyapReg\|_{\cC^0} < \expb$ and
  $\|\clyapReg\|_{\cC^1} < \exp(\Const \bar n)$.  Under assumption~\ref{a_almostTrivialp}
  we have
  \[
  \max_{k\in\{1,\cdots,\nz\}}\bclyapReg(\theta_{k,-})\in [-9/8,-7/8].
  \]
\end{rem}
Define the function $\zeta_n$ as:
\begin{equation} \label{e_definitionZ}%
\zeta_n=\ve\sum_{k=0}^{n-1}\clyapReg\circ F_\ve^k.
\end{equation}
\begin{lem}[Distortion]\label{l_distortion}
  For any $T>0$ there exists $C_T$ so that, for any $p\in\bT^2$ and $h>0$ sufficiently
  small, let $N = \pint{T\vei}$
  \begin{align*}
    \sup_{q\in\cW_N^\nt(p,h)}\mu_N(q)\leq%
    \expo{\zeta_N(p)+C_T h+2T\varrho+2\bar n\expb\ve}.
  \end{align*}
\end{lem}
\begin{proof}
  Let us introduce the convenient function
  $\clyapn{n}(p) = \partial_\theta\omega(p)+\partial_x\omega(p)s_n(p)$; then
  by~\eqref{eq:central-lyap} we can write
  $\mu_N(p) \le \exp\left(\ve\sum_{n = 0}^{N-1}\clyapn{n} \right))$.  On the other hand,
  by construction and the triangle inequality we have
  $\|\clyapReg-\clyapn{n}\| < 2\closeness$ if $n\ge\bar n$ (otherwise the trivial bound
  $\|\clyapReg-\clyapn{n}\| <2\expb$ holds); we conclude that
  $\mu_N(p)\leq\expo{\zeta_N(p)+2T\varrho+2\bar n\expb\ve}$; next, we need to compute the
  derivative of $\zeta_N$ along the $N$-step central direction.  By~\eqref{e_definitionZ}
  it follows
  \[
  \deh\zeta_N (\stable_N,1)=\ve\sum_{k=0}^{N-1}\langle \nabla \clyapReg\circ F_\ve^k,\deh
  F_\ve^k (\stable_N,1)\rangle;
  \]
  hence,~\eqref{e_defineSlopes} and~\eqref{e_trivialBound} imply that $\deh F_\ve^k
  (\stable_N,1)=e^{\cO(k \expb\ve)}(\stable_{N-k},1)$.  Thus there exists
  $b_T\sim\expo{\const T}>0$ such that
  \[
  |\deh\zeta_N (\stable_N,1)|\leq b_T\ve\sum_{k=0}^{N-1}e^{k\expb\ve}\leq \Const b_T\expb^{-1}.
  \]
  Accordingly,
  \begin{align*}
    \sup_{q\in\cW_N^\nt(p,h)}\mu_N(q)&\leq \expo{\zeta_N(p)+\Const b_T\expb^{-1}h+2T\varrho+2\bar n\expb\ve}.\qedhere
  \end{align*}
\end{proof}
\section{Standard pairs, families and couplings}\label{s_standardPairs}
\subsection{Definitions and basic facts}
In this section we recap the standard families formalism, first introduced by Dolgopyat
(see e.g.~\cite{DimaPH,DimaSRB,DimaAveraging}) to study statistical properties of
partially hyperbolic dynamical systems\footnote{ See also~\cite[Section 3.2]{DeL1} for a
  similar, but more general, account of the framework in this context.}.
\begin{rem}
  The educated reader will certainly notice that our regularity assumptions are stronger
  than the ones which are usually required to apply the coupling argument (see
  e.g.~\cite{Chernov}).  The stronger regularity conditions are in fact needed in order to
  obtain the refined statistical properties (\ie,\ the Local Central Limit Theorem) that
  we use to set up the coupling argument in an efficient manner.  Consequently, they are
  crucial to obtain the near-optimal bounds on the rate of decay of correlations that we seek.
\end{rem}
\subsubsection{Standard pairs}\label{subsec:Standardpairs}
Let us fix a small $\delta>0$, and $\dtt0, \dtt1>0$ large to be specified later; for
any $\spc1>0$ let us define the set of functions
\begin{align*}
  \Sigma_{\spc1}=\{G\in \cC^3([a,b], \bT)\st&
  a,b\in\bT, b-a\in[\delta/2,\delta],\\&
  \|G'\|\leq \ve \spc1,\, \|G''\|\leq \ve \dtt0 \spc1,\,\|G'''\|\leq \ve \dtt1 \spc1\}.
\end{align*}
Let us associate to any $G\in \Sigma_{\spc1}$ the map $\bG(x)=(x,G(x))$; the graph of any
such $G$ (\ie the image of $\bG$) will be called a \emph{proper $\spc1$-standard curve}.
With a little abuse of terminology, we refer to the quantity $b-a$ as the \emph{length of
  the curve}.  If we do not require the lower bound for the length of the curve, we obtain
the definition of a \emph{short $\spc1$-standard curve}; for ease of exposition we adopt
the convention that all standard curves are assumed to be proper unless otherwise
specified.  Also, with another convenient abuse of terminology, we use the term
\emph{$\spc1$-standard curve} to indicate also the function $G$ or the map $\bG$.  Two
$\spc1$-standard curves $\Ga$ and $\Gb$ are said to be \emph{stacked} if their projection
on the $x$ axis coincide; we say that $\Ga$ and $\Gb$ are $\Delta$\emph{-stacked} if they
are stacked and $\|\Ga-\Gb\|\nc1<\Delta$.


Let us fix $\dttrho0>0$ once again to be specified in due course.  For any $\spc2>0$
define the set of $\spc2$-\emph{standard} probability densities on the standard curve
$G$ as
\begin{align*}
  D_{\spc2}(G)=\left\{\rho\in \cC^2([a,b],\bRp)\st \int_a^b\rho(x)\deh x=1,\
\left\|\frac{\rho'}{\rho}\right\|\leq
  c_2,\,\left\|\frac{\rho''}{\rho}\right\|\leq \dttrho0\spc2\right\}.
\end{align*}
A \emph{$(\spc1,\spc2)$-standard pair} $\ell$ is given by $\ell=(\bG,\rho)$, where
$G\in\Sigma_{\spc1}$ and $\rho\in D_{\spc2}(G)$.  We similarly define \emph{short
  $(\spc1,\spc2)$-standard pairs}, by allowing $G$ to be a short $\spc1$-standard curve.
We define $|\ell|=b-a$ to be the \emph{length of} $\ell$.  A $(\spc1,\spc2)$-standard pair
$\ell=(\bG,\rho)$ uniquely identifies a probability measure $\mu_\ell$ on $\bT^2$ defined
as follows: for any Borel-measurable function $g$ on $\bT^2$ let
\[
\mu_\ell(g):=\int_a^{b} g(\bG(x))\rho(x) \deh{}x.
\]
Let $\stdpSet{\spc1,\spc2}$ denote the set of all $(\spc1,\spc2)$-standard pairs.

\subsubsection{Standard families}\label{sec:standardfamiles} A standard family can be conveniently regarded as a
\emph{random standard pair}.  More precisely: a \emph{$(\spc1,\spc2)$-standard family}
$\stdf$ is given by a Lebesgue probability space%
\footnote{ Recall that a probability space is a Lebesgue space if it is isomorphic to the
  disjoint union of an interval $[0,a]$ with Lebesgue measure and (at most) countably many
  atoms.} %
$\malpha=(\alphaset,\cF,\fm)$ and a $\cF$-measurable\footnote{ The set
  $\stdpSet{\spc1,\spc2}$ of $(\spc1,\spc2)$-standard pairs is in fact a space of smooth
  functions; it is thus a measurable space with the Borel $\sigma$-algebra.  More in
  detail, if $\bG:[a,b]\to\bT^2$ and $\rho:[a,b]\to\bR^{+}$ are defined as above, let
  $\hat\bG$ and $\hat\rho$ be defined by precomposing $\bG$ and $\rho$ respectively with
  the affine orientation-preserving map $[0,1]\to[a,b]$.  A standard pair-valued function
  is thus $\cF$-measurable if both maps $(\alpha,s)\mapsto \hat\bG_{\alpha}(s)$ and
  $(\alpha,s)\mapsto \hat\rho_{\alpha}(s)$ are jointly measurable. In particular, for any
  Borel set $E\subset\bT^2$, the function $\alpha\mapsto\mu_{\fellf{\alpha}}(E)$ is
  $\cF$-measurable.}  map $\fell:\alphaset\to\stdpSet{\spc1,\spc2}$.

For simplicity's sake, in this paper we will mostly restrict to standard families such
that $\malpha=(\alphaset,\cF,\fm)$ is a discrete probability space (\ie, $\alphaset$ is
at most countable and $\cF$ is the power set of $\alphaset$). We will thus imply that
$\alphaset$ is at most countable, and simply write $\malpha=(\alphaset,\fm)$, otherwise
explicitly stated.  We will denote the set of all $(\spc1,\spc2)$-standard families by
$\stdfSet{(\spc1,\spc2)}{}$.

A $(\spc1,\spc2)$-standard family $\stdf$ identifies a unique probability measure $\hat\mu_\stdf$
on the product space $\prodSet{\alphaset}{\bT^2}$ (with the product $\sigma$-algebra): for
any measurable function $\hat g$ on $\prodSet{\alphaset}{\bT^2}$ let
\[
\hat\mu_\stdf(\hat g):=\int_\alphaset\mu_{\fellf\alpha}(\hat g\prodElement{\alpha}{\cdot})\deh\fm.
\]
Define the \emph{support} of $\stdf$ as \(\supp\stdf =
\supp\hat\mu_\stdf\subset\prodSet{\alphaset}{\bT^2}\).  The natural projection
\(\npi:\prodSet{\alphaset}{\bT^2} \to \bT^2 \) induces a probability measure on $\bT^2$
which we denote by $\mu_\stdf=\npi_*\hat\mu_\stdf$; in other words, for any
Borel-measurable function $g$ of $\bT^2$, let
\[
\mu_\stdf(g):=\int_\alphaset\mu_{\fellf{\alpha}}(g)\deh\fm.
\]
Clearly, we have \( \supp\mu_\stdf=\npi\,\supp\stdf\).\footnote{
  This concept can be obviously applied to a single standard pair, considering it a family
  with just one element.  In such case, the support of the standard pair and the support
  of the associated measure can be trivially identified.}  We therefore obtain a
correspondence between $(\spc1,\spc2)$-standard families and probabilities on $\bT^2$; we
denote by $\sim$ the equivalence relation induced by the above correspondence \ie we let
$\stdf\sim\stdf'$ if and only if $\mu_\stdf=\mu_{\stdf'}$. We denote with $\eqc\stdf$ the
corresponding equivalence class, which therefore uniquely identifies a probability
measure.  We say that a probability measure $\mu$ \emph{admits a $(\spc1,\spc2)$-standard
  disintegration} if there exists a $(\spc1,\spc2)$-standard family $\stdf$ so that
$\mu_\stdf=\mu$; we write $\stdf\in\stdfSet{\spc1,\spc2}{\mu}$.

\subsubsection{Conditioning}
Let $\stdf=\SFF{\fell}{(\alphaset,\nu)}\in\stdfSet{(\spc1,\spc2)}{}$; a family
$\stdf'=\SFF{\fell'}{(\alphaset',\nu')}$ is said to be a \emph{subfamily of} $\stdf$
(denoted with $\stdf'\subfamily\stdf$) if
\begin{itemize}
\item $\supp\stdf'\subset\supp\stdf$, that is: $\alphaset'\subset\alphaset$ and
  $\fa\alpha\in\alphaset'$ we have $\supp\fellf\alpha'\subset\supp\fellf{\alpha}$;
\item for any measurable set $E\subset\prodSet\alphaset{\bT^2}$,
  $\hat\mu_{\stdf'}(E)={\hat\mu_{\stdf}(E\cap\supp\stdf')}/{\hat\mu_\stdf(\supp\stdf')}$.
\end{itemize}
Given $\alphaset'\subset\alphaset$, we define the \emph{subfamily conditioned on
  $\alphaset'$} to be
$\stdf\cond{\alphaset'}=\SFF{\fell\cond{\alphaset'}}{(\alphaset',\nu')}$, where
$\nu'(E)=\nu(E|\alphaset')$ and $\fell\cond{\alphaset'}$ is the restriction of $\fell$ on
$\alphaset'$.

\subsubsection{Convex combinations of pairs and families} We call a real number $\kappa$ a
\emph{weight} if $\kappa\in[0,1]$. %
Given a (at most countable) collection of $(\spc1,\spc2)$-standard families $\{\stdf_j =
\SFF{\fell_j}{\malpha_j}\}$ together with a collection of weights $\{\kappa_j\}$ such that
$\sum_j\kappa_j=1$, we can define the \emph{convex combination} $\sum_j\kappa_j\stdf_j$ as
the $(\spc1,\spc2)$-standard family $\stdf=\SFF{\fell}{\malpha}$ obtained by ``choosing a
standard family $\stdf_j$ at random with probability $\kappa_j$''.  More precisely, let
$\malpha=(\alphaset,\fm)$ be the discrete probability space given by
$\alphaset=\{(j,\alpha)\st \alpha\in\alphaset_j\}$ and measure $\fm=\sum_j\kappa_j\cdot
\inj_{j*}\fm_j$, where $\inj_j$ is the natural injection $\inj_j:\alphaset_j\to\alphaset$.
Last, let us define the random element $\fell$ as $\fell(j,\alpha)=\fellf{\alpha}_j$;
clearly $\mu_{\stdf}=\sum_j\kappa_{j}\mu_{\stdf_j}$.  With this in mind, observe that we
can recover the components of a convex combination by conditioning with respect to the
events $\bar\alphaset_k=\{(j,\alpha)\st j=k,\alpha\in\alphaset_k\}$. Observe, moreover,
that standard families can naturally be regarded as convex combinations of standard pairs.

\subsection{Standard pairs and dynamics}
Having made precise the concept of standard pair and families, our next step is to
illustrate their relation with the dynamics generated by the map $F_\ve$.
\subsubsection{Invariance}
As a first step we study the evolution of a $(\spc1,\spc2)$-standard pair.
\begin{prop}[Invariance]\label{p_invarianceStandardPairs}
  There exist $\spc1$, $\spc2$ such that, if $\ve$ is sufficiently small and $\ell$ is a
  $(\spc1,\spc2)$-standard pair, $F_{\ve*}\mu_{\ell}$ admits a $(\spc1,\spc2)$-standard
  disintegration.
\end{prop}
\begin{rem}
  The above proposition is a simplified version of the corresponding Proposition 3.3
  in~\cite{DeL1} where it is proved in a more general setting.  Since there are a few
  differences in the notation and terminology between this version and the one
  of~\cite{DeL1}, we prefer to give an adapted proof below for the reader's convenience.
  Despite its technical nature, the proof is instrumental for a few definitions which will
  be given later.  We thus prefer to give it now rather than relegating it to some appendix.
\end{rem}
\begin{proof}
  Let $\ell=(\bG,\rho)$ be a $(\spc1,\spc2)$-standard pair. For any sufficiently smooth function $A$ on $\bT^2$, by the
  definition of standard curve, it is trivial to check that:
  \begin{subequations}\label{e_AG}
    \begin{align}
      \|(A\circ\bG)'\| &\leq \|\deh A\| (1+\ve c_1)\label{e_AG'}\\
      \|(A\circ\bG)''\| &\leq \ve\|\deh A\| \dtt0c_1+\|\deh A\|\nc1(1+\ve c_1)^2\label{e_AG''}\\
      \|(A\circ\bG)'''\| &\leq \ve\|\deh A\| \dtt1c_1+\|\deh A\|\nc2(1+\ve(1+\dtt0) c_1)^3\label{e_AG'''}.
    \end{align}
  \end{subequations}
  Let us then introduce the maps $f_{\bG}= f\circ\bG$ and $\omega_{\bG}= \omega\circ\bG$.
  Recall that $\lambda > 2$, defined in Section~\ref{sec:results} denotes the minimal
  expansion of $f_\theta$; we will assume $\ve$ to be small enough (depending on our
  choice of $c_1$) so that $f_{\bG}'\geq \lambda-\ve c_1 \|\partial_\theta f\|>3/2$; in
  particular, $f_{\bG}$ is an expanding map.  Provided $\delta$ has been chosen small
  enough, $f_\bG$ is invertible. Let $\vf(x)=f_{\bG}\invr(x)$. Differentiating we obtain
  \begin{align}\label{e_estimatesVf}
    \vf'   &= \frac1{f_\bG'}\circ\vf &%
    \vf''  &= -\frac{f_\bG''}{f_\bG'^3}\circ\vf &%
    \vf''' &= \frac{3f_\bG''^2-f_\bG'''f_\bG'}{f_\bG'^5}\circ\vf.
  \end{align}
  We can thus write:
  \begin{align*}
    F_{\ve*}\mu_\ell(g)=\mu_{\ell}(g\circ F_{\ve})&=\int_{a}^{b} g(f_\bG(x),\bar
    G(x))\rho(x) \deh x\\&=\int_{f_{\bG(a)}}^{f_{\bG(b)}} g(x, \bar
    G(\vf(x)))
    \cdot\rho(\vf(x))\vf'(x)\deh x,
  \end{align*}
  where $\bar G(x):=G(x)+\ve\omega_\bG(x)$.  Fix a partition (mod $0$) of
  $[f_{\bG}(a),f_{\bG}(b)]=\bigcup_{j\in\cJ}[a_{j},b_{j}]$, with
  $b_{j}-a_{j}\in[\delta/2,\delta]$ and $b_{j}=a_{j+1}$.  We can thus write
  \begin{align*}
    F_{\ve*}\mu_\ell(g) &= \sum_j\fmm_{j}\int_{a_{j}}^{b_{j}} g(x, G_j(x))
    \cdot\rho_j(x)\deh x=
    \sum_j \fmm_{j}\mu_{(\bG_{j},\rho_{j})}(g).
  \end{align*}
  provided that $G_{j}= \bar G\circ\vf_{j}$ and $\rho_{j} = \fmm_{j}^{-1} \cdot
  \rho\circ\vf_{j}\cdot\vf_{j}'$ where $\vf_j=\vf|_{[a_j,b_j]}$ and $ \fmm_{j} =
  \int_{a_{j}}^{b_{j}} \rho(\vf_{j}(x))\vf'_{j}(x)\deh x$.  Observe that, by construction,
  we have $\sum_j \fmm_{j}=1$.  Differentiating the above definitions and
  using~\eqref{e_estimatesVf} we obtain
  \begin{subequations}\label{e_C1}
    \begin{align}
      G_j' &=%
      \frac{\bar G'}{f'_\bG}\circ\vf_j\label{e_c1'}\\%
      G_j'' &=%
      \frac{\bar G''}{f'^2_\bG}\circ\vf_j%
      -G_j'\cdot\frac{f_\bG''}{f_\bG'^2}\circ\vf_j\label{e_c1''}\\%
      G_j''' &=%
      \frac{\bar G'''}{f'^3_\bG}\circ\vf_j%
      -3 G''_j\cdot\frac{f_\bG''}{f_\bG'^2}\circ\vf_j-
      G'_j\cdot\frac{f_\bG'''}{f_\bG'^3}\circ\vf_j\label{e_c1'''}%
    \end{align}
  \end{subequations}
  and similarly
  \begin{subequations}\label{e_C2}
    \begin{align}
      \frac{\rho_j'}{\rho_j}&=\frac{\rho'}{\rho\cdot f'_\bG}\circ\vf_j-\frac{f_\bG''}{f_\bG'^2}\circ\vf_j\label{e_c2'}\\
      \frac{\rho_j''}{\rho_j}&=\frac{\rho''}{\rho\cdot
        f'^2_\bG}\circ\vf_j-3\frac{\rho'_j}{\rho_j}\cdot
      \frac{f_\bG''}{f_\bG'^2}\circ\vf_j -\frac{f_\bG'''}{f_\bG'^3}\circ\vf_j.\label{e_c2''}
    \end{align}
  \end{subequations}
  Using~\eqref{e_c1'}, the definition of $\bar G$
  and~\eqref{e_AG'} we obtain, for small enough $\ve$:
  \begin{align*}
    \|G'_j\| &\leq \left\|\frac{ G' + \ve\omega_\bG'}{f_\bG'}\right\|
    \leq\frac23(1+\ve\|\deh\omega\|)\ve c_1 + \frac23\ve\|\deh\omega\|\\%
    &\leq\frac34\ve c_1+ \ve D_1%
  \end{align*}
  where $D_1=\frac23\|\deh\omega\|$.  We can then fix $c_1$ large enough so that the
  right hand side of the above inequality is less than $c_1$.  Next we will use $C_*$ for a generic constant depending on $c_1,D_1,D_1',c_2, D_2$ and $\Const$ for a generic constant depending only on $F_\ve$. Then, we find\footnote{ The reader can easily fill in the
    details of the computations.}
  \begin{align*}
    \|G''_j\|&\leq \frac34\ve [ c_1 D_1+\Const]+\ve^2 C_*\,;&%
    \|G'''_j\|&\leq \frac34\ve\left[ c_1( D_1'+D_1\Const+\Const)+\Const\right]+\ve^2 C_*\\
    \left\|\frac{\rho'_j}{\rho_j}\right\|&\leq \frac34 c_2 + \Const+\ve C_*\,;&
    \left\|\frac{\rho''_j}{\rho_j}\right\|&\leq \frac34 c_2 [D_2 +\Const]+\Const+\ve C_*.
  \end{align*}
  We can then fix $c_1,D_1',c_2, D_2$ sufficiently large and then $\ve$ sufficiently small
  to ensure that the $(\bG_{j},\rho_{j})$ are standard pairs.  We have thus obtained a
  decomposition of $F_{\ve*}\mu_{\ell}$ given by the discrete standard family
  $\stdf'=\SFF{\ell_j}{(\cJ,\fmm_{j})}$.
\end{proof}
\begin{rem}\label{rem:inv-F} The construction described in the above proposition yields more than just a
standard disintegration of $F_{\ve*}\mu_\ell$. In fact, it gives an \emph{invertible} map
$\liftMap:\supp\ell\to \supp\stdf$ such that $F_\ve=\npi\circ\liftMap$ and
$\hat\mu_\stdf=\liftMap_*\mu_\ell$ (such map does not exist in general for a standard
disintegration of $F_{\ve*}\mu_\ell$).
\end{rem}
It is immediate to extend the above proposition to standard families: let
$\stdf=\SFF{\fell}{(\alphaset,\fm)}$ be a standard family; then by definition we have, for
any measurable function $g$:
\begin{align*}
  F_{\ve*}\mu_\stdf(g) &=%
  F_{\ve*} \sum_{\alpha\in\alphaset}\fm_{\alpha}\mu_{\ell_{\alpha}}(g) =
  \sum_{\alpha\in\alphaset}\fm_{\alpha}F_{\ve*}\mu_{\ell_{\alpha}}(g) =
  \sum_{\alpha\in\alphaset}\fm_{\alpha} \mu_{\stdf'_{\alpha}}(g)
\end{align*}
where $\stdf'_\alpha$ is the standard family obtained by applying
Proposition~\ref{p_invarianceStandardPairs} to $\ell_\alpha$.  We conclude that the convex
combination
\begin{align*}
  \stdf'=\sum_{\alpha\in\alphaset}\fm_{\alpha}\stdf'_{\alpha}
\end{align*}
is a standard disintegration of $F_{\ve*}\mu_{\stdf}$; moreover there exists an invertible
map (which we still denote) $\liftMap:\supp\stdf\to\supp\stdf'$ so that
$\npi\circ\liftMap=F_\ve\circ\npi$ and $\liftMap_*\hat\mu_\stdf=\hat\mu_{\stdf'}$.

\subsubsection{Pushforwards and filtrations}\label{ss_pushforwards}
A standard disintegration of $F_{\ve*}\mu_{\stdf}$ equipped with a map $\liftMap$ as above
is called a $(\spc1,\spc2)$-\emph{standard pushforward} of $\stdf$.  A (finite or
countable) sequence $\{\stdf_n\}$ is said to be a \emph{sequence of
  $(\spc1,\spc2)$-standard pushforwards} of $\stdf_0$ if for each $n\ge0$, $\stdf_{n+1}$
is a $(\spc1,\spc2)$-standard pushforward of $\stdf_n$. At times, when some confusion
might arise, we will write $\stdf_n(\stdf)$ to make clear that $\stdf_n$ is a pushforward
of the family $\stdf$.

Let us comment on the above important definition
\begin{rem}\label{rem:lofp}
  Consider a sequence of $(\spc1,\spc2)$-standard pushforwards of a standard pair $\ell$;
  it is instructive to consider the sequence $\stdf_n$ as a random process.  For each
  $p\in\supp\ell$, let $\alphaMap{n}:\supp\ell\to\alphaset_n$ be the map
  $\alphaMap{n}=\npiA\circ\liftMap^{n}$.\footnote{ Obviously, $\npiA(\alpha,p)=\alpha$,
    for each $\alpha\in\alphaset, p\in\bT^2$.} Next, let us introduce the shorthand
  (abusing but suggestive) notation $\fellf{p}_n=\fellf{\alphaMap{n}(p)}_n$. Accordingly,
  the sequence of functions $\{\fell_n\}$ can be regarded as a random process on the
  standard pair $\ell$ with values in the space of standard pairs.
\end{rem}

Observe moreover that our construction of $\liftMap$ implies the following important
property: given $\alpha\in\alphaset_n$ let $\gap_n(\alpha)$ be the connected subcurve
$\alphaMap{n}\invr(\alpha)\subset\supp\ell$ whose $n$-image is $\fellf{\alpha}_n$; then
let $\cF_n$ be the $\sigma$-algebra generated by the collection
$\{\gap_n(\alpha)\}_{\alpha\in\alphaset_n}$ (\ie, the $\sigma$-algebra generated by
$\alphaMap{n}$). The sequence $\{\cF_n\}$ is a filtration and the process
$\{\alphaMap{n}\}$ (or, loosely speaking, $\{\fell_n\}$) is (naturally) adapted to such a
filtration.

For each $p\in\supp\ell$ let us also introduce the shorthand notation
$\gap_n(p)=\gap_n(\alphaMap{n}(p))$: observe that standard distortion arguments yield:
\begin{align}\label{e_gapDistortion}
  \Const\invr\Lambda_n(p)\invr\le|\gap_n(p)|\le\Const\Lambda_n(p)\invr,
\end{align}
where $\Lambda_{n}(p)=\frac{\deh x_n}{\deh x_0}$ and the derivative is taken along the
curve; in particular $|\gap_n(p)|\le\Const2^{-n}$.

Henceforth we assume $\spc1$, $\spc2$ to be fixed in order for
Proposition~\ref{p_invarianceStandardPairs} to hold and we fix $\delta$ to be so small
that $\delta c_2<1/50$.  Moreover, since $\spc1$ and $\spc2$ are now fixed, we will refer
to a $(\spc1,\spc2)$-standard pair (\resp family, pushforward) simply as a \emph{standard
  pair} (\resp family, pushforward); we let --with a further slight abuse of notation--
$\eqc{\pFve\stdf}=\stdfSet{\spc1,\spc2}{F_{\ve *}\mu_\stdf}$.

The proof of Proposition~\ref{p_invarianceStandardPairs} in fact shows the existence of a
standard pushforward of any standard family $\stdf$.  A pair $\ell$ is said to be
\emph{$N$-prestandard} if $F^N_{\ve*}\mu_\ell$ admits a standard decomposition; we say
that $\ell$ is \emph{prestandard} if it is $N$-prestandard for some $N$.  We say that a
family $\stdf$ is \emph{$N$-prestandard} (\resp \emph{prestandard}) if every
$\ell\in\stdf$ is $N$-prestandard (\resp prestandard).
\begin{rem}\label{r_short}
  Consider a short standard pair $\ell$ of length at least $\delta_{*}$: the proof of
  Proposition~\ref{p_invarianceStandardPairs} implies that standard curves are expanded at
  an exponential rate.  We can conclude that $\ell$ is $\NRg$-prestandard with
  $\NRg\sim\Const|\log\delta_{*}|$.  We call $\NRg$ the \emph{recovery time} of $\ell$.
\end{rem}
\begin{rem}\label{r_prestandard}
  Let $\ell$ be a $(\spc1,\pres\spc2)$-standard pair with $\pres>1$: the proof of
  Proposition~\ref{p_invarianceStandardPairs} implies that densities on standard curves
  are regularized by the dynamics at an exponential rate; hence $\ell$ is
  $\NRg$-prestandard with $\NRg\sim \Const\, {\ln \pres}$.  Again, we call $\NRg$ the
  \emph{recovery time} of $\ell$.
\end{rem}

\begin{rem}\label{r_flatDensity}
  Consider a standard pair $\ell=(\bG,\rho)$; by definition of standard density, we have,
  for any $x\in[a,b]$:
  \begin{equation}
    \frac{\expo{-2c_2\delta}}{|\ell|}\leq\rho(x)\leq\frac{\expo{2c_2\delta}}{|\ell|}.\label{e_boundRho}
  \end{equation}
  Consequently, for any constant $m_{*}\leq1/2$, we can define $\hat\rho(x)$ so that
  $\rho(x)=m_*/|\ell|+\hat\rho(x)$, and by the above estimate and our choice for $\delta$
  we have $\hat\rho(x)\geq\rho(x)/3$.  Consequently, since $\hat\rho'=\rho'$ (and thus
  $\hat\rho''=\rho''$), we have:
  \begin{align*}
    \left\|\frac{\hat\rho'}{\hat\rho}\right\| &%
    \leq 3\left\|\frac{\rho'}{\rho}\right\|%
    \leq 3 c_2'&
    \left\|\frac{\hat\rho''}{\hat\rho}\right\| &%
    \leq 3\left\|\frac{\rho''}{\rho}\right\|%
    \leq 3 c_2''
  \end{align*}
  \ie  $\hat\rho(x)\in D_{3\spc2}(G)$.  The standard pair $\ell$ can thus be \emph{split}
  as:
  \[
  \ell \sim m_*\ell_* + (1-m_*)\hat\ell,
  \]
  where $\ell_*=(\bG,1/|\ell|)$ is a standard pair and $\hat\ell=(\bG,\hat\rho/(1-m_*))$
  is a $\cO(1)$-prestandard pair.
\end{rem}
A fundamental property of standard families is that any SRB measure is a weak limit of a
sequence of measures that can be disintegrated into standard families; we do not give the
proof of this fact here, since we will prove a slightly stronger statement in
Lemma~\ref{l_weakLimit}.

\section{Averaged dynamics}\label{s_averaging}

Standard pairs are a very convenient way to describe initial conditions which are, in a
sense, well distributed with respect to the dynamics (see the discussion at the beginning
of Section~\ref{sec:coupling-def} for further comments).  Let us start making this vague
statement more concrete by stating some results which follow from the  ones
that are proved in~\cite{DeL1}.  First of all, let us introduce some useful notation; recall
that for $p\in\bT^2$ we denote $(x_n(p),\theta_n(p)) = F_\ve^n(p)$; recall moreover the
definition of $\zeta_n$ given in~\eqref{e_definitionZ}; let
$\slo_n(p)=(\theta_n(p),\zeta_n(p))$ and define the polygonal interpolation
\begin{align*}
  \slo_\ve(t;p)=%
  \slo_{\pint{t\vei}}(p) +%
  (t\vei-\pint{t\vei})%
  (\slo_{\pint{t\vei}+1}(p)-\slo_{\pint{t\vei}}(p)).
\end{align*}
Let us also introduce the functions $\theta_\ve,\, \zeta_\ve$ so that
$\slo_\ve(t;p) = (\theta_\ve(t;p),\zeta_\ve(t;p))$.  Note that if $p = (x_0,\theta_0)$ is
distributed according to some measure $\mu$ (\eg $\mu = \mu_\ell$, where $\ell$ is a
standard pair), then for any $T > 0$ $\slo_\ve$ is naturally a random variable with values
in $\cC^0([0,T],\bT\times\bR)$ (and likewise $\theta_\ve$ and $\zeta_\ve$) and thus the
pushforward $\slo_{\ve*}\mu_\ell$ is a probability on $\cC^0([0,T],\bT\times\bR)$.

For any $t\geq 0$ and $\theta_*\in\bT$, we define the function
$\bar\slo(t;\theta_*) = (\bar\theta(t;\theta_*),\bar\zeta(t;\theta_*)))$ to be the
solution of the ODE problem
\begin{align}\label{e_averagedSlo}
  \frac{\deh}{\deh t}\bar\slo(t;\theta_*)&%
  =\left(\bar\omega(\bar\theta(t;\theta_*)),\bclyapReg(\bar\theta(t;\theta_*))\right), \text{ with }\bar \slo(0;\theta_*)=(\theta_*,0),
\end{align}
where $\bar\omega(\theta)=\int_{\bT}\omega(x,\theta)\rho_\theta(x) \deh x$,
$\bar \clyapReg(\theta)=\int_{\bT}\clyapReg(x,\theta)\rho_\theta(x)\deh x$, $\clyapReg$ is
defined in~\eqref{eq:regularizedpsi}, and $\rho_\theta$ is the density of the unique
absolutely continuous invariant measure for the expanding map $f_\theta$.  Observe
that~\eqref{e_averagedSlo} admits a unique solution since we use the regularized function
$\bclyapReg$, which is smooth by~\ref{rem:newpsi}.

Then, the \emph{Averaging Principle} (see~\cite[Theorem~2.1]{DeL1}) states that for any
$T > 0$, if the initial conditions $(x_0,\theta_0)$ are distributed on a standard
pair\footnote{ Notice that the definition of standard pair in fact depends on $\ve$,
  therefore this convergence holds for any sequence of standard pairs which in turn weakly
  converges to the flat standard pair $\{\theta=\theta_*\}$. } that crosses
$\{\theta=\theta_*\}$, the random variable $\slo_\ve$ converges in probability to
$\bar\slo(\cdot;\theta_*)$ on $[0,T]$ as $\ve\to 0$.

\subsection{Large and moderate deviations}\label{ss_limitThms}
Given the above facts, it is then natural to attempt a description of the behavior of
deviations from the averaged dynamics.  For $p = (x_0,\theta_0)$, let us define\footnote{
  The function $\Delta\slo$ (and thus $\Delta\theta$ and $\Delta\zeta$) indeed depend on
  $\ve$ (since so does $\slo_\ve$); however, we do not explicitly add a subscript $\ve$
  for ease of notation.}:
\begin{align}\label{e_defineDeltas}
\Delta\slo(t;p)&=(\Delta\theta(t;p),\Delta\zeta(t;p)):=\slo_\ve(t;p)-\bar\slo(t;\theta_0).
\end{align}
A first rough (but useful) result that follows from~\cite{DeL1} is the following
\begin{thm}\label{t_largeDevz}
  There exists $\ve_0>0$ such that if we fix $T>0$, then there exist $\bar C>0$ such that
  for any $\ve\leq \ve_0$, $R\ge\bar C\sqrt\ve$ and standard pair $\ell$, we have
  \begin{align*}
    \mu_\ell\left(\sup_{t\in[0,T]}|\Delta\theta(t;\cdot)|\ge R\right)&\le\expo{-c_T R^2\vei}
  \end{align*}
  where $c_T$ is a constant which depends on $T$ only.
  The same statement, where we replace $\Delta\theta$ with $\Delta\zeta$, holds.
\end{thm}
\begin{proof}
  In this proof we will be using notations introduced in~\cite{DeL1}; for simplicity, we
  will prove the statement for $\Delta\theta$ only, leaving the similar derivation of the
  case of $\Delta\zeta$ to the reader.  We will invoke~\cite[Theorem~2.2]{DeL1} with
  $A = \omega$; indeed the statement of~\cite[Theorem~2.2]{DeL1} involves a probability
  $\bP_{\omega,\ve}$ on the space $\cC^0([0,T],\bT)$; in our case we take
  $\bP_{\omega,\ve} = \theta_{\ve*}\mu_\ell$ (see~\cite[Remark~6.2]{DeL1}).  Hence, let
  $\widehat Q=\{p\in \bT^2\st \sup_{t\in[0,T]}|\Delta\theta(t;\cdot)|\ge R\}$. By
  construction, if we define
  $Q=\{\gamma\in \cC^0([0,T],\bT)\;:\; \|\gamma(\cdot)-\bar\theta(\cdot,
  \gamma(0))\|_\infty\ge R\}$,
  then ${\mu_\ell(\widehat Q)} = {\bP_{\omega,\ve}(Q)}$, which establishes the connection
  between the two notations.

  Next, by~\cite[Lemma~6.6 and Remark~6.7]{DeL1} we gather that, for any $\theta\in\bT$,
  $\cZ(\cdot,\theta)$ (defined in~\cite[(6.4)]{DeL1}) is finite only in a compact set on
  which it is bounded.  This, together with~\cite[Lemma~6.3]{DeL1} implies that there
  exists $c>0$ such that $\cZ^-_{\Delta_*}(b,\theta)\geq c (b-\bar\omega(\theta))^2$
  (see~\cite[(6.11)]{DeL1} for the definition of $\cZ^-_{\Delta_*}$) for all $\Delta_*>0$
  small enough, $\theta\in\bT$.  Let us fix
    some small $\Delta_*$, then (see~\cite[(6.12)]{DeL1} for the definition of
    $\mathscr{I}_{\theta,\Delta_*}^-$):
  \begin{equation}\label{eq:silly-billy}
    \mathscr{I}_{\theta_0,\Delta_*}^-(\gamma) =
    \int_{0}^{T}\cZ_{\Delta_*}^-(\gamma'(t),\gamma(t))\deh t\geq c
    \int_0^T\left(\gamma'(t)-\bar\omega(\gamma(t))\right)^2 \deh t,
 \end{equation}
 where we assumed that $\gamma$ is Lipschitz (otherwise
 $ \mathscr{I}_{\theta_0,\Delta_*}^-=\infty$ by definition).  Next, let
 $\eta(t)=\gamma(t)-\bar\theta(t, \gamma(0))$ and
 $\zeta(t)=\gamma'(t)-\bar\omega(\gamma(t))$; observe that
 $\eta'(t) = \gamma'(t)-\bar\omega(\bar\theta(t,\gamma(0)))$, so
\[
|\eta'(t)|^2\leq \left(|\zeta(t)|+\Const |\eta(t)|\right)^2\leq 2 |\zeta(t)|^2+\Const T\int_0^t|\eta'(s)|^2\deh s,
\]
which, by Gr\"onwall inequality, implies
\[
|\eta'(t)|^2\leq \Const(1+T\Const e^{\Const T^2})  |\zeta(t)|^2.
\]
Substituting the above in~\eqref{eq:silly-billy} yields
\[
 \mathscr{I}_{\theta_0,\Delta_*}^-(\gamma) \geq  C_T R^2
\]
for some $C_T>0$ dependent on $T$.  We can now conclude by applying~\cite[Theorem
2.2]{DeL1}.  To do so note that
$R_+(\gamma)= \ve^{1/6}\|\gamma-\bar\theta(\cdot, \gamma(0))\|_\infty^{2/3} \leq \frac
12\|\gamma-\bar\theta(\cdot, \gamma(0))\|_\infty$
provided that $\|\gamma-\bar\theta(\cdot, \gamma(0))\|_\infty\geq 8\ve^{1/2}$. This
implies that
$Q^+\subset Q_1=\{\gamma\in \cC^0([0,T],\bT)\;:\; \|\gamma(\cdot)-\bar\theta(\cdot,
\gamma(0))\|_\infty\ge \frac 12 R\}$.
In addition, $\varrho(Q_1,\theta_*)\geq 2R^{-1}$ which means that
$C_{\Delta_*,T}\ve^{1/72}\varrho(Q_1,\theta_*)^{-1/36}\leq 1/2$ provided $\bar C$ has been
chosen large enough. Accordingly, provided again $\bar C$ is large enough and $\ve$ small
enough:
\[
\mu_\ell(\widehat Q)=\bP_{\omega,\ve}(Q)\leq \bP_{\omega,\ve}(Q_1)\leq \expo{-\frac 1{16} C_T R^2\vei}\qedhere
\]
\end{proof}

We now proceed to prove two other results (Theorems~\ref{l_largeDevzLowerBound}
  and~\ref{t_forbiddenForRealDyn}) which also follow from the Large Deviations estimates
  obtained in~\cite{DeL1}. Loosely speaking these result state that if there exists an
  admissible \tpath\thetaa\thetab, then there exists an orbit of the real system
  connecting a neighborhood of $\{\theta = \thetaa\}$ to a neighborhood of
  $\{\theta = \thetab\}$; conversely if all \tpath\thetaa\thetab{s} are not admissible, we
  would like to say that no orbit of the real system connects the two said sets.  Indeed
  such statements could only have a chance to hold true in the limit $\ve\to0$, and even
  in this case there would be borderline admissible paths for which none of our statements
  would hold.  In order to properly state our results in the case $\ve > 0$ we need to refine the notion of admissibility that has been introduced in Section~\ref{sec:results};
recall the definition of $\Omega(\theta)$ given in~\eqref{e_definitionOmega}; in
particular $\Omega(\theta)$ is a closed interval for any $\theta\in\bT$.  For $\aeps > 0$
and $\theta\in\bT$, introduce the notations
\begin{align*}
\Omp(\theta) &= \Omega(\theta)\cup\partial_{\aeps}\Omega(\theta)&
\Omm(\theta) &= \Omega(\theta)\setminus\partial_{\aeps}\Omega(\theta),
\end{align*}
where $\partial_\aeps\Omega(\theta) = \{b\st \dist(b,\partial\Omega(\theta)) < \aeps\}$.
It is immediate to observe that if $\aeps' < \aeps$, $\Ompx{\aeps'}\subset\Omp$ and
$\Ommx{\aeps'}\supset\Omm$; moreover $\intr{\Omega(\theta)} = \bigcup_{\aeps > 0}\Omm$ and
$\Omega(\theta) = \bigcap_{\aeps > 0}\Omp$.  We say that a \tpath{\thetaa}{\thetab}{} $h$
of length $T$ is \emph{$\aeps$-admissible} if for any $s\in[0,T]$ we have
$\dfpath(s)\subset\intr\Omm(h(s))$; likewise we say that $\fpath$ is
\emph{$\aeps$-forbidden} if for some $s\in[0,T]$ we have
$\dfpath(s)\not\subset\Omp(h(s))$.  Observe that, by definition, and by compactness of the
graph of $\dfpath(s)$, $\fpath$ is admissible if and only if it is $\aeps$-admissible for
some $\aeps > 0$.

First of all let us prove an auxiliary
\begin{lem}\label{l_continuityOmegapm}
  Let $\bar\omega^+(\theta) = \max\Omega(\theta)$ and
  $\bar\omega^-(\theta) = \min\Omega(\theta)$; then $\bar\omega^{\pm}$ are (uniformly)
  continuous functions.
\end{lem}
\begin{proof} Let us prove\footnote{ We are indebted to Ian Morris for suggesting this
    simple argument.} the statement for $\bar\omega^+(\theta)$ (the statement for
  $\bar\omega^-$ follows by noting that $\min\Omega(\theta) = -\max [-\Omega(\theta)]$).
  It is possible to characterize $\bar\omega^+$ as follows (see~\cite[Proposition
  2.1]{Jenkinson}):
  \begin{align*}
    \bar\omega^+(\theta) = \sup_{x\in\bT}\limsup_{N \to\infty}\frac1N\sum_{n = 0}^{N}\omega(f_\theta^n(x),\theta).
  \end{align*}
  Observe that for any $\varrho > 0$, if $|\thetaa-\thetab|$ is sufficiently small, there
  exists a diffeomorphism $\Phi:\bT\to\bT$ with $\|\Phi-\Id\|\nc0 < \varrho$ so that
  $f_{\thetab} = \Phi\circ f_{\thetaa}\circ \Phi\invr$ (see e.g.~\cite[Lemma~2]{JenkMorris}).  We gather that
  \begin{align*}
    \bar\omega^+(\thetab) = \sup_{x\in\bT}\limsup_{N \to\infty}\frac1N\sum_{n = 0}^{N}\omega(f_{\thetab}^n(x),\thetab) %
    =\sup_{x\in\bT}\limsup_{N \to\infty}\frac1N\sum_{n = 0}^{N}\omega(\Phi\circ f_{\thetaa}^n(x),\thetab).
  \end{align*}
  Since $\|\omega(\cdot,\thetab)-\omega(\Phi(\cdot),\thetab)\|\nc0$ can be made
  arbitrarily small by choosing $\varrho$ arbitrarily small, and since $\omega$ is a
  smooth function, our lemma follows.
\end{proof}
Recalling the definition of the Large Deviation Rate Function (see~\cite[(6.12)]{DeL1}) and its
characterization given by~\cite[Lemma~6.6]{DeL1}; we have the following lower bound.
\begin{thm}\label{l_largeDevzLowerBound}
  Assume that there exists an $\aeps$-\admissiblep{\thetaa}{\thetab} of length $T > 0$ for
  some $T > 0$ and $\aeps > 0$; if $\ve > 0$ is sufficiently small (depending on $T$ and
  $\aeps$), for any standard pair $\ell$ whose support intersects $\{\theta=\thetaa\}$ we
  have, if $C$ is sufficiently large,
  \begin{align}\label{e_largeDevzLowerBound}
    \mu_\ell(\theta_{\pint{T\vei}}\in B(\thetab,C\ve^{5/12}))\ge\expo{-\Const T\vei}.
  \end{align}
\end{thm}
\begin{proof}
  Let us assume $\ve$ is sufficiently small (with respect to $\aeps$ and $T$) to be
  specified later and define the set
  \begin{align*}
    Q_\ve = \{\fpath\in\cC^0([0,T],\bT)\st \fpath(T)\in B(\thetab,C\ve^{5/12})\}.
  \end{align*}
  Once again we plan to apply~\cite[Theorem~2.2]{DeL1} with $A = \omega$ where
  $\bP_{\omega,\ve} = \theta_{\ve*}\mu_\ell$.  We thus look for $\Delta_* > 0$ such that
  $\inf_{\fpath\in Q_\ve^-}\mathscr{I}_{\thetaa,\Delta_*}^+(\fpath) < \infty$, where
  $\mathscr{I}_{\thetaa,\Delta_*}^+$ is defined\footnote{ The crucial properties of
    $\mathscr{I}^+_{\thetaa,\Delta_*}$ which we will use in the proof are that
    $\mathscr{I}^+_{\thetaa,\Delta_*}(\fpath)$ is finite only if $\fpath$ is Lipschitz,
    $\fpath(0) = \thetaa$ and $\fpath$ is $\aeps$-admissible for some $\aeps > \Delta_*$.}
  in~\cite[(6.12)]{DeL1} and
  $Q^-_\ve = \{\fpath\in Q_\ve\st B_{[0,T]}(\fpath,\Const\ve^{5/12})\subset Q_\ve\}$ for
  some universal $\Const$.

  Let $\bar\fpath$ be an $\aeps$-admissible \tpath\thetaa\thetab{} of length $T$ (which
  exists by hypothesis).  Let us choose $\Delta_* < \aeps$ (\eg $\Delta_* = \aeps/2$), so
  that $\mathscr{I}_{\thetaa,\Delta_*}^+(\bar\fpath) < \infty$; hence, provided that
  $\bar\fpath\in Q_\ve^-$ (which holds true if $C > \Const$) and $\ve$ is sufficiently
  small (with respect to $\Delta_*$) we obtain
  \begin{align*}
    \mu_\ell(\theta_{\pint{T\vei}}\in B(\thetab,C\ve^{5/12})) \ge \expo{-C_{T,\bar \fpath}\vei}.
  \end{align*}
  Observe that, in the above inequality, the right hand side depends on the path
  $\bar\fpath$, but we want a uniform bound.  However,
  Proposition~\ref{p_invarianceStandardPairs} implies that if some point in the $n$-th
  image of a standard pair $\ell$ belongs to $\{\theta\in B(\thetab,C\ve^{5/12})\}$, then
  there is a whole standard pair in any $n$-pushforward of $\ell$ which belongs to (a
  $\cO(\ve)$-neighborhood of) the given set.  Hence the above inequality indeed
  implies~\eqref{e_largeDevzLowerBound}, where $\expo\Const$ is proportional to the
  maximal expansion of $F_\ve$ along a standard curve.
\end{proof}
We now prove what can be regarded as a converse of the above theorem.
\begin{thm}\label{t_forbiddenForRealDyn}
  For any $\aeps > 0$, there exist $\varrho > 0$ and $\TForb > 0$ so that if $\ve$ is
  sufficiently small (depending on $\aeps$) the following holds: if every
  \tpath{\thetaa}{\thetab} is $\aeps$-forbidden, then for any $N \ge \pint{\TForb\vei}$
  \begin{align}\label{e_disjointness}
    F_\ve^N(\{\theta \in B(\thetaa,\varrho)\})\cap \{\theta\in B(\thetab,\varrho)\} = \emptyset.
  \end{align}
\end{thm}
\begin{proof}
  The idea of the proof is to argue by contradiction: assume that there is an orbit
  connecting $B(\thetaa,\varrho)$ to $B(\thetab,\varrho)$; by continuity a small
  neighborhood of the initial point of the orbit will have the same property, hence we
  have a positive measure set of trajectories that start in $B(\thetaa,\varrho)$ and end
  in $B(\thetab,\varrho)$; on the other hand, using the original orbit, we will construct
  a path $\fpath_*$ so that every sufficiently $\cC^0$-close path $h$ is so that the large
  deviation rate function $\mathscr{I}_{\theta_0,\Delta_*}^-(h) = \infty$.  We then use
  this fact to show that the above set of trajectories must have zero measure, thus giving
  a contradiction.

  Uniform continuity of $\bar\omega^{\pm}$ (proved in Lemma~\ref{l_continuityOmegapm})
  guarantees that there exists $\unifsmall > 0$ so that if
  $|\theta-\theta'| < \unifsmall$,
  $|\bar\omega^\pm(\theta)-\bar\omega^\pm(\theta')| < \aeps/16$.  Define
  $\TForb = \unifsmall/\|\omega\|$ and let
  $\varrho = \unifsmall\min\{1,\aeps/(8\|\omega\|)\}$.  Assume that $\ve$ is sufficiently
  small (depending on $\aeps$ only, as we will prescribe later) and that there exists
  $(x_0,\theta_0)$ so that $\theta_0\in B(\thetaa,\varrho)$ and
  $\theta_N\in B(\thetab,\varrho)$ for some $N \ge\TForb\vei$.  We can thus partition
  $[0,N]$ in intervals as follows:
  \begin{align*}
    [0,N] = \bigcup_{k = 0}^{l-1}[n_k,n_{k+1}]
  \end{align*}
  where $n_k\in\bZ$, $n_0 = 0$, $n_l = N$ and
  $\pint{\TForb\vei/2}\le n_{k+1}-n_{k}\le\pint{\TForb\vei}$.  Given $t\in[0,N\ve]$,
  let $k(t) = \max\{k:n_k \le t\vei\}$ and let us define the polygonal approximation of
  the orbit given by:
  \begin{align}\label{e_polygonalApprox}
    \fpath_0(t) = \theta_{n_{k(t)}} + \frac{t\vei-n_{k(t)}}{n_{k(t)+1}-n_{k(t)}} ({\theta_{n_{k(t)+1}}-\theta_{n_{k(t)}}}).
  \end{align}
  First we claim that $\fpath_0$ is $\aeps/2$-forbidden: assume by contradiction that
  $\fpath_0$ is not $\aeps/2$-forbidden; let $\tilde\fpath_0$ the polygonal path obtained
  using definition~\eqref{e_polygonalApprox}, replacing $\theta_{n_0}$ with $\thetaa$ and
  $\theta_{n_l}$ with $\thetab$, respectively.  Observe that the two paths $\fpath_0$ and
  $\tilde\fpath_0$ coincide except in the first and last interval, in particular
  $E_{\fpath_0} = E_{\tilde\fpath_0}$.  Moreover
  $\|\fpath-\tilde\fpath\|\nc0 \le \varrho\le \unifsmall$, which in particular implies
  that $\|\bar\omega^\pm\circ\tilde\fpath_0-\bar\omega^{\pm}\circ\fpath_0\| < \aeps/4$,
  and for any $t\in[0,T]/E_{\fpath_0}$, we have
  $|\fpath_0'(t)-\tilde\fpath_0'(t)| \le 2\varrho/\TForb \le \aeps/4$, from which we
  conclude that $\tilde\fpath_0$ is a non $\aeps$-forbidden \tpath\thetaa\thetab, which
  contradicts our assumptions.

  We thus proved that $\fpath_0$ is $\aeps/2$-forbidden; its length, however, can be
  arbitrarily long, and in order to use~\cite[Theorem~2.2]{DeL1}, we need to extract from
  $\fpath_0$ a sub-path $\fpath_*$ of bounded length which is also $\aeps/2$-forbidden.
  This is simple, since by definition there exists $\ksp\in\{0,\cdots,l-1\}$ so that the
  restriction of $\fpath_0$ to the corresponding interval $[n_\ksp\ve,n_{\ksp+1}\ve]$ is
  also $\aeps/2$-forbidden.  In other words, if we let, for ease of notation,
  $\theta_* = \theta_{n_{\ksp}}$, $\theta^* = \theta_{n_{\ksp}+1}$ and
  $T_* = (n_{\ksp+1} -n_\ksp)\ve$ and we define the path
  \begin{align*}
    \fpath_{*}(t) = \theta_*+\frac{t}{T_*}(\theta^*-\theta_*),
  \end{align*}
  we know by construction that $\fpath_{*}$ is $\aeps/2$-forbidden.  Since $\fpath_*$ is
  an affine (hence smooth) path and it is $\aeps/2$-forbidden, we conclude, by
  definition of $\unifsmall$ that, for any $t\in[0,T_*]$:
  \begin{align*}
    \frac1{T_*}(\theta^*-\theta_*)\nin[\bar\omega^-(\fpath_*(t))-7\aeps/16,\bar\omega^+(\fpath_*(t))+7\aeps/16].
  \end{align*}
  Assume that $(\theta^*-\theta_*)/T_* > \bar\omega^+(\fpath_*(t))+7\aeps/16$ for any
  $t\in[0,T_*]$ (the other possibility can be treated similarly and it is left to the
  reader).  Let $\varrho_* = \unifsmall\min\{1,\aeps/(32\|\omega\|)\}$ and
  \begin{align*}
    Q = \{\fpath\in\cC^0[0,{T_*}]: \|\fpath-\fpath_*\| < 3\unifsmall,
    |\fpath(0)-\theta_*| < \varrho_*
    ,|\fpath({T_*})-\theta^*| < \varrho_*\}.
  \end{align*}
  We now claim that $\mathscr{I}^-_{\theta_*,\aeps/4}(\fpath) = \infty$ for any
  $\fpath\in Q$ (see~\cite[(6.12)]{DeL1} for the definition\footnote{ Once again, the
    crucial properties of $\mathscr{I}^-_{\theta,\Delta_*}$ are that
    $\mathscr{I}^-_{\theta,\Delta_*}(\fpath)$ is $\infty$ if $\fpath$ is not Lipschitz,
    $\fpath(0)\ne\theta$, or $\fpath'(t)\nin \bar B(\Omega(\fpath(t)),\Delta_*)$ for a positive measure
    set of times $t$.} of $\mathscr{I}^-_{\theta,\Delta_*}$).  If $\fpath$ is not
  Lipschitz, or $\fpath(0)\ne\theta_*$, we conclude by definition that
  $\mathscr{I}^-_{\theta_*,\aeps/4}(\fpath) = \infty$ .  So we can assume
  $\fpath$ to be Lipschitz; by definition of $Q$ and $\unifsmall$, we can ensure that
  $(\theta^*-\theta_*)/T_* > \bar\omega^+(\fpath(t))+5\aeps/16$ for any $t\in[0,T_*]$; hence
  \begin{align*}
    -\varrho_* < \int_0^{T_*}\fpath'(t)dt - (\theta^*-\theta_*) < \int_0^{T_*}\left[\fpath'(t)-\bar\omega^+(\fpath(t))-5\aeps/16\right]dt.
  \end{align*}
  Since $T_* > \TForb/2$ and $\varrho_* \le \TForb\aeps/32$, we conclude that
  $\fpath'(t) > \bar\omega^+(\fpath(t))+\aeps/4$ on a positive measure set, which
  implies
  \begin{align}\label{e_rateFunctionOnQ}
  \mathscr{I}^-_{\theta_*,\aeps/4}(\fpath) = \infty \text{ for any }\fpath\in Q.
  \end{align}

  Now let $x_* = x_{\ksp}$; by continuity of $F_\ve$ there exists a neighborhood
  $B_*\ni x_*$ so that $F^{T_*\vei}_\ve(B_*\times\{\theta_*\})\subset \{\theta\in B(\theta^*,\varrho_*/2)\}$.  Let
  $Q_* = \theta_\ve(B_*\times\{\theta_*\})$; observe that, since $\theta_\ve(p)$ is
  $\|\omega\|$-Lipschitz for any $p\in\bT^2$, for any $\fpath\in Q_*$ we have
  $\|\fpath-\fpath_*\| < 2\|\omega\|T_* < 2\unifsmall$; hence $Q_*\subset Q$.
  Therefore, choosing $\bP_{\omega,\ve} = \theta_{\ve*}\Leb_{\theta_*}$ (where
  $\Leb_{\theta_*}$ is the one-dimensional Lebesgue measure restricted to
  $\{\theta = \theta_*\}$), we have $\bP_{\omega,\ve}(Q_*) > 0$.  According
  to~\cite[Theorem~2.2]{DeL1}, we need to build a neighborhood $Q_*^+$ of $Q_*$ defined as:
  \begin{align*}
    Q_*^+ = \bigcup_{\fpath\in Q_*} B_{[0,T_*]}(\fpath,\ve^{1/6}\|\fpath-\bar\theta(\cdot,\theta_*)\|\nc0^{2/3})
  \end{align*}
  Trivially, $\|\fpath-\bar\theta(\cdot,\theta_*)\|\le 1$.  Hence, by choosing $\ve$
  sufficiently small, we can guarantee that $Q_*^+\subset Q$ (in fact, for any
  $\fpath\in Q_*^+$ we have $\|\fpath-\fpath_*\|\le 2\unifsmall+\ve^{1/6} < 3\unifsmall$
  (the conditions at the boundary are satisfied by identical arguments).  But then we
  reach a contradiction, since by~\eqref{e_rateFunctionOnQ} and~\cite[Theorem~2.2]{DeL1},
  we then obtain $\bP_{\omega,\aeps}(Q_*) = 0$, provided that $\ve$ is chosen
  sufficiently small.
\end{proof}
We conclude this subsection with a useful characterization
\begin{lem}\label{l_forbiddenEquivalence}
  Every \tpath{\thetaa}{\thetab} is $\aeps$-forbidden if and only if
  \begin{align}\label{e_maxminOmega}
    \min_{\theta\in[\thetaa,\thetab]}\bar\omega^+(\theta)&\le-\aeps &\text{and}&&
    \max_{\theta\in[\thetab,\thetaa]}\bar\omega^-(\theta)&\ge\aeps.
  \end{align}
\end{lem}
\begin{proof}
  First, let $\min_{\theta\in[\thetaa,\thetab]}\bar\omega^+(\theta) > -\aeps $; in
  particular there exists $0 < \aeps_* < \aeps$ and $\varrho_* > 0$ so that
  $\min_{\theta\in[\thetaa,\thetab]}\bar\omega^+(\theta)+\aeps_* > \varrho_*$. Let
  $\fpath_*$ be a\footnote{ The function $\bar\omega^+(\theta)$ is continuous, therefore a
    solution of the given differential equation exists (by Cauchy--Peano Theorem), but in
    general is not unique.} path solving the differential equation
  $\fpath'_*(s) = \bar\omega^+(\fpath_*(s))+\aeps_*$ with initial condition
  $\fpath_*(0) = \thetaa$; since $\fpath'_*(s) \ge \varrho_*$ there exists
  $T \le 1/\varrho_*$ so that $\fpath_*(T) = \thetab$, so $\fpath_*$ is a
  \tpath{\thetaa}{\thetab}.  Our construction moreover guarantees that
  $\bar\omega^+(\fpath(s))\le\fpath'_*(s) < \bar\omega^+(\fpath(s))+\aeps$, which implies
  that $h_*$ is not $\aeps$-forbidden and contradicts our assumptions.  If we assume
  $\max_{\theta\in[\thetab,\thetaa]}\bar\omega^-(\theta) < \aeps$, a similar argument also
  allows to construct a non $\aeps$-forbidden \tpath{\thetaa}{\thetab}.  This concludes
  the proof of the direct implication.

  Let us prove the reverse implication.  Let $\fpath$ be a \tpath{\thetaa}{\thetab} of
  length $T$; we want to prove that it is $\aeps$-forbidden, \ie that
  $\dfpath(s)\not\subset\Omega_\ve^+(h(s))$ for some $s\in[0,T]$. Without loss of
  generality we can assume that $\fpath(s)\nin\{\thetaa,\thetab\}$ if $s\in(0,T)$.  Then
  either $\fpath([0,T]) = [\thetaa,\thetab]$ or $\fpath([0,T]) = [\thetab,\thetaa]$.  Let
  us assume the first possibility; the second case can be completed by the reader
  following an analogous argument.  Let $\theta_*\in[\thetaa,\thetab]$ so that
  $\bar\omega^+(\theta_*)\le -\aeps$: assume first that $\theta_*\in(\thetaa,\thetab)$.
  Then by our assumptions for any $\varrho > 0$ sufficiently small there exist
  $0\le s_- < s_+ \le T$ so that $\fpath(s_-) = \theta_*-\varrho$,
  $\fpath(s_+) = \theta_*+\varrho$ and
  $\fpath([s_-,s_+]) = [\theta_*-\varrho,\theta_*+\varrho]$.  By Lebourg's Mean Value
  Theorem (see~\cite[Chapter 2, Theorem 2.4]{Clarke}) there exists $s\in[s_-,s_+]$ so that
  $\dfpath(s)\ni2\varrho/(s_+-s_-) > 0$. Moreover, by construction
  $|\fpath(s)-\theta_*| < \varrho$; since $\varrho$ was arbitrary and by compactness of
  the graph of $\dfpath(s)$ we conclude that there exists $s$ so that $h(s) = \theta_*$
  and $\dfpath(s)\cap[0,\infty)\ne\emptyset$; hence $\fpath$ is $\aeps$-forbidden since
  $\bar\omega^+(\theta_*) = -\aeps$.  If, on the other hand, $\theta_* = \thetaa$ (\resp
  $\theta_* = \thetab$), the same arguments carries on, considering the half ball
  $[\theta_*,\theta_*+\varrho]$ (\resp $[\theta_*-\varrho,\theta_*]$).
\end{proof}
\begin{rem}
  The same argument used in the above proof indeed shows that no \tpath\thetaa\thetab{} is
  $\aeps$-admissible if and only if
\begin{align}\label{e_maxminOmega2}
    \min_{\theta\in[\thetaa,\thetab]}\bar\omega^+(\theta)&\le\aeps &\text{and}&&
    \max_{\theta\in[\thetab,\thetaa]}\bar\omega^-(\theta)&\ge-\aeps.
  \end{align}
\end{rem}
\subsection{Local Central Limit Theorem}
In~\cite{DeL1} we also obtained a Local Central Limit Theorem (see~\cite[Theorem~2.7 and
Proposition~7.1]{DeL1}):
\begin{thm}\label{thm:lclt}
  For any $T>0$, there exists $\ve_0>0$ and $0 < \alpha_0 < 1$ so that the following
  holds.  For any compact interval $I\subset\bR$, real numbers $\shiftPar>0$,
  $\ve\in(0,\ve_0)$ and $t\in[\ve^{1/2000},T]$, any standard pair $\ell$ which intersects
  $\{\theta=\theta_0\}$, we have:
  \begin{equation} \label{e_indicatorlclt}
    \ve^{-1/2}\mu_\ell(\deviation\theta(t;\cdot)\in
    \ve I + \shiftPar\veh) = \Leb\, I\cdot
    \frac{e^{-\shiftPar^2/(2\Var_t^2)}}{\Var_t\sqrt{2\pi}}+\cO(\ve^{\alpha_0}).
  \end{equation}
  where the variance $\Var_t^2 = \Var_t^2(\theta_0)$ is given by
  \begin{equation}\label{e_variancelclt}
    \Var_t^2 = \int_0^t e^{2\int_s^t\bar\omega'(\bar\theta(r,\theta_0))\deh r}\bVar^2(\bar\theta(s,\theta_0))\deh s,
  \end{equation}
  and $\bVar^2(\theta)$ is given by the usual Green--Kubo formula
  \begin{align*}
    \bVar^2(\theta) &= \int_{\bT}\left[\ho^2(x,\theta) + 2 \sum_{m=1}^{\infty}
      \ho(f_\theta^m(x),\theta)\ho(x,\theta)\right]\rho_\theta(x)\deh x,
  \end{align*}
  where $\ho(x,\theta)=\omega(x,\theta)-\bar\omega(\theta)$.
\end{thm}
Observe that, $\bVar$ defined above is uniformly bounded away from $0$ by
Assumption~\ref{a_noCobo} and compactness of $\bT$; hence we conclude that
\begin{align}\label{e_varianceEstimate}
  \Const t \leq \Var^2_t\leq\Const\expo{\const t}t
\end{align}
\subsection{Averaged dynamics: description}\label{subsec:averagedDynDescription}
In this section we will describe the averaged dynamics of the variable $\theta_n$ and of
the auxiliary variable $\zeta_n$.  As we noted earlier, assumption~\ref{a_discreteZeros}
enables us to give the following simple description of the averaged dynamics $\bar\theta$:
let us start by fixing some terminology and notation.  As already briefly mentioned in
Section~\ref{sec:results}, we define the intervals:
\begin{align*}
  \fbas{k}&:=[\theta_{k,+},\theta_{k+1,+}]\ni\theta_{k,-}&
  \bbas{k}&:=[\theta_{k-1,-},\theta_{k,-}]\ni\theta_{k,+}.
\end{align*}
By~\eqref{e_averagedSlo}, any point in $\intr \fbas{k}$ (\resp $\intr\bbas{k}$) converges
in forward time (\resp backward time) to $\theta_{k,-}$ (\resp $\theta_{k,+}$): we thus
call $\intr \fbas{k}$ (\resp $\intr\bbas k$) the \emph{(forward) basin of attraction} of
$\theta_{k,-}$ (\resp \emph{backward basin of attraction} of $\theta_{k,+}$).  In
particular, any sufficiently small ball $B_k$ containing $\theta_{k,-}$ is
forward-invariant, that is:
\begin{align}\label{e_sink}
  \fa\ k\in\{1,\cdots,\nz\},\,t>0,\,\theta_0\in B_k&&%
  |\bar\theta(t,\theta_0)-\theta_{k,-}|&\leq |\theta_0-\theta_{k,-}|.
\end{align}
Let us now define the sets
\begin{subequations}\label{e_definitionW}
  \begin{align}
    \VH{k}&:=\{\theta\in \fbas{k}\st \bar\omega'(\theta)<\bar\omega'(\theta_{k,-})/2;\: \bclyapReg(\theta)<-3/4\}\\
    \VS{k}&:=\{\theta\in\bbas{k}\st \bar\omega'(\theta)>\bar\omega'(\theta_{k,+})/2\};
  \end{align}
\end{subequations}
observe that $\VH{k}\ne\emptyset$ by Assumption~\ref{a_almostTrivialp} and
Remark~\ref{rem:newpsi}.  For fixed $r_{-},r_+>0$ small, define
$\happy_k=B(\theta_{k,-},r_{-})$ and $\sad_{k}=B(\theta_{k,+}, r_+)$.  We prescribe $r_-$
(\resp $r_+$) to be small enough so that $\happy_k\subset\VH{k}$ (\resp
$\sad_k\subset\VS{k}$) for any $k$.  Define moreover
\begin{equation}\label{eq:HS-def}
  \uhappy=\bigcup_k\happy_k \hskip2cm
  \usad=\bigcup_k\sad_k.
\end{equation}
Finally, let us define $\hhappy_k=B(\theta_{k,-},3r_-/4)$ and
$\uhhappy=\bigcup_k\hhappy_k$.

By~\eqref{e_sink}, each of the sets $\happy_k$ is invariant for the averaged dynamics;
using~\eqref{e_averagedSlo} we thus conclude that $\zeta$ has an average negative drift on
$\happy_k$ whose rate is strictly less than $-1/2$.  This will imply that center vectors
are, \emph{on average}, contracted at an exponential rate, as long as the trajectory stays
in one of the $\happy_k$'s.  We will then use Large Deviation estimates (\ie
Theorem~\ref{t_largeDevz}), to obtain similar result for the real dynamics.

\subsection{Averaged dynamics: further properties}\label{ss_furtherProperties}
Let us now introduce a few additional notions which we will need, in particular, in the
case which~\ref{a_fluctuation} does not hold.
\begin{rem}
  The goal of this section is to define trapping sets for the dynamics in an abstract
  manner.  The reason for this is twofold: first it makes very clear where
  assumption~\ref{a_fluctuationGap} is actually used (see Lemma~\ref{l_fluctuationGap});
  second, it allows to prove all our results with little or no reference to the actual
  geometry of the trapping sets, which we think would be useful for further generalization
  to higher dimensional settings.
\end{rem}
For each $\theta\in\bT$ and $T > 0$ we define the sets
\begin{align*}
  \freach{\theta,T} &=\{\theta'\in\bT\st\exists\,\text{\tpath{\theta}{\theta'}{}
                      of length $\le T$ that is not $\aeps$-forbidden})\};\\
  \breach{\theta,T} &=\{\theta'\in\bT\st\exists\,\text{\eadmissiblep{\theta'}{\theta}{}
                      of length $\le T$})\}.
\end{align*}
Observe that $\freach{\theta,T}$ is given by \emph{end points} of paths \emph{starting
  from} $\theta$, while $\breach{\theta,T}$ is given by \emph{starting points} of paths
\emph{ending at} $\theta$.  We denote with
$\fbreach{\theta} = \bigcup_{T > 0}\fbreach{\theta,T}$.
\begin{lem}\label{l_trivialAdmissibleProperties} The following properties hold for any $\theta\in\bT$ and $\aeps > 0$:
  \begin{enumerate}
  \item $\fbreach{\theta}$ is connected (\ie an interval) for any $\theta\in\bT$;
  \item \label{pp_relationFbreach}
    $\theta\in\breach{\theta'}\Rightarrow\freach{\theta}\cap\breach{\theta'}\ne\emptyset\Rightarrow\theta'\in\freach{\theta}\;$;
  \item \label{pp_simpleInclusion} if $\theta'\in\fbreach{\theta}$, then $\fbreach{\theta'} \subset\fbreach{\theta}\;$;
  \item if $0\in\intr\Omega(\theta)$ then $\theta\in\fbreach{\theta}$, provided that
    $\aeps$ has been chosen small enough.
  \end{enumerate}
\end{lem}
The proof of the above properties readily follows from the definition and it is left to
the reader

\begin{lem}\label{l_fbreachOpen}
  $\fbreach{\theta,T}$ is an open set for any $T > 0$, $\aeps > 0$ and $\theta\in\bT$.
\end{lem}
\begin{proof}
  Let us prove the statement for $\freach{\theta}$; the proof for $\breach{\theta}$ is
  similar.  Assume that $\theta'\in\freach{\theta}$; then there exists a
  \tpath\theta{\theta'} $\fpath$ which is not $\aeps$-forbidden.  Recall that the graph
  $\{s,\dfpath(s)\}_{s\in[0,T]}$ is compact and that by definition of a $\aeps$-forbidden
  path, for any $s$ $\dfpath(s)\subset\Omp(s)$, which is an open set.  We
  conclude that if $|\varrho|$ is sufficiently small, the path
  $\fpath_\varrho(s) = \fpath(s)+\varrho\aeps s$ is also not $\aeps$-forbidden.  Then
  our statement holds since $\fpath_{\varrho}(T) = \theta'+\varrho\aeps T$.
\end{proof}
Notice that if $\aeps' < \aeps$ we have $\freachp\theta\subset\freach\theta$ and
$\breachp\theta\supset\breach\theta$; in particular we can define
$\freacha{\theta} = \bigcap_{\aeps > 0}\freach{\theta}$ and
$\breacha{\theta} = \bigcup_{\aeps > 0}\breach{\theta}$.

Further, by~\ref{a_noCobo}, we know that for any $\theta\in\bT$ we have
$\bar\omega(\theta)\in\intr\Omega(\theta)$ (see e.g.~\cite[Lemma~6.3]{DeL1}).  In
particular, there exists $\varrho > 0$ so that
$\Omega(\theta_{i,\pm})\supset(-\varrho,\varrho)$ for each $i = 1,\cdots,\nz$.  Since
$f_\theta$ is a smooth family of expanding maps, we conclude that, possibly by choosing a
smaller $\varrho$, we can find open neighborhoods $\Theta_{i,\pm}\ni\theta_{i,\pm}$ so
that for each $i = 1,\cdots,\nz$ And $\theta\in\Theta_{i,\pm}$, we have
$\Omega(\theta)\supset(-\varrho,\varrho)$.  We conclude (unsurprisingly) that
$\fbas{i}\subset\breach{\theta_{i,-}}$, provided that $\aeps$ is small enough.
Moreover, observe that by definition $\fa\theta,\theta'\in\Theta_{i,\pm}$ we have
$\fbreach{\theta} = \fbreach{\theta'}$.  Finally, observe that by possibly decreasing
$\varrho$ we can assume that if $\theta$ does not belong to any of the $\Theta_{i,\pm}$'s,
$|\bar\omega(\theta)| > \varrho$.  We conclude that for sufficiently small $\aeps$,
$\fbas{i}\subset\breach{\theta_{i,-},\TBas}$ where we define $\TBas := 2\varrho\invr$.

We are now in the position to prove genericity of Condition~\ref{a_fluctuationGap}:
\begin{lem}\label{l_genericFluctuationGap}%
  Condition~\ref{a_fluctuationGap} holds for a set that is $\cC^4$-open and dense in the set of
  $\omega$ which satisfy~\ref{a_noCobo},~\ref{a_discreteZeros} and~\ref{a_almostTrivial}.
\end{lem}
\begin{proof}
  Let $F_\ve$ be so that~\ref{a_noCobo},~\ref{a_discreteZeros} and~\ref{a_almostTrivial}
  are satisfied.  Assume that there exists an interval $J$ so that neither
  property~\ref{p_twoway} nor~\ref{p_oneway} holds.  To fix ideas let us assume that
  $\bar\omega(\theta)\ge0$ for any $\theta\in J$ (the other case can be treated
  analogously) \ie  $J = [\theta_{k,+},\theta_{k,-}]$ for some $k\in\{1,\cdots,\nz\}$.
  By assumption there exists $\theta_*\in J$ so that $0\in\partial\Omega(\theta_*)$.  By
  our previous discussion we know that $\theta_*\nin\Theta_{k,+}\cup\Theta_{k,-}$; thus
  let $\bar\omega_J(\theta)$ be a $C^\infty$ bump function which is $0$ on
  $\bT\setminus J$ and $1$ on $J\setminus\Theta_{k,+}\cup\Theta_{k,-}$.  Then for any
  $\eps > 0$, if we let
  $\omega(x,\theta)\mapsto\omega(x,\theta)+\eps\bar\omega_J(\theta)$, we obtain a
  dynamical system which satisfies property~\ref{p_oneway} in $J$, since
  $0\nin\cl\Omega(\theta_*)$.  Since $\omega_J$ is supported away from $\bT\setminus J$,
  the same construction can be applied independently to all other $J$'s for
  which~\ref{a_fluctuationGap} is not satisfied, which concludes the proof.  Observe in
  fact that our construction does not interfere with assumptions~\ref{a_noCobo} (since our
  perturbation is a function that is constant in $x$),~\ref{a_discreteZeros}
  and~\ref{a_almostTrivial} (since our perturbation is supported away from the set
  $\{\theta_{i,\pm}\}_{i = 1,\cdots,\nz}$).
\end{proof}
For $i = 1,\cdots,\nz$ define the $\aeps$-\emph{trapping set} of $\sink i$
\begin{align*}
  \trap_i = \{\theta\in\bT\st \freach{\theta}\subset\breach{\theta_{i,-}}\}.
\end{align*}
Observe that if $\aeps' < \aeps$, we have $\trapp_i\supset\trap_i$.  We say that a sink
$\sink i$ is \emph{recurrent} if $\trap_i\ne\emptyset$ for sufficiently small $\aeps$ and
\emph{transient} otherwise.

\begin{lem}[Properties of trapping sets]\label{l_propertiesTrappingSet} If $\aeps$ is
  sufficiently small, the following properties hold:%
  \begin{enumerate}
  \item \label{pp_trivialInclusion} There exists $\TTrap > 0$ so that $\trap_i\subset\breach{\sink i,\TTrap}$;
  \item \label{pp_invariance} if $\theta\in\trap_i$, then $\freach{\theta}\subset\trap_i$;
  \item \label{pp_disjoint} either $\trap_i\cap\trap_j = \emptyset$ or
    $\trap_i = \trap_j$;
  \item \label{pp_trapSink} $\sink i$ is recurrent if and only if
    $\trap_i\supset\Theta_{i,-}$;
  \item \label{pp_emptyTrap} $\sink i$ is transient if and only if $\exists$
    $j\in\{1,\cdots,\nz\}$ s.t.
    $\theta_{j,-}\in\freach{\theta_{i,-}}\setminus\breach{\theta_{i,-}}$;
  \end{enumerate}
\end{lem}
\begin{proof}
  Choose an arbitrary $\theta\in\trap_i$ and let $j\in\{1,\cdots,\nz\}$ so that
  $\fbas{j}\ni\theta$; if $\aeps$ is sufficiently small,
  $\theta\in\breach{\sink j,\TBas}$; then, by definition of $\trap_i$, we have
  $\sink j\in\breach{\sink i}$.  Observe that since there are only finitely many pairs of
  sinks, there exists $T' > 0$ so that for any $i,j\in\{1,\cdots,\nz\}$, either
  $\theta_{i,-}\in\breach{\theta_{j,-},T'}$ or $\theta_{i,-}\nin\breach{\theta_{j,-}}$.
  Hence, by Lemma~\ref{l_trivialAdmissibleProperties}\ref{pp_simpleInclusion} we have
  $\theta\in\breach{\sink i,\TTrap}$, where we set $\TTrap = \TBas+T'$; this
  proves~\ref{pp_trivialInclusion}.  On the other hand,~\ref{pp_invariance} follows from
  the fact that if $\theta'\in\freach{\theta}$, we have
  $\freach{\theta'}\subset\freach{\theta}\subset\breach{\theta_{i,-}}$.  Assume now
  $\theta\in\trap_i\cap\trap_j\ne\emptyset$: then by~\ref{pp_invariance} we have
  $\sink i\in\trap_j$ and $\sink j\in\trap_i$; by~\ref{pp_trivialInclusion}, we conclude
  that $\sink i\in\breach{\sink j}$ and $\sink j\in\breach{\sink i}$, which by definition
  imply respectively that $\trap_i\subset\trap_j$ and $\trap_j\subset\trap_i$,
  proving~\ref{pp_disjoint}.

  Now assume $\trap_i\ne\emptyset$ and let $\theta\in\trap_i$:
  by~\ref{pp_trivialInclusion} $\theta_{i,-}\in\freach{\theta}$ and thus,
  by~\ref{pp_invariance}, $\theta_{i,-}\in\trap_i$.  Then, by construction,
  $\fa\theta'\in\Theta_{i,-}$, $\freach{\theta'} = \freach{\theta_{i,-}}$, which in
  particular proves~\ref{pp_trapSink}.  In turn~\ref{pp_trapSink} implies that
  $\trap_i = \emptyset$ if and only if
  $\freach{\theta_{i,-}}\setminus\breach{\theta_{i,-}}\ne\emptyset$; let
  $\theta\in\freach{\theta_{i,-}}\setminus\breach{\theta_{i,-}}$.  Then
  $\theta\in\fbas{j}\subset\breach{\theta_{j,-}}$ for some $j\ne i$, which in turn
  implies~\ref{pp_emptyTrap}.
\end{proof}
In general it is possible for a system to have no recurrent sinks.  However, as the
following lemma shows, Condition~\ref{a_fluctuationGap} guarantees that this cannot
happen.
\begin{lem}\label{l_fluctuationGap}
  Assume that Condition~\ref{a_fluctuationGap} holds and $\aeps$ is sufficiently small: then
  \begin{enumerate}
  \item \label{pp_backAndForth} for any $i,j\in\{1,\cdots,\nz\}$ we have $\sink j\in\freach{\sink i}$ if and only if
    $\sink i\in\breach{\sink j}$.
  \item \label{pp_recurrentSink} for any $\theta\in\bT$ there exists a recurrent sink $\sink i$ so that
    $\theta\in\breach{\theta_{i,-},\TTrap}$.
  \end{enumerate}
\end{lem}
\begin{proof}
  By Lemma~\ref{l_trivialAdmissibleProperties}\ref{pp_relationFbreach} we have that if
  $\sink i\in\breach{\sink j}$ then $\sink j\in\freach{\sink i}$ thus we only need to
  prove the direct implication.  Assume by contradiction that one can find arbitrarily
  small $\aeps$ so that there exists a non $\aeps$-forbidden \tpath{\sink i}{\sink j} but
  yet all \tpath{\sink i}{\sink j}s are not $\aeps$-admissible.  By
  Lemma~\ref{l_forbiddenEquivalence} we gather that
  $\min_{\theta\in[\sink i,\sink j]}\bar\omega^+(\theta) > -\aeps$ or
  $\max_{\theta\in[\sink j,\sink i]}\bar\omega^-(\theta) < \aeps$.  Indeed by the same
  argument used in the proof of Lemma~\ref{l_forbiddenEquivalence} we can conclude that if
  every \tpath{\sink i}{\sink j} is not $\aeps$-admissible, then
  $\min_{\theta\in[\sink i,\sink j]}\bar\omega^+(\theta) \le \aeps$ and
  $\max_{\theta\in[\sink j,\sink i]}\bar\omega^-(\theta) \ge -\aeps$.  Since $\aeps$ is
  arbitrarily small, we conclude that
  $\min_{\theta\in[\sink i,\sink j]}\bar\omega^+(\theta) = 0$ or
  $\max_{\theta\in[\sink j,\sink i]}\bar\omega^-(\theta) = 0$.  In either case,
  assumption~\ref{a_fluctuationGap} is violated, which is a contradiction.  This
  proves~\ref{pp_backAndForth}.

  Let now $\theta\in\bT$ be arbitrary and assume by contradiction that every sink
  $\sink i$ so that $\breach{\sink i}\ni\theta$ is transient.  Let $i_0$ so that
  $\theta\in\fbas{i_0}$: in particular $\theta\in\breach{\sink {i_0},\TBas}$.  Since
  $\sink{i_0}$ is transient, by~\ref{l_propertiesTrappingSet}\ref{pp_emptyTrap} there
  exists another sink $\sink{i_1}\in\freach{\sink{i_0}}\setminus\breach{\sink{i_0}}$; by
  part~\ref{pp_backAndForth} and
  Lemma~\ref{l_trivialAdmissibleProperties}\ref{pp_simpleInclusion} we conclude that
  $\sink{i_0}\in\breach{\sink{i_1}}$, and therefore $\theta\in\breach{\sink{i_1}}$ by
  Lemma~\ref{l_trivialAdmissibleProperties}\ref{pp_simpleInclusion}.  Hence $\sink{i_1}$
  is also transient and we can again
  apply~\ref{l_propertiesTrappingSet}\ref{pp_emptyTrap}.  By repeating this construction,
  we obtain a sequence of sinks $\{\sink{i_k}\}$; since there are only finitely many
  sinks, eventually we have $\sink{i_k} = \sink{i_{l}}$ for some $l > k$, which in
  particular implies $\sink{i_{k+1}}\in\breach{\sink {i_k}}$, which
  contradicts~\ref{l_propertiesTrappingSet}\ref{pp_emptyTrap}.

  Hence, we conclude that $\sink{i_k}$ is recurrent; by definition we have
  $\sink{i_0}\in\breach{\sink{i_k},T'}$ (where $T'$ was defined in the proof of
  Lemma~\ref{l_propertiesTrappingSet}), which gives $\theta\in\breach{\sink{i_k,\TTrap}}$,
  as we needed to show.
\end{proof}%

\begin{rem}\label{r_fluctuationAssumption}%
  In this language,~\ref{a_fluctuation} states that $\breacha{\theta_{1,-}} = \bT$; by
  Lemma~\ref{l_fbreachOpen} and compactness of $\bT$ we conclude that if $\aeps$ is
  sufficiently small, $\breach{\theta_{1,-}} = \bT$, which gives $\trap_1 = \bT$.  In
  particular, by Lemma~\ref{l_propertiesTrappingSet}\ref{pp_disjoint}, there can be only
  one trapping set: for any $i = 1,\cdots,\nz$, either $\trap_i = \emptyset$ or
  $\trap_i = \trap_1$.

  On the other hand,~\ref{a_completeness} implies that $\fbreach{\theta} = \bT$ for any
  $\theta\in\bT$.
\end{rem}%
We will henceforth fix $\aeps > 0$ so small that all above results hold true.  Observe
that Lemma~\ref{l_fluctuationGap}\ref{pp_recurrentSink}, together with
Theorem~\ref{l_largeDevzLowerBound} immediately implies that for any standard pair $\ell$:
\begin{align}\label{e_eventuallyTrapped}
  \mu_\ell\bigg(\theta_\pint{T\vei }\nin\bigcup_{i = 1,\cdots,\nz}\trap_i\bigg) \le (1-\expo{-\const\vei})^\pint{T/\TTrap};
\end{align}
in other words: any point on a standard pair will eventually be trapped by some $\trap_i$.
Observe moreover that Theorem~\ref{t_forbiddenForRealDyn} implies that if $\sink i$ is
recurrent:
\begin{align}\label{e_protoInv}
  F_\ve^n (\bT\times \trap_i^+)\subset\bT\times\trap_i^- \text{ for any }n > \pint{\TForb\vei}.
\end{align}
where $\trap_i^+ = B(\trap_i,\varrho)$,
$\trap_i^- = \{\theta\st B(\theta,\varrho)\subset\trap_i\}$ and $\varrho$ and $\TForb$ are
the constants appearing in the statement of Theorem~\ref{t_forbiddenForRealDyn}.
\begin{cor}
  If $\sink i$ is recurrent, there exists an $F_\ve$-invariant
  $\invariant_i\subset\bT\times\trap_i$ which attracts every point in $\bT\times\trap_i$.
\end{cor}
\begin{proof}
  Let us define $\invariant_i^{(0)} = \bT\times\cl\trap_i$; by~\eqref{e_protoInv} we have
  \begin{align*}
    F_\ve^n\invariant_i^{(0)}\subset\intr\invariant_i^{(0)}  \text{ for any }n > \pint{T\vei}.
  \end{align*}
  Let us define $\invariant_i^{(s)} = F_\ve^{s\pint{T\vei}}\invariant_i^{(0)}$; then
  $\invariant_i^{(s)}\supset \invariant_i^{(s+1)}$; define
  $\invariant_i = \bigcap_{s\ge0}\invariant_i^{(s)}\ne\emptyset$.  By definition
  $\invariant_i$ is invariant for $F_\ve^{\pint{T\vei}}$. In fact, it is invariant by $F_\ve$: let $p\in \invariant_i$, then in particular
  $F_\ve p\in F_\ve \invariant_i^{(s)}$ for any $s > 0$; then by~\eqref{e_protoInv} we
  conclude that $F_\ve p\in \invariant_i^{(s)}$ for any $s\ge0$, that is
  $F_\ve p\in \invariant_i$.
\end{proof}
We will from now on assume $r_-$ so small that $\happy_k\subset\Theta_{k,-}$ for all
$k = 1,\cdots,\nz$.
\section{From averaged to true dynamics}\label{s_averaged2real}
In this section we show that the true dynamics behaves similarly to the averaged one with
very high probability. To this end we will follow the dynamics in rather long time steps.
This strategy will be employed also in the following sections, using possibly even longer
time steps.  Unfortunately, this requires a somewhat cumbersome notation.  To guide the
reader through the various future constructions we establish the following conventions:
\begin{nrem}\label{rem:notation}
  In the following we will introduce constants $T_{\sharp}$, where $\sharp$ stands for
  some generic subscript, to designate a \emph{macroscopic} time step \ie, a time step of
  order $1$ for the averaged motion.  To such times we will associate the corresponding
  \emph{microscopic} time steps for the map $F_\ve$ which we will consistently denote with
  $N_{\sharp}=\pint{T_{\sharp}\vei}$.

 We will also need to consider $\cO(\log\vei)$
  multiples of such macroscopic times: to this end we will introduce various constants
  denoted with $\cR_\sharp$ and we will let $\Tgen_\sharp=\pint{\cR_\sharp\ln \vei}$.

  In this way the reader will be able to immediately distinguish shorter time steps
  (e.g. $N_\sharp$) from the (logarithmically) longer ones (e.g.
  $\Tgen_\sharp N_{\sharp'}$).
\end{nrem}

\subsection{Escape and contraction}\label{ss_contraction}
Lemma~\ref{l_goodSet} below essentially states that if $\ell$ is supported on some set
$\{\theta\in\happy_k\}$, the $O(\vei)$-image of $\ell$ will escape from
$\{\theta\in\happy_k\}$ with exponentially small probability.  Additionally, we have some
bounds on the random variable $\zeta_n$, which controls the contraction in the center
direction. Recall, from the previous section, that $(x_n(p),\theta_n(p))=F^n_\ve(p)$ and
$\zeta_n(p)$, defined in~\eqref{e_definitionZ}, are considered to be random variables when
$p\in\bT^2$ is distributed on a standard pair $\ell$.  Given a standard pair $\ell$,
define $\avgtheta\ell=\mu_\ell(\theta_0)$; given a set $P\subset\bT$, we say that $\ell$
\emph{\ispinnedto{}} $P$ if $\avgtheta\ell\in P$.

Let us fix at this point $\TCn>0$ sufficiently large\footnote{ The choice for $\TCn$
  depends on a number of assumptions in Lemmata~\ref{l_goodSet},~\ref{l_deepPurple}
  and~\ref{l_escapeFromAlcatraz}; it is however important to observe that such
  requirements depend on $f$ and $\omega$ only.}  and recall that, following the
convention introduced in the above Notational Remark~\ref{rem:notation}, we let
$\NCn=\pint{\TCn\vei}$.
\begin{lem}\label{l_goodSet}
  If $\TCn$ is sufficiently large and $\ve$ sufficiently small, for any standard pair
  $\ell$ \pinnedto{} $\happy_k$ for some $k$,\footnote{ Recall that $\happy_k$ and
    $\hhappy_k$ are defined in Subsection~\ref{subsec:averagedDynDescription}.} we have
  \begin{align*}
    \mu_{\ell}(\theta_{\NCn}\in\hhappy_k, \zeta_\NCn\le -9\TCn/16) \ge 1-\expo{-\const\vei}.
  \end{align*}
\end{lem}
\begin{proof}
  Fix $TCn > 0$ to be specified later and define the set
  \begin{align*}
    \qh&=\{ p\in\supp\ell \st
    \sup_{t\in[0,\TCn]}|\Delta\theta(t,p)|< r_-/8,\,
    \sup_{t\in[0,\TCn]}|\Delta\zeta(t,p)|< \TCn/16 \}.
  \end{align*}
  We claim that we can choose $\ve$ sufficiently small so that for any $p\in\qh$, we have
  $\theta_{\NCn}(p)\in\hhappy_k$ and $\zeta_\NCn(p)\le -9\TCn/16$.  This would then prove
  our lemma, since Theorem~\ref{t_largeDevz} implies that
  $\mu_\ell(\qh)\ge 1-\exp(-\const\vei)$.

  To prove our claim, it is convenient to make our set $\happy_k$ fuzzy; for
  $\fuzz\in\fuzziness$, define $\happy_{k,\fuzz}=B(\theta_{k,-},\fuzz r_{-})$.  First, we
  assume $\TCn$ to be so large that, for any $k$,
  $\bar\theta(\TCn;\happy_{k})\subset\happy_{k,1/2}$.  Then, since
  $\avgtheta\ell\in\happy_k$, we can assume $\ve$ to be small enough to ensure that
  $\theta_\NCn(\qh) \subset \happy_{k,3/4} = \hhappy_k$, which proves the first part of
  our claim.

  Additionally, observe that $\theta_n(\qh)\subset\happy_{k,5/4}$ for any $0\le n\le\NCn$;
  by choosing a smaller $r_-$ if necessary we can assume that $\bclyapReg(\theta)<-5/8$
  for any $\theta\in\happy_{k,5/4}$.  Thus, for any $p\in\qh$,
  $\bclyapReg(\theta_n(p))<-5/8$: hence~\eqref{e_averagedSlo} and the definition of
  $\qh$ then imply that:
  \begin{equation*}
    \zeta_n(p)\le -\frac58n\ve+ \frac\TCn{16},
  \end{equation*}
  which concludes the proof of our claim.
\end{proof}
Lemma~\ref{l_goodSet} implies that as long as a standard pair \ispinnedto{} $\uhappy$, it
will stay there with large probability for an exponentially long time.
\begin{cor}\label{c_iterateGoodSet}
  Let $\ell$ be a standard pair \pinnedto{} $\happy_k$ for some $k$.  For any $l>0$:
  \begin{align*}
    \mu_\ell(\theta_{l\NCn}\in\hhappy_k) \ge (1-\expo{-\const\vei})^l.
  \end{align*}
\end{cor}
\begin{proof}
  The proof follows by induction on $l$: Lemma~\ref{l_goodSet} proves the base step $l=1$.
  Assume now that the statement holds for $l-1$.  Let
  $\alphaset'_\NCn=\alphaMap{\NCn}(\theta_\NCn\in\hhappy_k)$, where $\alphaMap{\NCn}$ was
  defined in Remark~\ref{rem:lofp}; by definition of standard curve, for any $\alpha\in\alphaset'_\NCn$ and
  $p,q\in\gap_\NCn(\alpha)$ (recall that $\gap_\NCn(\alpha) = \alphaMap{\NCn}^{-1}(\alpha)$), we have $|\theta_\NCn(q)-\theta_\NCn(p)|<\Const\ve$; this in turn
  implies that $\avgtheta{\fellf{\alpha}}\in\happy_k$.  Then, by the inductive assumption
  and Lemma~\ref{l_goodSet},
\[
\begin{split}
 \mu_\ell(\theta_{l\NCn}\in\hhappy_k) &\ge
  \mu_{\stdf_{\NCn}}(\theta_{(l-1)\NCn}\in\hhappy_k\cond\alphaset'_\NCn)\mu_\ell(\alphaset'_\NCn)\\
  &\geq  (1-\expo{\const\vei})^{l-1} \mu_{\ell}(\theta_{\NCn}\in\hhappy_k)\geq  (1-\expo{\const\vei})^l.\qedhere
\end{split}
\]
\end{proof}
The above corollary allows to obtain sharper information on the $\theta$ variable by
means of the following lemma.
\begin{lem}\label{l_deepPurple}
  If $\TCn$ is sufficiently large, there exists $C,\Rdp>0$ so that if $\ve$ is
  sufficiently small, for any standard pair $\ell$ \pinnedto{} $\happy_k$ and for any
  $\Rgen\ge\Rdp$, letting  $\Tgen=\pint{\Rgen\log\vei}$:
  \begin{align}\label{e_deepPurple}
    \mu_{\ell}(\avgtheta{\fellf{\cdot}_{\Tgen\NCn}}\nin B(\theta_{k,-},C\sqrt\ve))<\frac13,
  \end{align}
  where, recall, we consider $\fellf{\cdot}_{\Tgen\NCn}$ to be a random standard pair according
  to Remark~\ref{rem:lofp}.
\end{lem}
\begin{proof}
  Define the function $\lyapFunc:\bT\to\bRp$:
  \begin{align*}
    \lyapFunc(\theta)= \min\{|\theta-\theta_{k,-}|,r_-\}.
  \end{align*}
  We will use $\lyapFunc$ as a sort of Lyapunov function, namely, we claim that if $\ell$
  \ispinnedto{} $\happy_k$, $\lyapFunc$ satisfies the following geometric drift condition:
  \begin{align}\label{e_geometricDrift}
    \mu_\ell(\lyapFunc\circ \theta_{\NCn}) \le \frac12 \mu_\ell(\lyapFunc) + C_{\TCn}\sqrt\ve,
  \end{align}
  where, in the above expression, we regard $\theta_{n}$ as a random variable on $\ell$
  and to simplify the exposition we will abuse notation and write $\mu_\ell(\lyapFunc)$
  instead of $\mu_{\ell}(\lyapFunc\circ\theta_0)$.  Moreover $C_{\TCn}$ is a constant
  which depends on $\TCn$ only.  In fact, first observe that, by
  Theorem~\ref{t_largeDevz}, since $\|\lyapFunc\|_\infty\le 1$:
    \begin{align*}
      \mu_{\ell}(\lyapFunc\circ \theta_{\NCn}) =%
      \mu_{\ell}((\lyapFunc\cdot \Id_{B(\bar\theta(\TCn;\avgtheta\ell),\ve^{1/2-\alpha_0/2})})\circ
      \theta_{\NCn})+\cO(\expo{-C_{\TCn}\ve^{-\alpha_0}}).
    \end{align*}
   Now let us subdivide the interval ${B(\bar\theta(\TCn;\avgtheta\ell),\ve^{1/2-\alpha_0/2})}$ in
    $\cO(\ve^{-1/2-\alpha_0/2})$ intervals $I_j$ of size $\cO(\ve)$,
    so that we can write
    \begin{align*}
      \mu_{\ell}((\lyapFunc\cdot
      \Id_{B(\bar\theta(\TCn;\avgtheta\ell),\ve^{1/2-\alpha_0/2})})\circ
      \theta_{\NCn}) = \sum_{j}\mu_{\ell}((\lyapFunc\cdot
      \Id_{I_j})\circ \theta_{\NCn}).
    \end{align*}
    Using Theorem~\ref{thm:lclt} and the fact that $\lyapFunc$ is Lipschitz yields, on
    each interval $I_j$,
\begin{align*}
  \mu_\ell((\lyapFunc\cdot\Id_{I_j})\circ \theta_{\NCn}) =&
  \int_{I_j}\frac{e^{-(y-\bar\theta(\TCn;\avgtheta{\ell}))^2/(2\ve\Var_{\TCn}^2)}}{\Var_{\TCn}\sqrt{2\pi \ve}}\lyapFunc(y)\deh
  y \\
  &+ \cO\left(\ve^{\alpha_0-\frac12}\int_{I_j} \lyapFunc(y)dy+\ve^{\frac32}\right).
\end{align*}
Hence, summing over all intervals and using standard Large Deviations bounds for the
Normal Distribution, we obtain
\begin{align}\label{e_estimateLyapFunc}
  \mu_\ell(\lyapFunc\circ \theta_{\NCn}) &= 
  \int\frac{e^{-(y-\bar\theta(\TCn;\avgtheta{\ell}))^2/(2\ve\Var_{\TCn}^2)}}{\Var_{\TCn}\sqrt{2\pi \ve}}\lyapFunc(y)\deh
  y + \cO(\ve^{1-\alpha_0/2})\notag\\ &\phantom{=}
  +\cO\left(\ve^{\alpha_0-1/2}\int_{\bar\theta(\TCn;\avgtheta\ell)-\ve^{1/2-\alpha_0/2}}^{\bar\theta(\TCn;\avgtheta\ell)+\ve^{1/2-\alpha_0/2}}\lyapFunc(y)\deh y\right)\notag\\
  \intertext{and, again, since $\lyapFunc$ is Lipschitz:}
  \mu_\ell(\lyapFunc\circ \theta_{\NCn}) &= 
  (1+\cO(\ve^{\alpha_0/2}))\lyapFunc(\bar\theta(\TCn;\avgtheta{\ell}))+
  (\Var_{\TCn}+1)\cO(\sqrt\ve).
  \end{align}
  Then, let us assume that $\TCn$ has been chosen sufficiently large that for any
  $\theta\in\happy_k$:
  \begin{align*}
    |\bar\theta(\TCn;\theta)-\theta_{k,-}| \le \frac13|\theta-\theta_{k,-}|,
  \end{align*}
  so that in particular we have $\lyapFunc(\bar\theta(\TCn;\theta)) <\lyapFunc(\theta)/3$.
  Since
  $\mu_\ell(\lyapFunc) = \lyapFunc(\avgtheta\ell) +\cO(\ve)$,~\eqref{e_estimateLyapFunc}
  reads:
  \begin{align*}
    \mu_\ell(\lyapFunc\circ\theta_{\NCn}) <
    \frac{1+\ve^{\alpha_0/2}}3\mu_\ell(\lyapFunc)+ C_{\TCn}\sqrt\ve
  \end{align*}
  which, gives~\eqref{e_geometricDrift}, provided $\ve$ is chosen small enough.

  Observe now that by Corollary~\ref{c_iterateGoodSet},
  $\mu_\ell(\avgtheta{\fellf{\cdot}_{l\NCn}}\nin\happy_k) =
  \fm_{l\NCn}(\avgtheta\fell\nin\happy_k)< l\expo{-\const\vei}$.
  Using~\eqref{e_geometricDrift} we can then conclude that there exists $\Rdp > 0$
  sufficiently large so that for any $\Rgen \ge\Rdp$ and $\Tgen = \pint{\Rgen\log\vei}$:
  \begin{align*}
    \mu_{\stdf_{\Tgen\NCn}}(\lyapFunc) &= %
    \mu_{\stdf_{(\Tgen-1)\NCn}}(\lyapFunc\circ \theta_\NCn)\\
    &=\mu_{\stdf_{(\Tgen-1)\NCn}}(\lyapFunc\circ
    \theta_\NCn|\avgtheta{\fell}\in\happy_k)\nu_{(\Tgen-1)\NCn}(\avgtheta{\fell}\in\happy_k) + \Const\Tgen\expo{-\const\vei} \\
    &\le \frac12\mu_{\stdf_{(\Tgen-1)\NCn}}(\lyapFunc|\avgtheta{\fell}\in\happy_k) +
      C_{\TCn}\sqrt\ve + \Const\Tgen\expo{-\const\vei} \\
                                       &\le \frac12 \mu_{\stdf_{(\Tgen-1)\NCn}}(\lyapFunc) + C_{\TCn}\sqrt\ve+
                                         \Const\Tgen\expo{-\const\vei}.\intertext{Iterating the above estimates $\Tgen$ times yields}
                                         \mu_{\stdf_{\Tgen\NCn}}(\lyapFunc)&\le C_{\TCn}\sqrt\ve(1+\mu_\ell(\lyapFunc)) + \Const \Tgen^2\expo{-\const \vei} \le C_{\TCn}\sqrt\ve.
  \end{align*}
  Markov Inequality thus implies~\eqref{e_deepPurple} \eg choosing $C=3C_{\TCn}$.
\end{proof}
\subsection{Attractors: return to $\uhappy$}
We now proceed to describe the dynamics outside $\uhappy$ (for its definition,
see~\eqref{eq:HS-def}); indeed we will not need very refined results in this region;
essentially we will only prove that the dynamics comes to $\uhappy$ with very large
probability in time $\cO(\log\vei)$.
\begin{lem} \label{l_escapeFromAlcatraz}%
  If $\TCn$ is sufficiently large, there exists $\beta>0$ and $\RSl>1$ such that if $\ve$
  is sufficiently small, for any standard pair $\ell$ we have:
  \begin{align}\label{e_escapeFromAlcatraz}
    \mu_\ell(\theta_{\TSl\NCn}\nin\uhhappy)<\ve^\beta.
  \end{align}
  where, according to Notational Remark~\ref{rem:notation},  $\TSl=\pint{\RSl\log\vei}$.
\end{lem}
\begin{proof}
  Fix $\RSl>1$ sufficiently large to be specified later.  We will prove the lemma in two
  steps; first let us show the following auxiliary result:
  \begin{sublem} \label{l_homeRun}%
    There exists $\Thr > 0$, and $c=c(\RSl)$ so that if $\ve$ is sufficiently small, for
    any $\pint{\Thr\vei} = \Mhr\le N\le\TSl\NCn$ and standard pair $\ell$ \notpinnedto{}
    $\usad$:
    \begin{align*}
      \mu_\ell(\theta_{N}\nin\uhhappy) < \expo{-c\vei}.
    \end{align*}
  \end{sublem}
  \begin{proof}
    Our stipulations on $\bar\omega$ guarantee that, if $\theta\not\in\usad$, there exists
    $\Thr>0$ such that if $T>\Thr$, then $\bar\theta(T;\theta)\in\uhhappy$.  According to
    Notational Remark~\ref{rem:notation}, let
    $\Mhr=\pint{\Thr\vei}$ and write $N=l\NCn+M$ where $\Mhr\le M <\Mhr+\NCn$.  By
    Large Deviations arguments analogous to the ones used in the proof of
    Lemma~\ref{l_goodSet} we conclude that
    \begin{align*}
      \mu_\ell(\theta_M\nin\uhhappy) < \expo{-\bar c\vei}.
    \end{align*}
    Let $\stdf_M$ be a standard $M$-pushforward of $\ell$; observe that if
    $\theta_M(p)\in\uhhappy$, necessarily $\fellf{p}_M$ \ispinnedto{} $ \uhappy$ (recall
    that $\fellf{p} _M$ was defined in Remark~\ref{rem:lofp}); we thus conclude that
    $\mu_\ell(\avgtheta{\fell_M}\nin\uhappy)<\expo{-\bar c\vei}$.  Hence, since
    \begin{align*}
      \mu_\ell(\theta_N\nin\uhhappy) &\le \mu_{\ell}(\theta_N\nin\uhhappy\cond\avgtheta{\fell_M}\in\uhappy)
                                       + \expo{-\bar c\vei} \\&\le
                                                                \mu_{\stdf_M}(\theta_{N-M}\nin\uhhappy\cond\avgtheta{\fell}\in\uhappy) +\expo{-\bar c\vei};
    \end{align*}
    our result then follows by applying Corollary~\ref{c_iterateGoodSet}.
  \end{proof}
  Observe that Sub-Lemma~\ref{l_homeRun} proves Lemma~\ref{l_escapeFromAlcatraz} in the
  particular case $\avgtheta\ell\not\in\usad$.  If this is not the case, we have $\ell$
  \ispinnedto $\usad=\bigcup_k\sad_k$: for ease of exposition, let $k$ be fixed so that
  $\avgtheta\ell\in\sad_k$ and let us drop $k$ from our notations; that is, let
  $\theta_+=\theta_{k,+}$, $\sad=\sad_k$.  In order to prove our lemma we claim that it
  suffices to show that for some $\beta'>0$:
  \begin{align}\label{e_simplerEscape}
    \mu_\ell(\theta_N\in\hat\sad \text{ for all }0\le N\le \TSl\NCn-\Mhr)<\ve^{\beta'}
  \end{align}
  where $\hat\sad$ is an $\cO(\ve^{1/4})$-neighborhood of $S$.  In fact,
  by~\eqref{e_simplerEscape}, with probability $1-\ve^{\beta'}$, there exists some
  $0\le N \le \TSl\NCn-\Mhr$ so that $\avgtheta{\ell_N(p)}\nin\usad$; applying
  Sub-Lemma~\ref{l_homeRun} to such standard pair then guarantees that the
  $(\TSl\NCn-N)$-iterate of $\ell_N$ will be supported in $\uhhappy$ with probability
  $1-\expo{-\bar c\vei}$, which in turn implies~\eqref{e_escapeFromAlcatraz} and
  concludes our proof.

  We are thus left to prove~\eqref{e_simplerEscape}: fix $\soRad>0$ to be specified later
  and define the function $\lyapFunc:\bT\to\bR$:
  \begin{align*}
    \lyapFunc(\theta)=
    \begin{cases}
      0&\text{ if }\theta\nin\hat\sad\\
      \frac1{\soRad\sqrt\ve}&\text{ if } |\theta-\theta_+|\le \soRad\sqrt\ve\\
      \frac1{|\theta-\theta_+|} &\text{ otherwise}.
    \end{cases}
  \end{align*}
  We claim there exists $\vt\in(0,1)$ so that for any $0 < k < \TSl$:
  \begin{align}\label{e_escapeLyapunov}
    \mu_{\ell}(\lyapFunc\circ \theta_{k\NCn})\le \Const\vt^{k}\mu_\ell(\lyapFunc)+\Const \expo{-\const\vei},
  \end{align}
  Then, if $\bar \cR\sim\log\vei$ is so that $\vt^{\bar \cR}=\cO(\ve^{\beta'})$,
  by~\eqref{e_escapeLyapunov} we gather
  $\mu_{\ell}(\lyapFunc\circ\theta_{\bar \cR\NCn})\le\Const\ve^{\beta'}$.  Markov
  Inequality then implies that
  \begin{align*}
    \mu_{\ell}(\theta_{\bar \cR\NCn}\in\hat\sad)=\mu_\ell\left(\lyapFunc\circ\theta_{\bar
    \cR\NCn}>1/(2r_+)\right)<\Const\ve^{\beta'},
  \end{align*}
  provided that $\ve$ is sufficiently small, which in turn
  implies~\eqref{e_simplerEscape}, choosing $\RSl$ sufficiently large (\eg,
  $\RSl\ge 2\bar \cR$ is certainly enough).

  Thus, to conclude our proof, it suffices to prove~\eqref{e_escapeLyapunov}.  We have
  three cases:
  \begin{enumerate}
  \item \label{i_ellnotsad} $\avgtheta\ell\nin \sad$
  \item \label{i_ellnotdesperate} $\avgtheta\ell\in \sad$,
    $|\avgtheta\ell-\theta_+|\ge\soRad\sqrt\ve$
  \item \label{i_elldesperate} $|\avgtheta\ell-\theta_+| < \soRad\sqrt\ve$.
  \end{enumerate}
\begin{subequations}\label{e_escapeLyapunovBounds}
  Let us first consider case~\ref{i_ellnotsad}: in this case, for any $\Mhr < N <\TSl\NCn$, we
  claim that:
    \begin{align}
      \mu_{\ell}(\lyapFunc\circ\theta_{N})\le \expo{-\const\vei}.
    \end{align}
    The above estimates immediately follows by Sub-Lemma~\ref{l_homeRun} if
    $\avgtheta\ell\nin\usad$, and by a similar large deviations argument
    otherwise.\footnote{ In fact $\hat\sad$ is a repelling set for the averaged dynamics,
      hence if $\avgtheta\ell\in\usad\setminus\sad$, the averaged dynamics will certainly
      keep $\theta$ away from $\hat\sad$. }

    Let us now consider case~\ref{i_ellnotdesperate}: by definition of
    $\VSshort$ (see~\eqref{e_definitionW}), we know that if $\theta_0\in\VSshort$,
    $|\bar\omega(\theta_0)|\ge\bar\omega'(\theta_+)|\theta_0-\theta_+|/2$.  We assume
    $\TCn$ so large that for any $\theta_0\in\sad$, we have either
    $\bar\theta(\TCn;\theta_{0})\nin\hat\sad$ or
    $|\bar\theta(\TCn;\theta_{0})-\theta_{+}|\geq 2|\theta_0-\theta_{+}|$.  Hence, we
    have
    \begin{align*}
      \mu_{\ell}\left(|\theta_{\NCn}-\theta_+|\le\frac32|\avgtheta\ell-\theta_+|\right)
      \le
      \mu_{\ell}\left(|\theta_{\NCn}-\bar\theta(\TCn;\avgtheta\ell)|\ge\frac12|\avgtheta\ell-\theta_+|\right).
    \end{align*}
    Assuming $\soRad$ sufficiently large, we can apply Theorem~\ref{t_largeDevz}
    and gather that
    \begin{align*}
      \mu_{\ell}\left(|\theta_{\NCn}-\bar\theta(\TCn;\avgtheta\ell)|\ge\frac12|\avgtheta\ell-\theta_+|\right)
      \le
      \expo{-C_{\TCn}|\avgtheta\ell-\theta_+|^2\vei}.
    \end{align*}
    Consequently:
    \begin{align*}
      \mu_\ell(\lyapFunc\circ\theta_{\NCn}) \le \frac23\mu_\ell(\lyapFunc) +
      \frac1{\soRad\sqrt\ve}\expo{-C_{\TCn}|\avgtheta\ell-\theta_+|^2\vei}.
    \end{align*}
    Choosing $\soRad$ to be so large\footnote{ Observe that the choice of
      $\soRad$ depends on $C_{\TCn}$ and thus on $\TCn$.} that
    \begin{align*}
      \frac1{\soRad\sqrt\ve}\expo{-C_{\TCn}|\avgtheta\ell-\theta_+|^2\vei} \le
      \frac1{12|\avgtheta\ell-\theta_+|},
    \end{align*}
    we obtain
    \begin{align}
      \mu_\ell(\lyapFunc\circ\theta_{\NCn}) \le \frac56\mu_\ell(\lyapFunc).
    \end{align}
    Finally, we need to consider case~\ref{i_elldesperate}:
    first of all by definition of $\lyapFunc$ we can immediately conclude that for any
    $n\ge0$:
    \begin{align}
      \mu_\ell(\lyapFunc\circ\theta_n)\le \frac76\mu_\ell(\lyapFunc).
    \end{align}
    Moreover, using once again Theorem~\ref{thm:lclt} and the lower bound
    in~\eqref{e_varianceEstimate} we can choose $\Tesc{}$ to be so large that, for any
    $\shiftPar\in\bR$,
    $\mu_\ell(\Delta\theta({\Tesc{}};\cdot)\ve^{-1/2}\in B(\shiftPar,2\soRad))<1/3$.  We
    thus conclude that for any standard pair $\ell$ so that
    $|\avgtheta\ell-\theta_+|<\soRad\sqrt\ve$:
    \begin{align}
      \mu_\ell(\lyapFunc\circ\theta_{\Nesc{}}) \le \frac13 \cdot \frac1{\soRad\sqrt\ve} +
      \frac1{2\soRad\sqrt\ve} \le \frac56\mu_\ell(\lyapFunc).
    \end{align}
  \end{subequations}
  Observe that by possibly increasing $\Tesc{}$, we can guarantee $\Nesc{}=p\NCn$ for
  some $p\in\bN$.  Collecting bounds~\eqref{e_escapeLyapunovBounds} we can therefore
  conclude that, for any $k\ge0$ and for any sequence $\stdf_n$ of pushforwards of
  $\ell$: {\newcommand{\altvt}{\vt_*}
    \begin{align*}
      \mu_{\ell}(\lyapFunc\circ\theta_{kp\NCn}) &=
                                                  \mu_{\stdf_{(k-1)p\NCn}}(\lyapFunc\circ\theta_{p\NCn}) \le
                                                  \altvt\mu_{\stdf_{(k-1)p\NCn}}(\lyapFunc)+\expo{-\const\vei}\\
                                                &\le \altvt^k\mu_\ell(\lyapFunc)+\Const\expo{-\const\vei},
    \end{align*}
    for some $\altvt\in(0,1)$ (e.g., $\altvt=(5/6)(7/6) = 35/36$ works).  The above
    inequality immediately implies~\eqref{e_escapeLyapunov}, choosing $\vt=\altvt^{1/p}$.}
\end{proof}
The above results show quantitatively that the dynamics tends to concentrate around the
sinks of the averaged dynamics, where most of the center vectors are contracted at an
exponential rate.  This fact will be the crucial ingredient in our arguments.
\section{Coupling: basic facts and definitions}\label{sec:coupling-def}
We are now ready to start the discussion of statistical properties of the map $F_\ve$.  As
anticipated, we will classify its SRB measures (in the sense of
Remark~\ref{r_endomorphismSRB}) and study their statistical properties using the framework
of standard pairs.

The main advantage of using standard pairs is that we are in a sense able to separate the
deterministic behavior (at the level of standard pairs) from the stochastic behavior
(regarding standard pairs as ``atoms'').

Let us start by recalling the main ideas: the crucial observation (due to Dolgopyat, see
e.g.~\cite{DimaPH,DimaSRB,DimaAveraging}) is that SRB measures can be written as weak
limits of measures that admit a standard decomposition (see also Lemma~\ref{l_weakLimit}).

Then, given any two such measures let $\stdfa$ and $\stdfb$ be the associated standard
families.  If we can prove that the measures induced by $\pFve^n\stdfa$ and
$\pFve^n\stdfb$ are asymptotically equal; then this would imply uniqueness of the SRB
measure for the system.  If, additionally, we can control the rate at which
$\pFve^n\stdfa$ and $\pFve^n\stdfb$ are approaching, then we are able to retrieve precise
information about the rate of mixing of sufficiently smooth observables.

This project can be carried out provided that there exists only one trapping set for the
dynamics (that is guaranteed if~\ref{a_fluctuation} holds).  Otherwise,
if~\ref{a_fluctuationGap} holds, it is still possible to prove uniqueness of SRB measure
supported on each set $\{\theta\in\trap_i\}$, and obtain information about the rate of
mixing of sufficiently smooth observables supported in each of these sets.

The strategy which is most commonly employed in order to study the above mentioned
asymptotic equality, and estimate the speed of mixing, is the \emph{coupling
  technique}.\footnote{ Coupling has been long used in abstract Ergodic Theory under the
  name of \emph{joining}, but it has been re-introduced in the the study of the statistical
  properties of smooth systems (smooth Ergodic Theory) by Lai-Sang Young~\cite{Young99},
 borrowing it from the theory of Markov chains.  The version we are going to present here
  has been developed by Dmitry Dolgopyat in the standard pair framework.}
\subsection{Basic coupling definitions}\label{ss_couplingDefinitions}
Let us start by recalling some useful definitions: a \emph{coupling} of two probability
measures $\mu_0$ and $\mu_1$ (on the measurable space $\bT^2$) is given by a probability
measure $\couplingmu$ on the product space $\bT^2\times\bT^2$ whose marginals on the first
and second factor coincide with $\mu_0$ and $\mu_1$ respectively.  We denote with
$\Gamma(\mu_0,\mu_1)$ the set of couplings of the two probability measures.  The
\emph{Wasserstein distance} of two probability measures $\mu_0$, $\mu_1$ is defined as
\begin{equation}\label{eq:Wasserstein}
  \wDist(\mu_0,\mu_1)=\inf_{\couplingmu\in\Gamma(\mu_0,\mu_1)}\bE_{\couplingmu}(\vdist).
\end{equation}
Note that the definition of $\wDist$ depends on the choice of the distance.  In this
paper, we find convenient to employ the \emph{vertical distance} $\vdist$ given by:
\begin{equation}\label{eq:Wasserstein-dist}
\vdist((x,\theta),(x',\theta'))=\begin{cases} 1 &\textrm{ if } x\neq x'\\
|\theta-\theta'|&\textrm{ otherwise.}
\end{cases}
\end{equation}
Observe that $\vdist$ controls the standard Euclidean distance (which we denote by
$\dist$), \ie $\dist(p,p')\le\sqrt 2\, \vdist(p,p')$ for any $p$, $p'\in\bT^2$.

When two measures admit a standard decomposition, it is convenient to describe their
couplings in terms of standard families.  Let us start by considering standard pairs: by a
$(\spc1',\spc2')$-\emph{standard couple} $\fellC=(\fellCa,\fellCb)$ we mean a couple of
$(\spc1',\spc2')$-standard pairs $\fellCa$ and $\fellCb$ and a measure $\boldsymbol \mu$
on $\bT^2$ such that the marginals are the measures are $\mu_{\fellCa}$ and
$\mu_{\fellCb}$ respectively.\footnote{ We do not include explicitly $\boldsymbol \mu$ in
  the notation to make it more readable and as it does not create confusion.}  Let now
$\stdfa$ and $\stdfb$ be two (pre)standard families; a $(\spc1',\spc2')$-\emph{standard
  coupling} of $\stdfa$ and $\stdfb$ is a random element
$\stdfC=\SFF{\fellC}{(\alphaset,\fm)}$ where
$\fellC:\alphaset\to\stdpSet{\spc1',\spc2'}\times\stdpSet{\spc1',\spc2'}$ is a random
couple $\fellC=(\fellCa,\fellCb)$ of $(\spc1',\spc2')$-standard pairs so that
$\stdfCi=\SFF{\fellCi}{(\alphaset,\fm)}\sim\stdfi$.

Given two (pre)standard families $\stdfa$ and $\stdfb$, there is no canonical choice of a
standard coupling.  A simple, but important example is given by the \emph{independent}
coupling: we let $(\alphaset,\fm)=\bigtimes_{i=\privatea,\privateb}(\alphaseti,\fmi)$ and
we define
$\fellCf{(\alpha_\privatea,\alpha_\privateb)} =
(\fellaf{\alpha_\privatea},\fellbf{\alpha_\privateb})=(\fellCfa{\alpha_\privatea,
  \alpha_\privateb},\fellCfb{\alpha_\privatea, \alpha_\privateb})$
equipped with the product measure.  Observe that, in this case, the families $\stdfCa$ and
$\stdfCb$ are independent random variables. As for standard families, we will declare two
coupling equivalent if their marginal measures $\mu_{\stdf^i}$ are the same, and we will
designate the equivalence classes by $ \eqc{\stdfC}$.

Let $\stdfC=\SFF{\fellC}{(\alphaset,\fm)}$; given $\alphaset'\subset\alphaset$ we
define the \emph{subcoupling conditioned on $\alphaset'$} to be
$\stdfC\cond{\alphaset'}=\SFF{\fellC\cond{\alphaset'}}{(\alphaset',\fm')}$, where, once again,
$\nu'(E)=\nu(E|\alphaset')$.

We say that $\ellC=(\ellCa,\ellCb)$ is a \emph{matched \couple{}} (\resp
$\Delta$-\emph{matched \couple{}}) if $\ella$ and $\ellb$ are stacked (\resp
$\Delta$-stacked), have equal densities (see Section~\ref{subsec:Standardpairs} for the
definition of ``stacked") and are coupled along the vertical direction, that is:
$\boldsymbol\mu(g)=\int g(\bG^0(x), \bG^1(x))\rho(x) \deh x$ where $\rho$ is their common
density.  We will call this the {\em canonical coupling} for matched pairs.  Note that
  \begin{equation}\label{eq:wasserstein}
    \wDist(\mu_{\ellCa},\mu_{\ellCb}) \le \Delta.
  \end{equation}

Recall that, given the (pre)standard families $\{\stdf^i\}$ it is defined for all $n$ the
pushforward $\eqc{F_\ve^n\stdf^i}=\eqc{\stdf^i_n}$. Then, given any standard coupling
$\stdfC$ of $\stdf^0,\stdf^1$, we define $ \eqc{F_\ve^n\stdfC}$ as the equivalence class
of the product coupling of $\eqc{\stdf^i_n}$.  A sequence of (pre)standard couplings
$(\stdfC_n)_n$ is said to be a \emph{pushforward} the (pre)standard coupling $\stdfC$ if
$\stdfC_n\in\eqc{F_\ve^n\stdfC}$. Note that such a definition is much less stringent than
the notion of pushforward for a standard family, it is then not surprising that we will
need a more stringent definition for standard couplings as well.

If $\stdfCa_n$ is a pushforward of a standard family $\stdfCa$ and, for each
$\alpha\in \alphaset_n$, $\fellCf{\alpha}_n$ is a $\Delta$-matched pair, for some
$\Delta$, then we say that we have a {\em matched pushforward}.\footnote{ Note that, with
  the above definition, $\stdfCb$ will not be, in general, a standard family. Yet, it will
  be $n$-prestandard and that is all is needed in the following.}
\begin{rem}\label{rem:lCofp} Note that, if we have a matched pushforward $\stdfC_n$ of
  $\ellC$, then Remark~\ref{rem:lofp} applies to the families $\fellCa_n$ and, by the
  matching property, to $\fellCb_n $ as well. It makes thus sense to write $\fellCf{p}_n$.
\end{rem}
\begin{nrem}\label{r_standardNotation}
  In the sequel, to simplify our notation, we adopt the convention that the symbols
  $\stdf, \alphaset, \fm, \fell$ will carry subscripts and superscripts in the natural
  consistent way.
\end{nrem}
\subsection{The holonomy map}\label{ss_couplingEstimates} Loosely speaking, in the hyperbolic setting, the
coupling technique is based on the dynamical idea of ``linking mass of standard pairs to
nearby ones along stable manifold''.  In our setting, since we lack a stable direction, we
will ``link mass'' along curves that approximate the center direction for at least
$\NCn = \cO(\vei)$ iterates (as it turns out, we will use local $\NCn$-step center
manifolds $\cW_{\NCn}^\nt$, that have been defined in Section~\ref{s_geometry}).  In the
next section we will show that~\ref{a_almostTrivialp} guarantees average contraction along
such curves and, as a consequence, modulo large deviations, they can indeed serve the
purpose of stable manifolds.  The crucial issue, however, is that the regularity of the
holonomy map along the curves $\cW_{\NCn}^\nt$ deteriorates very quickly compared to the
average contraction rate on $\cW_{\NCn}^\nt$.  This fact is the main obstacle to set up an
efficient coupling strategy and, in fact, will force us to use very short center manifolds
(see Remark~\ref{r_holonomyBad}).

To make precise the above issue let us start by properly defining the \emph{holonomy map}
along center manifolds: for some small $\Delta>0$, let $\bGa=(x,\Ga(x))$ and
$\bGb=(x,\Gb(x))$ be two $\Delta$-stacked standard curves above $[a,b]$ (recall the
appropriate definitions given in Section~\ref{subsec:Standardpairs}).  Then, for
$s\in[0,1]$ define the interpolating curves $\bGs{s}$ by convex combination, \ie
\begin{align*}
  \Gs{s}(x)&=(1-s)\Ga(x)+s\Gb(x)& \bGs{s}&=(x,\Gs{s}(x)).
\end{align*}

Let $\holoDiff_n(s;x)$ be the unique solution of the following
non-autonomous ODE:
\[
\frac{\deh}{\deh s} \holoDiff_n (s;x)=
\frac{\Gb(\holoDiff_n(s;x))-\Ga(\holoDiff_n(s;x))}
{1 - \stable_n(\bGs{s}(\holoDiff_n(s;x)))\dGs{s}(\holoDiff_n(s;x))}
\stable_n(\bGs{s}(\holoDiff_n(s;x)))
\]
with initial condition $\holoDiff_n(0;x)=x$.  By invariance of the center cone and by
definition of standard curve it is immediate to show that the above problem is
well-defined\footnote{The reader can easily check that the differential equation admits a
  unique solution; recall~\eqref{e_defineSlopes} for the definition of $\stable_n$.}
provided that $x\in[a',b']$ where $a'=a+2\Delta \Kconec$ and $b'=b-2\Delta \Kconec$
(recall the definition of $\Kconec$ given in~\eqref{e_definitionKconeuc}).
Likewise, it is not difficult to check that
\( \pi F_\ve^n(\bGa(x)) = \pi F_\ve^n(\bGs{s}(\holoDiff_n(s;x)))\)
for all $x\in[a',b']$ and all $s\in[0,1]$, where $\pi$ is the projection on the
$x$-coordinate.  Let us define the \emph{$n$-step holonomy map}
$\holoMap_n:[a',b']\to [a,b]$ as $\holoMap_n(x)= \holoDiff_n(1;x)$; observe that:
\begin{equation}\label{e_holonomyDef}
  \pi F_\ve^n(\bGa(x)) = \pi F_\ve^n(\bGb(\holoMap_n(x))).
\end{equation}
The map $\holoMap_n$ is an orientation preserving diffeomorphism; geometrically, if
$\pa\in\supp\ella$ and $\pb=\bGbb(\holoMap_n(\pi \pa))\in\supp\ellb$, then $\pa$ and $\pb$
are joined by a local $n$-step center manifold $\cW_n^\nt$

We are now ready to state the relevant properties of the holonomy map $\holoMap_n$.
\begin{prop}\label{p_regularityHolonomy}
  Let $\bGa$ and $\bGb$ be two $\Delta\ve$-stacked standard curves above $[a,b]$.  For any
  $T>0$ let $N=\pint{T\vei}$ (according to Notational Remark~\ref{rem:notation}), and
  $\holoMap_N$ be the $N$-step holonomy map between the two curves;
  \begin{enumerate}
    \item\label{i_betterBoundHolonomy} for any 
    $\pa\in\bGa([a+2\Delta\Kconec \ve,b-2\Delta\Kconec \ve])$, let
    $\pb=\bGbb(\holoMap_N(\pi \pa))$.  Then, recalling the notation $\pii_n=F^n_\ve \pii$:
    \begin{equation}\label{e_betterBoundHolonomy}
      \vdist(\pa_{N},\pb_{N}) \le (1+C_T\ve+2T\closeness)\expo{\zeta_{N}(\pa)}\Delta\ve.
    \end{equation}
    \item\label{i_boundu} let $\ua(x)=\Ga{}'(x)\vei$ and
    $\ub(x)=\Gb{}'(\holoMap_N(x))\vei$ and define $\ui_n(x)=\Xi_{\pii(x)}^{(n)}\ui(x)$,
    where $\Xi_{p}^{(n)} = \Xi_{p_{n-1}}\circ\cdots\circ\Xi_{p_0}$ and $\Xi_p$ was defined
    in Section~\ref{s_geometry}; then $\ui_n(x)$ are
    the rescaled slopes of the image curve at the point $\pii_n(x)=F^n_\ve\circ \bGi(x)$
    and they satisfy, for each $0\le n\le N$:
    \begin{equation}\label{e_boundu}
      \|\ub_{n}-\ua_n\|_{\infty} \leq \Const\exp(\expb T) \Delta [\ve + (2\lambda^{-1})^n];
    \end{equation}
    where recall that $\lambda > 2$ is the minimal expansion of $f_\theta$ (see
    Section~\ref{sec:results}) and
    $\expb=\|\partial_\theta\omega\|+\gamma^u\|\partial_x\omega\|$ (see~\eqref{e_trivialBound});%
    \item\label{i_holonomy'} there exists $\consth_T\sim \Const T\exp(\expb T)$ such that:
    \begin{align}\label{e_holonomy'}
      \|\log  \holoMap'_N\| &\leq \consth_T\Delta.
    \end{align}
  \end{enumerate}
\end{prop}
\begin{proof}
  Since the standard curves $\bGi$'s are $\Delta\ve$-stacked and $\pa$ and $\pb$ are
  joined by a center manifold (which we will denote by $\cW_{N}^\nt(\pa)$) whose tangent
  belongs to the center cone, we gather that $\dist(\pa,\pb)\leq \Const\Delta\ve$;
  moreover $\pi\pa_{N}=\pi\pb_{N}$ as already observed earlier.  Moreover, let
  $h \le (1+\Const\ve)\Delta\ve$ be the distance of the projection of $\pa$ and $\pb$ on
  the $\theta$ coordinate; by~\eqref{e_defineSlopes} and definition of the center
  cone~\eqref{e_definitionCones}, we obtain that:
  \begin{align*}
    \dist(\pa_{n},\pb_{n}) \le
    (1+\Kconec)\sup_{q\in\cW_{N}(\pa,\pb)}\frac{\mu_{N}(q)}{\mu_{N-n}(F_\ve^n q)}\cdot h%
  \end{align*}
  from the above and~\eqref{e_trivialBound} we conclude that for any $0\le n\le N$
  \begin{align}\label{e_trivialBoundIntegrated}
    \dist(\pa_n,\pb_n)\le \Const\expb n\ve\Delta\ve.
  \end{align}
  If $n = N$, since $\pi\pa_N = \pi\pb_N$, we have the better estimate
  \begin{align*}
    \dist(\pa_{N},\pb_{N}) = \vdist(\pa_{N},\pb_{N}) \le \sup_{q\in\cW_{N}(\pa,\pb)}\mu_{N}(q)\cdot h%
  \end{align*}
  which using Lemma~\ref{l_distortion} immediately implies~\eqref{e_betterBoundHolonomy}
  and proves item~\ref{i_betterBoundHolonomy}.  On the other hand, since $0\leq n\leq N$,~\eqref{e_evolutionXi}
  and~\eqref{e_trivialBoundIntegrated} yield, for small enough $\ve$,
  \begin{align*}
    |\ub_n(x)-\ua_n(x)| &= |\Xi_{p_{n-1}^1(x)}(u^1_{n-1}(x))-\Xi_{p_{n-1}^0(x)}(u^0_{n-1}(x))|\\
    &\leq \Const\exp(\expb T)\Delta\ve+2\lambda^{-1}|\ub_{n-1}(x)-\ua_{n-1}(x)|\leq\notag\\
    &\leq \frac{\Const\exp(\expb T)\Delta\ve}{1-2\lambda^{-1}}+(2\lambda^{-1})^{n}|\ub(x)-\ua(x)|;
  \end{align*}
  by the definition of stacked curves we obtain $|\ub(x)-\ua(x)|\leq\Const\Delta$, which
  implies~\eqref{e_boundu} and proves item~\ref{i_boundu}.

  Let us now prove item~\ref{i_holonomy'}: differentiating~\eqref{e_holonomyDef} yields:
  \begin{equation*}
    \underbrace{\deh\pi\,\deh F^N_\ve(\bGa(x)) (1,\Ga{}'(x))}_{\text{$x$-expansion along }\bGa} =
    \underbrace{\deh\pi\,\deh F^N_\ve(\bGb(\holoMap_N(x))) (1,\Gb{}'(\holoMap_N(x)))}_{\text{$x$-expansion
        along }\bGb}\holoMap'_N(x),
  \end{equation*}
  which we can rewrite, using~\eqref{e_formulaExpansionX} and letting
  $\Gamma_N=\prod_{k=0}^{N-1}\partial_x f\circ F_\ve^k$, as

  \begin{equation}\label{e_bXn}
    \holoMap'_N(x)=\frac{\Gamma_N(\pa(x))}{\Gamma_N(\pb(x))}%
    \prod_{n=0}^{N-1}\frac{1+\ve \frac{\partial_\theta f}{\partial_x f}(\pa_n(x))\ua_{n}(x)}{1+\ve \frac{\partial_\theta f}{\partial_x f}(\pb_{n}(x))\ub_n(x)}.
  \end{equation}
   Then, using~\eqref{e_trivialBoundIntegrated} we gather that:
  \begin{equation}\label{e_dominantTermHolonomy}
    \left|\log\frac{\Gamma_N(\pa(x))}{\Gamma_N(\pb(x))}\right|\leq\Const T\exp(\expb T)\Delta,
  \end{equation}
  and using item (b) we can conclude:
  \[
  \left|\log \prod_{n=0}^{N-1}\frac{1+\ve \frac{\partial_\theta f}{\partial_x
        f}(\pa_n(x))\ua_n(x)}{1+\ve \frac{\partial_\theta f}{\partial_x
        f}(\pb_n(x))\ub_n(x)}\right|\leq \Const(T+1)\exp(\expb T)\Delta\ve.
  \]
  The above estimates imply~\eqref{e_holonomy'} and consequently conclude the proof of our
  lemma.
\end{proof}
\begin{rem}\label{r_holonomyBad}
  Item~\ref{i_holonomy'} in the above lemma implies in particular that the $N$-step
  holonomy map $\holoMap_N$ has good regularity properties for $T=\cO(1)$ provided that
  $\Delta=\cO(1)$, \ie if the stacked pairs are at a distance $\cO(\ve)$.  A Local Central
  Limit Theorem is thus crucial for the effectiveness of the coupling procedure, since it
  provides information about the distribution of standard pairs at the $\cO(\ve)$-scale.
\end{rem}
\section{The coupling argument}\label{s_coupling}
\subsection{An informal exposition of the global strategy}
For simplicity let us first assume that $\nz=1$ so that $\{\theta=\theta_{1,-}\}$ is the
only attractor for the averaged dynamics. Since we expect the real dynamics to be well
approximated by a $\sve$-diffusion around the averaged dynamics, we will be able to
conclude (see the Bootstrap Lemma~\ref{l_bootstrap}) that, if we let any two standard
families evolve for a sufficiently long time (which turns out to be $\cO(\vei\log\vei)$),
then a substantial portion of their mass will be carried by standard pairs which are
supported in a $\cO(\sqrt\ve)$ neighborhood of $\theta_{1,-}$.  Using a Local Central
Limit Theorem (see Theorem~\ref{thm:lclt}) we can control effectively the distribution of
such standard pairs with a $\cO(\ve)$-resolution.  Once two standard pairs are stacked at
a distance $\ve$, since we are close to a sink and by~\ref{a_almostTrivial}, the averaged
system will make them (slowly) approach to each other (see
Lemma~\ref{l_couplingStep}). Once they are sufficiently close
(e.g. $\cO(\ve^{1+\expoCloseness})$ for some $\expoCloseness>0$) we can show that the real
dynamics follows the average one almost all the time (a part from rare large deviations)
and that the distance between the standard pairs keeps contracting forever with positive
probability. Thus we can couple (see the Coupling procedure,
Lemma~\ref{l_couplingProcedure}) almost all their mass forever.  We then conclude by
iteratively applying the same argument to the mass in the leftover pieces (see
Lemma~\ref{l_coupling}).

If $\nz>1$ there are two possibilities: if Assumption~\ref{a_fluctuation} holds,
then our Large Deviations results (see Theorem~\ref{l_largeDevzLowerBound}) allow us to
prove that any standard pair will have some positive (although exponentially small in
$\vei$) probability of being close to $\theta_{1,-}$ after time $\cO(\vei\log\vei)$; we
can then conclude the proof by applying the argument above to this tiny amount of mass at
each step.

Otherwise, if~\ref{a_fluctuation} does not hold but~\ref{a_fluctuationGap} holds,
  Lemma~\ref{l_fluctuationGap}\ref{pp_recurrentSink} guarantees the existence of at least
  one and at most $\nz$ recurrent sinks.  Then, once again using Large Deviations
  arguments, any standard pair that is supported on $\{\theta\in\trap_i\}$ will have
  positive (although exponentially small in $\vei$) probability of being close to
  $\sink i$ after time $\cO(\vei\log\vei)$; moreover~\eqref{e_protoInv} guarantees that
  any pushforward will also enjoy the same property.  As before, we can conclude by
  iteratively applying the same argument to this tiny amount of mass at each step.  Of
  course this procedure does not necessarily yield a unique invariant measure, but rather as many
  distinct invariant measures as the number of distinct trapping sets.
\subsection{The basic Coupling Step}\label{ss_couplingStep}
We now describe the core of our coupling argument, \ie, we describe how to actually
couple two standard pairs (or more precisely, the processes they generate) for $\cO(\vei)$
iterations.

Recall that at the beginning of the previous section we fixed the constant $\TCn > 0$ (and
correspondingly $\NCn$) sufficiently large; fix $\closeness$ so that
$\TCn\closeness < 1/64$ (recall that $\closeness$ was introduced in the definition
(see~\eqref{eq:regularizedpsi}) of $\clyapReg$).  Recall also the definitions of a matched
couple and of a matched pushforward given in Section~\ref{ss_couplingDefinitions}
\begin{lem}[Coupling Step]\label{l_couplingStep}%
  For any $\bar\Delta>0$, there exist $\bar\ve>0$ so that the following holds.  For any
  $0\le\NCoup\le \NCn$,
  $\ve\in(0,\bar\ve)$, $\Delta\in(0,\bar\Delta)$ and $\Delta\ve$-matched standard couple
  $\ellC$, there exist sequences of pushforwards
  $\left(\coupledStepSeq{\stdfC}n\right)_{n=0}^{\NCoup}$ and
  $\left(\uncoupledStepSeq{\stdfC}n\right)_{n=0}^{\infty}$ so that
  $\coupledStepSeq{\stdfC}\NCoup$ is a matched pushforward and
  \begin{align*}
    \eqc{\pFve^n\ellC}\ni\mC\coupledStepSeq{\stdfC}n+(1-\mC)\uncoupledStepSeq{\stdfC}n.
  \end{align*}
In addition,
  \begin{enumerate}
  \item \label{i_subcurve} $\coupledStepSeq\stdfa0=\{\{0\},\ell_0\}$, $\ell_0\subset\ellCa$ and
\[
\hat\mu_{\stdfa}(\supp\coupledStepSeq\stdfa0) = \mC=\mC(\Delta)=(1-\caa\Delta\ve)\cdot\exp(-4\consth_{\TCn}\Delta),
\]
where $\caa$ is a constant which does not depend on $\ellC$ and
  $\consth_{\TCn}$ is defined in Proposition~\ref{p_regularityHolonomy}\ref{i_holonomy'};
\item \label{i_matchedCloseness} any
  $\fellCf{p}_{\NCoup}\in \coupledStepSeq{\stdfC}\NCoup$ is a
  $\Delta\ve\expo{\zeta_{\NCoup}(p)+\Const\ve+1/32}$-matched standard couple;
  \item\label{i_leftoversPrestandard} $\uncoupledStepSeq{\stdfC}n$ is a standard coupling
  provided that $n\ge\NCoup+ \Const|\log(\const\Delta)|$.
  \end{enumerate}
\end{lem}
\begin{proof}
  To fix ideas, for $i=\privatea,\privateb$, let $\ellCi=(\Gi,\rho)$ and $[a,b] =
  \pi(\supp\ellCa) = \pi(\supp\ellCb)$; let $\holoMap_\NCoup$ be the $\NCoup$-step holonomy map
  between $\ellCa$ and $\ellCb$, defined in Section~\ref{ss_couplingEstimates}.

  Fix $\caa$ to be specified later and define $\aaa$ and $\ba$ so that
  \begin{align*}
    \int_a^{\aaa}\rho(x)\deh x = \int_{\ba}^b\rho(x)\deh x = \frac12\caa\Delta\ve.
  \end{align*}
  By~\eqref{e_boundRho}, the definition of $\holoMap_{\NCoup}$ and our estimates for the
  center cone, we can choose $\caa$ to be so large that the interval $[\aaa,\ba]$ is in
  the domain of definition of $\holoMap_\NCoup$; we let
  $[\ab,\bb]=\holoMap_\NCoup[\aaa,\ba]$.  Moreover, eventually by further increasing
  $\caa$, we also guarantee that for $i\in\{\privatea,\privateb\}$:
  \begin{equation}
    \ai-a,b-\bi\ge \Delta\ve\label{e_shortLengthBound}.
  \end{equation}
  Finally, we assume $\caa$ to be sufficiently large so that
  \begin{align*}
    \int_a^{\ab}\rho(x)\deh x = \int_{\bb}^b\rho(x)\deh x \le \caa\Delta\ve.
  \end{align*}

 For $i\in\{\privatea,\privateb\}$ let us cut the standard pair
  $\ellCi$ at the points $\ai$ and $\bi$;
  in doing so we obtain two (very) short standard pairs (which we denote by $\ellLi$ and
  $\ellRi$) whose lengths are bounded below by~\eqref{e_shortLengthBound} and a (possibly
  short) standard pair, which we denote by $\elli_*$, with
  $|\elli_*|\geq\delta/2-\Const\Delta\ve$ (see Figure~\ref{f_couplingStep} for a sketch of
  our setup).
  \begin{figure}[!ht]
    \centering
    \includegraphics{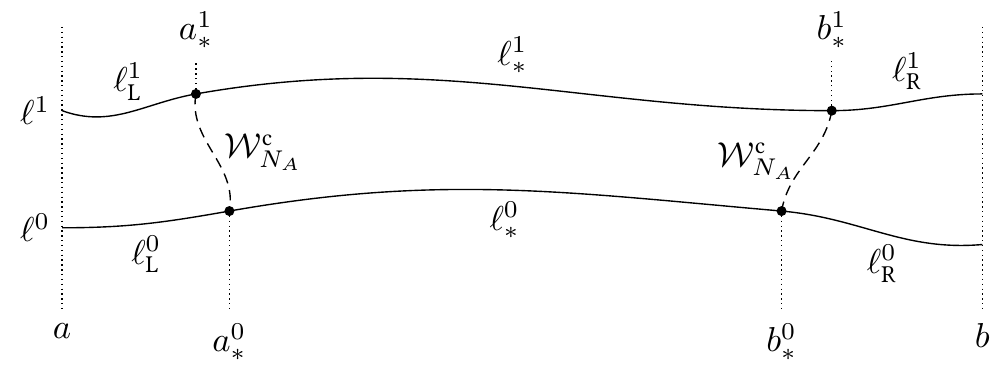}
    \caption{Setup for our decomposition.}
    \label{f_couplingStep}
  \end{figure}

  Let us introduce some notation: we define $\elli_*=(\Gi_*,\rhoi_*)$
  (\resp $\ellLi=(\GLi,\rhoLi)$, $\ellRi=(\GRi,\rhoRi)$), where $\Gi_*$ (\resp $\GLi$,
  $\GRi$) is the restriction of $\Gi$ to the interval $[\ai,\bi]$ (\resp $[a,\ai]$,
  $[\bi,b]$) and $\rhoi_*=\rho/\mi_*$ (\resp $\rhoLi=\rho/\mLi$, $\rhoRi=\rho/\mRi$) where
  $\mi_*=\int_{\ai}^{\bi}\rho(x)\deh x$ (\resp  $\mLi=\int_{a}^{\ai}\rho(x)\deh x$,
  $\mRi=\int_{\bi}^{b}\rho(x)\deh x$).  In particular, our construction yields
  $\mLa=\mRa=\caa\Delta\ve/2$ and $\ma_*=1-\caa\Delta\ve$.

  Let $\rhoa_\Cp=\rhoa_*$ and define $\rhob_\Cp$ on
  $[\ab,\bb]$ as the push-forward $\rhob_\Cp=\holoMap_{\NCoup*}\rhoa_\Cp$.  More explicitly,
  for any $\xb\in[\ab,\bb]$ let $\xa=\holoMap_\NCoup\invr\xb$, then:
  \begin{equation}
    \rhob_\Cp(\xb)=\frac{\rhoa_\Cp(\xa)}{\holoMap_\NCoup'(\xa)};\label{e_definitionRhobCp}
  \end{equation}
  observe that $\rhob_\Cp$ is not necessarily a standard density.  We now claim that
  \begin{equation}\label{e_estimateXi}
    \rhob_\Cp(\xb)\leq\expo{2\consth_{\TCn}\Delta}\mb_*\rhob_*(\xb).
  \end{equation}
  In fact, by~\eqref{e_definitionRhobCp} and since by definition $\mi_*\rhoi_* = \rho_i$:
  \begin{align*}
    \left|\log\frac{\rhob_\Cp(\xb)}{\mb_*\rhob_*(\xb)}\right|&=
    \left|\log\frac1{\ma_*}\frac{\ma_*\rhoa_*(\xa)}{\mb_*\rhob_*(\xb)}\frac{1}{\holoMap'_\NCoup(\xa)}\right|
    \\&\leq%
    \left|\log{\ma_*}\right|+\left|\log\frac{\rho(\xa)}{\rho(\xb)}\right|+\left|\log
      \holoMap_\NCoup'(\xa)\right|\leq
    \Const\Delta\ve + \consth_{\TCn}\Delta,
  \end{align*}
  where the first two terms can be bounded using invariance of the center cone and the
  definition of standard density, and the third one using
  Proposition~\ref{p_regularityHolonomy} (where the reader can also find the definition of
  $\consth_{\TCn}$).  Thus, provided that $\bar\ve$ is sufficiently
  small,~\eqref{e_estimateXi} holds, which in turn implies that there exist positive
  densities $\rhoi_\Rd$ so that, letting $\mC=\ma_*\expo{-4\consth_{\TCn}\Delta}$:
  \begin{subequations}\label{e_definitionRho}
    \begin{align}
      \ma_*\rhoa_*(\xa)&=\mC\rhoa_\Cp(\xa)+(\ma_*-\mC)\rhoa_\Rd(\xa)\label{e_definitionRhoa}\\
      \mb_*\rhob_*(\xb)&=\underbrace{\mC\rhob_\Cp(\xb)}_{\text{coupled}}
                         +\underbrace{(\mb_*-\mC)\rhob_\Rd(\xb)}_{\text{uncoupled}};\label{e_definitionRhob}
    \end{align}
  \end{subequations}
  where, in particular $\rhoa_\Rd=\rhoa_\Cp=\rho_*^0$.

  Let us now define $\coupledStepSeq{\elli}{} = (\Gi_*,\rhoi_\Cp)$ and $\elli_\Rd=(\Gi_*,\rhoi_\Rd)$; let
  furthermore
  \begin{align*}
  \uncoupledStepSeq{\stdfi}{} = \frac\mLi{1-\mC} \ellLi + \frac\mRi{1-\mC} \ellRi +
  \frac{\mi_*-\mC}{1-\mC} \elli_\Rd.
  \end{align*}
  We let $\coupledStepSeq{\ellC}{}$ be the coupling of $\coupledStepSeq{\ella}{}$ and
  $\coupledStepSeq{\ellb}{}$ given by
  \[
  \boldsymbol\mu(g)=\int g(\bG_*^0(x), \bG_*^1(H_N(x)) \rhoa_\Cp(x) \deh x,
  \]
  and $\uncoupledStepSeq{\stdfC}{}$ be the independent coupling of
  $\uncoupledStepSeq{\stdfa}{}$ and $\uncoupledStepSeq{\stdfb}{}$.  Also, let
  $\coupledStepSeq{\stdfC}n$ and $\uncoupledStepSeq{\stdfC}n$ be pushforwards of
  $\coupledStepSeq{\ellC}{}$ and $\uncoupledStepSeq{\stdfC}{}$, respectively.  We claim
  that these couplings satisfy properties~\ref{i_subcurve}--\ref{i_leftoversPrestandard}.

  In fact,~\ref{i_subcurve} follows by our construction.  Then, observe that, by
  Proposition~\ref{p_regularityHolonomy}\ref{i_betterBoundHolonomy}, the \couples{}
  $\fellCf{p}_\NCoup$ are in fact matched.  Since ${\coupledStepSeq{\ella}{}}$ is
  standard, $\fellCfa{p}_\NCoup$ will also be standard, and consequently so will be
  $\fellCfb{p}_\NCoup$, since the two pairs have equal densities;
  item~\ref{i_matchedCloseness} then follows by estimates~\eqref{e_betterBoundHolonomy}.

\item We now proceed to prove item~\ref{i_leftoversPrestandard}: let $\fellCfi{\alpha}_\NCoup$ be a pair in
  $\uncoupledStepSeq{\stdfi}{\NCoup}$; there are two possibilities:
  \begin{enumerate}[label=\roman*.,ref = \roman*]
  \item\label{i_caseLeftRight} $\fellCfi{\alpha}_\NCoup$ belongs to the $\NCoup$-th image of either $\ellRi$ or
    $\ellLi$
  \item\label{i_caseUncoupled} $\fellCfi{\alpha}_\NCoup$ belongs to the $\NCoup$-th image of $\elli_\Rd$.
  \end{enumerate}
  In the first case, we know by~\eqref{e_shortLengthBound} that the length of the short
  curves $\ellLi$ or $\ellRi$ is bounded below by $\const\Delta\ve$; Remark~\ref{r_short}
  then implies that the pairs $\ellLi$ and $\ellRi$ are
  $\Const|\log(\const\Delta\ve)|$-prestandard, which in particular proves
  item~\ref{i_leftoversPrestandard} in case~\ref{i_caseLeftRight}.  We are left with
  case~\ref{i_caseUncoupled}: by definition $\ella_\Rd$ is a standard pair, hence so will
  be $\fellaf{\alpha}_\NCoup$.  We therefore only need to prove our statement for pairs in
  the image of $\ellb_\Rd=(\Gb_*,\rhob_\Rd)$.  By~\eqref{e_definitionRhob} we have:
  \[
  (\mb_*-\mC)\rhob_\Rd=\mb_*\rhob_*-\mC\rhob_\Cp
  \]
  Let us denote by $\rhob_{\Rd,\NCoup}$ the pushforward of $\rhob_\Rd$ by $F_{\ve*}^\NCoup$; we
  obtain:
  \begin{align}\label{e_equationRhoR}
    (\mb_*-\mC)\rhob_{\Rd,\NCoup}(x_\NCoup)&=
    (\mb_*-\mC)\rhob_\Rd(\xb_0(x_\NCoup))\frac{\deh \xb_0}{\deh x_\NCoup}\notag\\
    &=\mb_*\rhob_*(\xb_0(x_\NCoup))\frac{\deh \xb_0}{\deh x_\NCoup}-%
    \mC\rhob_\Cp(\xb_0(x_\NCoup))\frac{\deh \xb_0}{\deh x_\NCoup},
  \end{align}
  where we denote with $\xii_0(x_{\NCoup})$ the $x$-coordinate of the point
    $\pii_0(x_{\NCoup})\in\gapi_\NCoup(\alpha)\subset\supp\elli_0$ so that
    $\pi F_\ve^{\NCoup}(\pii_0(x_{\NCoup})) = x_\NCoup$ (recall the definition of
    $\gapi_\NCoup$ given in Section~\ref{ss_pushforwards}).
  The first term is the push-forward of a standard density (and thus a standard density);
  the second term is also a standard density, since by our construction
  $\rhob_\Cp(\xb_0(x_\NCoup))\frac{\deh \xb_0}{\deh
    x_\NCoup}=\rhoa_\Cp(\xa_0(x_\NCoup))\frac{\deh \xa_0}{\deh x_\NCoup}$,
  which is the push-forward of a standard density.  We now take derivatives
  of~\eqref{e_equationRhoR} and obtain:
  \begin{align*}
    \left\|\frac{{\rhob_{\Rd,\NCoup}}'}{\rhob_{\Rd,\NCoup}}\right\|&\leq%
    \left\|\frac{\mb_*\rhob_*+\mC\rhob_\Cp}{\mb_*\rhob_*-\mC\rhob_\Cp}\right\|\spc2%
    &\left\|\frac{{\rhob_{\Rd,\NCoup}}''}{\rhob_{\Rd,\NCoup}}\right\|&\leq%
    \left\|\frac{\mb_*\rhob_*+\mC\rhob_\Cp}{\mb_*\rhob_*-\mC\rhob_\Cp}\right\| \dttrho0\spc2,%
  \end{align*}
  and using~\eqref{e_estimateXi}:
  \begin{align*}
    \left\|\frac{\mb_*\rhob_*+\mC\rhob_\Cp}{\mb_*\rhob_*-\mC\rhob_\Cp}\right\|&\leq
    \frac{2}{1-\expo{-2\consth_{\TCn}\Delta}}
    \leq2\left(1+\frac1{2\consth_{\TCn}\Delta}\right).
  \end{align*}
  Hence,
  $\rhob_{\Rd,\NCoup}\in
  D_{2\left(1+1/{2\consth_{\TCn}\Delta}\right)\spc2}(\Gb_{\Rd,\NCoup})$
  and by Remark~\ref{r_prestandard} we can thus conclude that any pair in
  case~\ref{i_caseUncoupled} is a $\Const|\log(\const\Delta)|$-prestandard pair.
\end{proof}
\begin{cor}\label{c_wasserstein}
  For any $\bar\Delta>0$, there exist $\bar\ve>0$ so that the following holds.  For any
  $\NCoup\le \NCn$, $\ve\in(0,\bar\ve)$, $\Delta\in(0,\bar\Delta)$ and $\Delta\ve$-matched
  standard couple $\ellC$ we have
  \begin{align*}
    \wDist(F_{\ve*}^{\NCoup}\mu_{\ella},F_{\ve*}^{\NCoup}\mu_{\ellb})&\le\Const\Delta.
   \end{align*}
\end{cor}
\begin{proof}
  Applying Lemma~\ref{l_couplingStep} to $\ellC$ we obtain
  \begin{align*}
    \eqc{\pFve^\NCoup\ellC}\ni\mC\coupledStepSeq{\stdfC}\NCoup+(1-\mC)\uncoupledStepSeq{\stdfC}\NCoup.
  \end{align*}
  By item~\ref{i_matchedCloseness} we gather that $\coupledStepSeq{\stdfC}\NCoup$ is a
  $\Const\Delta\ve$-matched coupling; thus:
  \begin{align*}
    \wDist(F_{\ve*}^{\NCoup}\mu_{\ella},F_{\ve*}^{\NCoup}\mu_{\ellb})&
    \le \Const\mC\Delta\ve + (1-\mC),
  \end{align*}
  from which we conclude using the estimate for $\mC$ given in item~\ref{i_subcurve}.
 \end{proof}
\subsection{The global Coupling procedure}
The idea is now to iterate Lemma~\ref{l_couplingStep} with $\NCoup = \NCn$ and discard
those \couples{} which at step $k$ are not exponentially close in $k$.  The crucial fact
to prove is that if we start coupling pairs which are sufficiently close, this strategy
can be carried out with probability arbitrarily close to $1$; this, together with other
useful estimates, is the content of the following lemma, whose proof will be given in
Section~\ref{ss_proofCouplingProcedure}.  Recall the definition of conditioned
subcouplings given in Section~\ref{ss_couplingDefinitions}.

\begin{lem}\label{l_couplingProcedure}
  For any $\smallness>0$, provided that $\ve$ is small enough, there exists
  $\expoCloseness>0$ so that for any $\ve^{1+\expoCloseness}$-matched standard couple
  $\ellC$, there exists a sequence $\auxFamily{\stdfC}k\in\eqc{\pFve^{k\NCn}\ellC}$,
  $k\in\bN$, and random variables\footnote{ Recall that, according to Notational
    Remark~\ref{r_standardNotation}, $\auxFamily{\alphaset}k$ is the index set of
    $\auxFamily{\stdfC}k$ and $\auxFamily\fm{k}$ the corresponding measure.}
  $\auxFamily\butime{k}:\auxFamily{\alphaset}k\to\bZ\cup\{\infty\}$ satisfying the
  following properties:
  \begin{enumerate}
  \item \label{i_gettingCloser} for any $k\ge0$, $\alpha\in\auxFamily{\alphaset}k$ so that $\auxFamily\butime{k}
    (\alpha)=\infty$, the standard couple $\auxFamily{\fellCf{\alpha}}k$ is
    $\Const\expo{-\const k}\ve^{1+\expoCloseness/2}$-matched.
  \item \label{i_compatibility} for any $l< k\le k'$ we have
    $\auxFamily\fm{k}(\{\auxFamily\butime k=l\})=\auxFamily\fm{k'}(\{\auxFamily\butime {k'}=l\})$; moreover
    $\auxFamily{\stdfC}{k'}|\{\auxFamily\butime
    {k'}=l\}\in\eqc{\pFve^{({k'-k})\NCn}\auxFamily{\stdfC}k|\{\auxFamily\butime k=l\}}$;
    finally, the family $\auxFamily{\stdfC}{l}|\{\auxFamily\butime {l}=l-1\}$ is
    $l\NCn$-prestandard.
  \item \label{i_tailBoundM} $\MCp{k}=\auxFamily{\fm}k(\auxFamily\butime k=\infty)$ is a
    non-increasing sequence in $[0,1]$ so that, for all $k\in\bN$,
    \begin{equation}
      \MCp{k}\geq\expo{-\smallness}.\label{e_firstMass}
    \end{equation}
Moreover, if $k'\ge k$, we have
    \begin{equation}
      \MCp{k}-\MCp{k'}\leq \smallness \exp(-\const k/\log\vei).\label{e_tailboundM}
    \end{equation}
  \end{enumerate}
\end{lem}
\begin{rem}
  As we already explained, the lack of uniform hyperbolicity implies that the dynamics
  might fail to bring together at a uniform rate two standard pairs which started close
  together.  When such a failure happens, we declare the couple to \emph{break up} and we
  give up tracking them in the future.  The above lemma tells us that if two standard
  pairs are sufficiently close, such break ups are relatively unlikely.  The random
  variable $\auxFamily\butime k$ in the above statement keeps track of the coupling step
  at which the corresponding couple broke up (up to the $k$-th step); if it is $\infty$,
  it means that the couple did not break up (yet).  Thus~\eqref{e_firstMass} guarantees
  that a break up will happen with probability which is arbitrarily small with
  $\smallness$; similarly~\eqref{e_tailboundM} gives an exponential tail bound (with rate
  $\cO({\ve/\log\vei})$) on the probability of a break up occurring after $k$ steps.
\end{rem}

Observe that Lemma~\ref{l_couplingProcedure} requires the standard couple $\ellC$ to be
$\Const\ve^{1+\expoCloseness}$-matched; the following lemma specifies under which
conditions the dynamics will, in $\cO(\vei\log\vei)$ iterations, bring a portion of the
image of two standard pairs in such a convenient position.  Recall that a standard pair
$\ell$ \emph{\ispinnedto{} the trapping set $\trap_i$} if $\avgtheta\ell\in\trap_i$;
likewise a standard family $\stdf$ is said to \emph{\bepinnedto{} $\trap_i$} if for any
$\alpha\in\alphaset$, $\avgtheta{\fellf{\alpha}}$ \ispinnedto{} to $\trap_i$.  A standard
couple $\ellC = (\ella,\ellb)$ is said to \emph{\bepinnedto{} the trapping set $\trap_i$}
if both $\ella$ and $\ellb$ \arepinnedto{} $\trap_i$.  Finally, we denote with $\nz_{i}$
the number of sinks $\sink j$ that are contained in
$\trap_i$\footnote{\label{f_nzIndependent} Remark that $\nz_{i}$ does not depend on
  $\aeps$, provided $\aeps$ has been chosen small enough.}.
\begin{rem}\label{r_eventualInvariance}
  Observe that by~\eqref{e_protoInv}, if $\ell$ \ispinnedto{} $\trap_i$, then any
  $n$-pushforward of $\ell$ \ispinnedto{} $\trap_i$ provided that $n\ge\pint{\TForb\vei}$.
\end{rem}

\begin{lem}[Bootstrap]\label{l_bootstrap}
  Let $\sink i$ be a recurrent sink; for any $\expoCloseness>0$, there exist
  $\RBs,\bar\ve>0$ so that for any $\Rgen\ge\RBs$,
  $\Tgen=\pint{\Rgen\log\vei}, \ve\in(0,\bar\ve)$ and any standard couple $\ellC$
  \pinnedto{} $\trap_i$, we have
  \begin{align*}
    [F_{\ve*}^{\Tgen\NCn}\mu_{\ellC}]\ni\mB\bootstrapped\stdfC{N}+(1-\mB)\notBootstrapped\stdfC{N},
  \end{align*}
where:
  \begin{enumerate}
  \item $\bootstrapped\stdfC{}=(\bootstrapped\stdfCa{},\bootstrapped\stdfCb{})$ is an
    $\ve^{1+\expoCloseness}$-matched standard coupling;
  \item $\notBootstrapped\stdfa{N}$ and $\notBootstrapped\stdfb{N}$ are
    $\cO(\expoCloseness\log\vei)$-prestandard families;
  \item \label{i_estimatemB} $\mB=\mB(\Rgen)$ is a non-increasing function of $\Rgen$;
    moreover if $\nz_{i}=1$, $\mB$ can be chosen to be uniform in $\ve$; otherwise
    $\mB\sim\expo{-\const\vei}$.
  \end{enumerate}
\end{lem}
The proof of Lemma~\ref{l_bootstrap} will be given in Section~\ref{sec:coup-prooftwo}.
Recall the definition of Wasserstein distance given in~\eqref{eq:Wasserstein}.  We now see
how the previous results allow us to prove the following
\begin{lem}[Coupling Lemma]\label{l_coupling}
  There exist $\bar\ve>0$ so that, if $\ve\in(0,\bar\ve)$, for any two standard pairs
  $\ella$, $\ellb$ \pinnedto{} the same trapping set:
  \begin{align*}
    \wDist(F_{\ve*}^{n}\mu_{\ella},F_{\ve*}^{n}\mu_{\ellb})&\le\Const\expo{-\const \mB\cdot
      n\ve/\log\vei}
  \end{align*}
\end{lem}
\begin{proof} Our main task is essentially a bookkeeping problem: as we push forward a
  standard couple we will produce matched pairs (hopefully more and more of them),
  prestandard pairs that cannot be used for anything as yet, and standard pairs that have
  recovered and are ready to reenter in the dating business. To keep track of all these
  objects some notation is needed. Let $\smallness>0$ small and $r\in\bN$ large enough to
  be specified later; define $\TCp$ by requiring that
  \begin{align*}
    \RCp = \pint{\TCp\log\vei}=2\pint{r\log\vei}>\pint{\RBs\log\vei},
  \end{align*}
  where $\RBs$ is the constant appearing in Lemma~\ref{l_bootstrap}.

  To fix ideas we assume that $\ella$ and $\ellb$ both \pinnedto{} $\trap_i$; we will now
  inductively define:
  \begin{itemize}
  \item for $q\ge 0$, a sequence $(\availableC{\stdfC}{q})_q$ of couplings of $\NCn$-prestandard
    families \pinnedto{} $\trap_i$ and a corresponding sequence of weights $(\MUS{q})_q$
  \item for $q\ge 1$, a sequence $(\datingC{\stdfC}{q})_q$ of $\Const\ve^{1+\expoCloseness}$-matched standard
    couplings \pinnedto{} $\trap_i$ and a corresponding sequence of weights $(\MCS{q})_q$.
  \end{itemize}
  The reader should think of such families as a bookkeeping device to account for the
  dynamics after $q\RCp\NCn$ iterates.  Roughly speaking $\datingC{\stdfC}{q}$ are the
  standard pairs that we are able to couple at time $q\RCp\TCn$. At later times some of
  this standard pairs break up (this is recorded by the random variables
  $\auxFamily\butime{k}$ defined in Lemma~\ref{l_couplingProcedure}) or lose some mass (in
  form of, possibly very short, prestandard pairs) while some standard pairs never had a
  chance to couple. The family $\availableC{\stdfC}{q}$ contains all the standard pairs
  that are available to try a new coupling in the time interval
  $[q\RCp\TCn, (q+1)\RCp\TCn]$. The reason why such a scheme is going to converge is that,
  as time goes on, less and less mass uncouples (see Lemma~\ref{l_couplingProcedure}),
  while it is always possible to couple a fix percentage of the uncoupled mass (see
  Lemma~\ref{l_bootstrap}).

  Let us now describe the induction step: at step $q$, we inductively assume
  that $\MUS{q}$ and $\availableC\stdfC q$ are defined, together with $\MCS{s}$ and
  $\datingC\stdfC{s}$ for $0< s\leq q$ and construct $\MCS{q+1}$, $\datingC\stdfC{q+1}$,
  $\MUS{q+1}$ and $\availableC\stdfC {q+1}$.

  For the base step, let  $\MUS0=1$ and $\availableC\stdfC 0=\ellC$.

  Next, consider $q>0$. Let  $\expoCloseness>0$ be the constant given by Lemma~\ref{l_couplingProcedure}.  By our inductive assumptions $\availableC\stdfC q$ is a
  \couple{} of $\NCn$-prestandard families; let $\availableC\stdfC q_{\NCn}$ be a standard
  pushforward of $\availableC\stdfC q$; then, we can apply Lemma~\ref{l_bootstrap} to each couple of
  standard pairs in $\availableC\stdfC q_{\NCn}$ with $\Rgen=(\RCp-1)/\log\vei\ge\RBs$
  and obtain
  \[
  \eqc{\pFve^{\RCp\NCn}\availableC\stdfC{q}}\ni \mB\bootstrapped{\prepC{\availableC\stdfC{q}}}{{\RCp\NCn}} +
  (1-\mB)\notBootstrapped{\prepC{\availableC\stdfC{q}}}{{\RCp\NCn}}.
  \]
  We define $\datingC\stdfC{q+1}=\bootstrapped{\prepC{\availableC\stdfC{q}}}{{\RCp\NCn}}$
  and $\MCS{q+1}=\MUS{q}\mB$.  Observe that, by construction, $\datingC\stdfC{q+1}$ is a
  $\ve^{1+\expoCloseness}$-matched standard coupling \pinnedto{} $\trap_i$.  Let us
  denote with $\auxFamily{\datingC\stdfC{q+1}}k$ the sequence of couplings and with
  $\auxFamily{\datingC\butime{q+1}}k$ the sequence of random variables which we obtain by
  applying Lemma~\ref{l_couplingProcedure} to each standard coupling in
  $\datingC\stdfC{q+1}$.

  \newcommand{\blob}[2]{Y_{#2}^{#1}}

  Note that, according to Lemma~\ref{l_couplingProcedure}, a certain number of pairs will
  break up as times goes by. The variables $\auxFamily{\datingC\butime{q}}k$ keep track of
  when such breakups occurred. Moreover, recall that Lemma~\ref{l_couplingProcedure}
  asserts that if a standard couple in $\auxFamily{\datingC\stdfC{q}}k$ broke up at time
  $s$ (\ie $\auxFamily{\datingC\butime{q}}k=s$), then it will recover at time
  $s\NCn$. Hence the couple in the family $\auxFamily{\datingC\stdfC{q}}k$ that broke up at
  the step $\cO(k/2)$ have recovered (that is, are standard), thus available for starting
  again a coupling procedure.

  Then, we define the coupling $\availableC\stdfC{q+1}$ so that
  \begin{align*}
    \MUS{q+1}\availableC\stdfC{q+1}%
    &=\MUS{q}(1-\mB)\notBootstrapped{\prepC{\availableC\stdfC{q}}}{{\RCp\NCn}}+\\&\phantom{=}
                                                                                   +\sum_{s=1}^{q}\MCS{s} \auxFamily{\datingC\fm s}{\blob q
                                                                                   s\RCp}\left(\auxFamily{\datingC\butime s}{\blob q s\RCp}\in[(\blob qs-1)\RCp/2,\blob q
                                                                                   s\RCp/2)\right)\\
    &\phantom{=\sum_{s=1}^{q}\MCS{s}}\times\auxFamily{\datingC\stdfC s}{\blob q
      s\RCp}\cond\left\{\auxFamily{\datingC\butime s}{\blob q s\RCp}\in[(\blob qs-1)\RCp/2,\blob q
      s\RCp/2)\right\}
  \end{align*}
  where $\blob q s=q-s+1$.  In the above expression, the first term accounts for standard
  pairs which did not come close enough during the current step and we could not start coupling.
  The second terms account for standard pairs which we coupled in some
  previous step, broke up and recovered some time between the beginning and the end of the
  current step.  Correspondingly we let
  \begin{align}\label{e_MUS}
    \MUS{q+1}&=%
               \MUS{q}(1-\mB)
               \sum_{s=1}^{q}\MCS{s}\left(\MCp{(\blob q s-1)\RCp/2}-\MCp{\blob q s\RCp/2}\right).%
  \end{align}
  Observe that by definition $\availableC\stdfC{q+1}$ \ispinnedto{} $\trap_i$.  Now that we
  defined the auxiliary sequences of couplings, we claim that
  \begin{align}\label{e_boundMUS}
    \frac{\MUS{q+1}}{\MUS{q}}\in[\vartheta_*,\vartheta],
  \end{align}
  where $\vt = 1-\frac12\mB$ and $\vt_*=1-\mB$; observe that both $\vt$ and $\vt_*$
  increase with $r$ by Lemma~\ref{l_bootstrap}\ref{i_estimatemB}.  In fact
  by~\eqref{e_MUS} and the definition of $\MCS{s}$, we have:
  \begin{align*}
    \frac{\MUS{q+1}}{\MUS{q}}&=%
    (1-\mB)+\mB\sum_{s=0}^{q-1}\frac{\MUS{s}}{\MUS{q}}\left(\MCp{(\blob
        qs-2)\RCp/2}-\MCp{(\blob qs-1)\RCp/2}\right);
  \end{align*}
  the above immediately implies the lower bound $\frac{\MUS{q+1}}{\MUS{q}}\ge\vt_*$, since
  every term of the sum is positive.  In order to prove the upper bound observe that, by
  the lower bound and the above equation:
  \begin{align*}
    \frac{\MUS{q+1}}{\MUS{q}}&\le(1-\mB)+\vt_*\invr\mB\smallness\left[\sum_{k=0}^{q-1}\vt_*^{-k}\exp(-\const
      r k)\right]
  \end{align*}
  where we used~\eqref{e_tailboundM}; observe that by choosing $\gamma$ small and $r$
  large we can make the second term arbitrarily small, from which we conclude
  that~\eqref{e_boundMUS} holds.

  Let us now fix $n > 0$; let $k = \pint{n/\NCn}$, $q=\pint{k/\RCp}$ and, for
  $0\le s\le q$ define ${\leftoverm}_s=k-s\RCp$.  Let us first construct a coupling
  $\stdfC_{k\NCn}\in\eqc{\pFve^{k\NCn}\ellC}$ given by:
  \begin{align*}
    \stdfC_{k\NCn} &:= \sum_{s=1}^{q}\MCS{s} \auxFamily{\datingC\fm s}{{\leftoverm}_s}(\auxFamily{\datingC\butime s}{{\leftoverm}_s}=\infty)\cdot\auxFamily{\datingC\stdfC
      s}{{\leftoverm}_s}|\{\auxFamily{\datingC\butime s}{{\leftoverm}_s}=\infty\}\\
      &\phantom=+    \sum_{s=1}^{q}\MCS{s} \auxFamily{\datingC\fm s}{{\leftoverm}_s}(\auxFamily{\datingC\butime
      s}{{\leftoverm}_s}\in[\blob{q-1} s\RCp/2,{\leftoverm}_s))\cdot
    \auxFamily{\datingC\stdfC s}{{\leftoverm}_s}|\{\auxFamily{\datingC\butime s}{{\leftoverm}_s}\in[\blob {q-1} s\RCp/2,{\leftoverm}_s)\}
    \\&\phantom{=}+  \MUS{q}\availableC\stdfC{q}_{{\leftoverm}_q\NCn},
  \end{align*}
  where, $\availableC\stdfC{q}_{n}$ is an arbitrary $n$-pushforward of
  $\availableC\stdfC{q}$.  In the above expression, the first term accounts for pairs
  which we coupled at earlier steps and have not broken up yet; the second term accounts
  for all pairs which we coupled at any of the previous steps, broke up and have not
  recovered yet.  The third and last term accounts for $\NCn$-prestandard pairs which were
  uncoupled but recovered by the beginning of step $q$ and will try to get coupled in this
  step.

  For pairs belonging to the first term we can use
  Lemma~\ref{l_couplingProcedure}\ref{i_gettingCloser} and obtain that every pair in
  $\auxFamily{\datingC\stdfC s}{{\leftoverm}_s}|\{\auxFamily{\datingC\butime
    s}{{\leftoverm}}=\infty\}$
  is $\Const\ve^{1+\expoCloseness/2}\expo{-\const{\leftoverm}_s}$-matched.  For pairs
  belonging to the families appearing in the remaining two terms we do not have any
  estimate on the Wasserstein distance , therefore we can only bound it with $1$.

  Thus, we can use Corollary~\ref{c_wasserstein} and conclude that:
  \begin{align*}
    \wDist(F_{\ve*}^{n}\mu_{\ella},F_{\ve*}^{n}\mu_{\ellb})&\le\wDist(F_{\ve*}^{n-k\NCn}\mu_{\stdfCa_{k\NCn}},F_{\ve*}^{n-k\NCn}\mu_{\stdfCb_{k\NCn}})\\
                                                           &\le \Const\sum_{s=1}^{q}\MCS{s}\auxFamily{\datingC\fm s}{{\leftoverm}_s}(\auxFamily{\datingC\butime s}{{\leftoverm}_s}=\infty)\ve^{\expoCloseness/2}\expo{-\const{\leftoverm}_s}\\
                                                           &\phantom\leq +\Const\sum_{s=1}^{q}\MCS{s}\auxFamily{\datingC\fm s}{{\leftoverm}_s}(\auxFamily{\datingC\butime s}{{\leftoverm}_s}\in[\blob {q-1} s\RCp/2,\bar
                                                             m_s))+ \Const\MUS{q}
    \\&=\tI+\tII+\tIII.
  \end{align*}
  Let us estimate term $\tI$: by our estimate for $\MUS{s}$ we gather
  \begin{align*}
    \tI&\le\Const\mB\ve^{\expoCloseness/2}\sum_{s=1}^{q}\vartheta^s\expo{-\const{\leftoverm}_s}\\
    &\le\Const\mB\ve^{\expoCloseness/2}\sum_{s=1}^{q}\vartheta^s\expo{-\const(q-s)\RCp}\le\Const\mB\ve^{\expoCloseness/2}
    \vartheta^q.
  \end{align*}
This proves exponential decay for term $\tI$.
Similarly, for term $\tII$: using~\eqref{e_tailboundM} we obtain
  \begin{align*}
    \tII&\leq \Const\mB\smallness \sum_{s=1}^{q}\vartheta^s\expo{-\const(q-s)r}\leq
    \Const\mB\smallness \vartheta^q,
  \end{align*}
  by choosing $r$ sufficiently large.  We already proved, just after~\eqref{e_MUS},
  exponential decay for term $\tIII$, (\ie $\MUS{q}\le\vt^q$).  The proof then readily
  follows by collecting all above estimates.
\end{proof}
\subsection{Proof of the Main Theorem}\label{ss_proofMainTheorem}
Our Main Theorem is a direct consequence of Lemma~\ref{l_coupling} and the definition of
Wasserstein distance (see~\eqref{eq:Wasserstein}).  First, we owe to the reader the proof
of the the following
\begin{lem}\label{l_weakLimit}
  Let $\mu$ be an SRB measure and let $B(\mu)$ denote its ergodic basin (see
  Remark~\ref{r_endomorphismSRB}); then
  \begin{enumerate}
  \item \label{pp_weakLimit} $\mu$ is a weak limit of standard families.
  \item \label{pp_weakLimitTrap} if $\Leb(B(\mu)\cap\trap_i) > 0$, then $\mu$ is a weak
    limit of standard families that \arepinnedto{} $\trap_i$
  \end{enumerate}
\end{lem}
\begin{proof}
  For ease of notation, let $B = B(\mu)$; by Fubini's Theorem there exists a standard
  pair %
  $\ell = (\bG,\rho)$ (e.g. horizontal and with constant density) which intersects $B$ and
  so that $\mu_\ell(B) > 0$; let us denote by $\mu_{\ell,B}$ the normalized restriction of
  $\mu_\ell$ to $B$, \ie for any test function $\Phi$ we let
  $\mu_{\ell,B}(\Phi) = \mu_\ell(B)\invr\cdot\mu_\ell(\Id_B\cdot\Phi)$.  Observe that by
  definition of $B$:
  \begin{align*}
    \frac1n\sum_{k= 0}^{n-1}F_{\ve*}^{k}\mu_{\ell,B}\to\mu\text{ weakly as } n\to\infty.
  \end{align*}
  Fix $\varrho > 0$ be arbitrarily small; since the set $\bG^{-1}(B)\subset[a,b]$ is
  measurable, it can be approximated with a finite number of disjoint intervals up to
  error $\varrho$.  We conclude that there exist $N > 0$ and an $N$-prestandard family
  $\stdf_B$ so that $\|\mu_{\stdf_B}-\mu_{\ell,B}\|\nTV < \varrho$, where $\|\cdot\|\nTV$
  denotes the total variation norm.  Hence, for any $n$:
  \begin{align*}
    \left\|\frac1n\sum_{k = 0}^{n-1}F_{\ve*}^{k}\mu_{\stdf_B} - \frac1n\sum_{k =
    0}^{n-1}F_{\ve*}^k\mu_{\ell,B}\right\|\nTV < \varrho.
  \end{align*}
  Moreover, observe that for any $n$
  \begin{align*}
    \left\|\frac1n\sum_{k = 0}^{n-1}F_{\ve*}^{k}\mu_{\stdf_B} - \frac1{n-N}\sum_{k =
    N}^{n-1}F_{\ve*}^{k}\mu_{\stdf_B}\right\|\nTV     < 2\frac{N}{n-N}.
  \end{align*}
  Since $\frac1{N-n}\sum_{k = N}^{n-1}F_{\ve*}^k\mu_{\stdf_B}$ can be decomposed, by
  definition, in a standard family, the proof of~\ref{pp_weakLimit} follows choosing $n$
  sufficiently large.

  The proof of~\ref{pp_weakLimitTrap} also follows from the same argument, since our
  assumption guarantees that we can choose $\ell$ to \bepinnedto{} $\trap_i$;
  by~\eqref{e_protoInv} the standard family
  $\frac1{N-n}\sum_{k = N}^{n-1}F_{\ve*}^k\mu_{\stdf_B}$ \ispinnedto{} to $\trap_i$, which
  proves~\ref{pp_weakLimitTrap}.
\end{proof}
We now proceed to the proof of the Main Theorem, which we will now state, as promised in
Section~\ref{sec:results}, in a stronger version.  Let us denote with $\nTraps \le \nz$
the number of disjoint non-empty trapping sets $\trap_i$ and, for any $i$, recall that we
denote by $\nz_{i}$ the number of sinks $\sink j$ that are contained in $\trap_i$ (recall
also Footnote~\ref{f_nzIndependent}).  We now prove the following
\begin{thm}\label{t_mainTheoremForTraps}
  Assume that~\ref{a_noCobo},~\ref{a_discreteZeros},~\ref{a_almostTrivialp} hold and let
  $\sink i$ be a recurrent sink.  Then there exists a unique SRB measure $\mu_{\ve,i}$ so
  that $\supp\mu_{\ve,i}\subset\{\theta\in\trap_i\}$ (in particular, if
  $\trap_i = \trap_j$ then $\mu_{\ve,i} = \mu_{\ve,j}$).  The measure $\mu_{\ve,i}$ enjoys
  exponential decay of correlation for H\"older observables in the following sense.  There
  exist $C_1, C_2, C_3, C_4 > 0$ (independent of $\ve$) so that for any
  $\holexpa\in(0,3]$, $\holexpb\in(0,1]$ and any two functions
  $A\in\cC^\holexpa(\{\theta\in\trap_i\})$, $B\in\cC^\holexpb(\bT^2)$:
  \begin{align*}
    \left| \Leb(A\cdot B\circ F_\ve^n) - \Leb(A)\mu_{\ve,i}(B)\right|\leq
    C_1\sup_\theta\|A(\cdot, \theta)\|\nc\holexpa\sup_x\|B(x,\cdot)\|\nc\holexpb e^{-\holexpa\holexpb c_{\ve,i} n},
  \end{align*}
  where
  \begin{equation}\label{e_lowerBoundRateForTrap}
    c_{\ve,i}=
    \begin{cases}
      C_2\ve/\log\vei&\text{if }\nz_{i}=1,\\
      C_3\exp(-C_4\vei)&\text{otherwise.}
    \end{cases}
  \end{equation}
\end{thm}
Our Main Theorem then follows as a corollary:
\begin{cor}\label{c_mainTheoremCorollary}
  Under assumptions~\ref{a_noCobo},~\ref{a_discreteZeros},~\ref{a_almostTrivialp}
  and~\ref{a_fluctuationGap}, if $\ve>0$ is sufficiently small, $F_\ve$ admits exactly
  $\nTraps$ SRB measures.

  Under assumptions~\ref{a_noCobo},~\ref{a_discreteZeros},~\ref{a_almostTrivialp}
  and~\ref{a_fluctuation}, there exists a unique SRB measure $\mu_\ve$ for $F_\ve$; moreover $\mu_\ve$
  enjoys exponential decay of correlations as stated in the Main Theorem.
\end{cor}
\begin{proof}
  If~\ref{a_fluctuation} holds, then (see Remark~\ref{r_fluctuationAssumption})
  $\trap_1 = \bT$ and thus $\nz = \nz_{1}$.  Then, Theorem~\ref{t_mainTheoremForTraps},
  immediately implies existence and uniqueness of the SRB measure $\mu_\ve$ for $F_\ve$
  and that $\mu_\ve$ enjoys the required properties.

  On the other hand, if~\ref{a_fluctuationGap} holds, we want to prove that there cannot
  be any other SRB measure than the ones found by Theorem~\ref{t_mainTheoremForTraps}.  We
  can argue as follows: let $\mu$ be an SRB measure; as in the proof of
  Lemma~\ref{l_weakLimit}, there exists a standard pair $\ell$ so that
  $\mu_{\ell}(B(\mu)) > 0$; by~\eqref{e_eventuallyTrapped}, we gather that, for some
  $n > 0$ and $i\in\{1,\cdots,\nz\}$,
  $\Leb\left(B(\mu)\cap F_\ve^{-n}\{\theta\in\trap_i\}\right) > 0$.  Since by definition
  $B(\mu)$ is a $F_\ve$-invariant set, we gather
  $F_{\ve*}^n\Leb\left(B(\mu)\cap \{\theta\in\trap_i\}\right) > 0$, but since $F_{\ve}$ is a
  local diffeomorphism, $F_{\ve*}^n\Leb$ is absolutely continuous with respect to the
  Lebesgue measure.  Consequently, we have
  $\Leb\left(B(\mu)\cap \{\theta\in\trap_i\}\right) > 0$, hence $\mu = \mu_{\ve,i}$ by
  Lemma~\ref{l_weakLimit}\ref{pp_weakLimitTrap}.  We thus conclude that $F_\ve$ admits
  exactly $\nTraps$ SRB measures.
\end{proof}
\begin{proof}[Proof of Theorem~\ref{t_mainTheoremForTraps}]
  Let $\sink i$ be a recurrent sink and $\ell$ be a standard pair \pinnedto{} $\trap_i$.
  First, we prove that the sequence $F_{\ve*}^n\mu_\ell$ weakly converges to a SRB measure
  $\mu_{\ve,i}$ which is independent of $\ell$.  In fact,
  Remark~\ref{r_eventualInvariance} implies that if $n > \pint{\TForb\vei}$, the measure
  $F_\ve^{n}\mu_{\ell}$ can be decomposed in a standard family which \ispinnedto{} $\trap_i$.
  Then, for any $n > \pint{\TForb\vei}$, $m > 0$ and H\"older observable
  $B\in\cC^{\holexp}(\bT^2,\bR)$ (where $\holexp\in(0,1]$), Lemma~\ref{l_coupling} implies:
  \begin{align}\label{e_protoDecay}
    |F_{\ve*}^{n+m}\mu_{\ell}(B)-F_{\ve*}^{n}\mu_{\ell}(B)| &\le 
    \int_{\alphaset_m}d\fm(\alpha)\left|
    F_{\ve*}^n\mu_{\fellf{\alpha}_m}(B)-F_{\ve*}^n\mu_{\ell}(B)\right|\notag\\ &
    \le \Const\expo{-\holexp c_\ve (n-\pint{\TForb\vei})}\|B\|_{x,\holexp}\notag\\ &
    \le \Const\expo{-\holexp c_\ve n}\|B\|_{x,\holexp},
  \end{align}
  where $\stdf_m = \SFF{\alphaset_m}{\ell_m}$ is a standard $m$-pushforward of $\ell$,
  $\|B\|_{x,\holexp}=\sup_x\|B(x,\cdot)\|_{\cC^\holexp}$ and $c_\ve$
  satisfies~\eqref{e_lowerBoundRateForTrap} by Lemma~\ref{l_bootstrap}\ref{i_estimatemB}.

  In particular, $F_{\ve*}^n\mu_{\ell}(B)$ is a Cauchy sequence and, if $B$ is Lipschitz,
  this implies that the sequence of probability measures $F_{\ve*}^n\mu_{\ell}$ has a
  unique weak accumulation point.  Lemma~\ref{l_coupling} also implies that the sequence
  $F_{\ve*}^n\mu_{\ell'}$ has the same weak accumulation point for any $\ell'$ which
  \ispinnedto{} $\trap_i$ and that convergence is exponentially fast.  Let us denote by
  $\mu_{\ve,i}$ this accumulation point; by construction it is $F_\ve$-invariant.

  We now show that $\mu_{\ve,i}$ is indeed a SRB measure in the sense of Remark~\ref{r_endomorphismSRB}: consider a measurable partition $\{I_\xi\}_{\xi\in \Xi}$ of
  $\bT\times\trap_i$ in horizontal segments\footnote{ Notice that by
    Lemma~\ref{l_propertiesTrappingSet}\ref{pp_trapSink} $\trap_i$ contains a neighborhood
    of $\sink i$.} of length between $\delta/2$ and $\delta$ with indices in some measure
  space $\Xi$. That is, we let $I_\xi= [a_\xi,b_\xi]\times \{y_\xi\}$ for some
  $a_\xi,b_\xi,y_\xi\in\bT$ with $\delta/2\leq b_\xi-a_\xi\leq \delta$.  Let
  $\Leb_i = \Leb_{\{\theta\in\trap_i\}}$ be the restriction of Lebesgue measure to
  $\{\theta\in\trap_i\}$ normalized to be a probability measure.  Then by definition
  $\lim_{n\to\infty}\frac1n\sum_{k = 0}^{n-1}F_{\ve*}^k\Leb_i$ is a convex combination of
  SRB measures whose ergodic basin intersects $\{\theta\in\trap_i\}$ in a positive
  Lebesgue measure set.  By Lemma~\ref{l_weakLimit}\ref{pp_weakLimitTrap} and our previous
  argument, we conclude that any such measure has to be equal to $\mu_{\ve,i}$; we
  conclude that $\mu_{\ve,i}$ is itself an SRB measure.  By invariance of $\mu_{\ve,i}$
  and since it can be approximated by standard families \pinnedto{} $\trap_i$, we
  conclude using~\eqref{e_protoInv}, that $\supp \mu_{\ve,i}\subset\{\theta\in\trap_i\}$.
  Moreover, by our construction, it is clear that if $\trap_i = \trap_j$, then
  $\mu_{\ve,i} = \mu_{\ve,j}$.

  In order to conclude, we need to check that we have exponential decay of correlations
  for H\"older observables.  To start, let us first assume
  $A\in\cC^3(\{\theta\in\trap_i\})$ and consider the measurable partition
  $\{I_\xi\}_{\xi\in \Xi}$ introduced above.  Then we can write:
  \begin{equation}\label{eq:startingpoint0}
    \Leb_i(A\cdot B\circ F_\ve^n) =\int_{\Xi} \nu(d\xi)\int_{I_\xi} A(x,y_\xi) B\circ
    F_\ve^n(x,y_\xi) \deh x,
  \end{equation}
  where $\nu$ is the natural factor measure on $\Xi$.  Next, we set
  $\deh\hat \nu= \left[\int_{I_\xi}A(x,y_\xi) dx\right]d\nu$ and
  $\hat A_\xi(x)=A(x,y_\xi)\left[\int_{I_\xi}A(x,y_\xi) dx\right]^{-1}$.  In particular,
  $\int_{\Xi}\hat\fm(d\xi) = \Leb_i(A)$.  Then, by definition,
  $\ell_\xi=( \bG_\xi, \hat A_\xi)$ with $\bG_\xi(x)=(x,y_\xi)$, is a standard pair
  provided $\min A\geq 1$ and $\|A(x, \cdot)\|_{\cC^3}\leq C$ for some appropriate
  constant $C>0$ (see Section~\ref{subsec:Standardpairs} to recall definitions and
  notations).  Thus we can write
  \[
  \Leb_i(A\cdot B\circ F_\ve^n) =\int_{\Xi}
  \hat\nu(d\xi)\mu_{\ell_\xi}(B\circ F_\ve^n)=\int_{\Xi}
  \hat\nu(d\xi)F_{\ve*}^n\mu_{\ell_\xi}(B).
  \]
  since
  $|F_{\ve*}^n\mu_{\ell_\xi}(B)-\mu_{\ve,i}(B)| < \Const\expo{-\holexp c_\ve
    n}\|B\|_{x,\holexp}$,
  we obtain exponential decay of correlations, provided that $A$ satisfies the additional
  properties listed above.

  Let us now consider the case of a general $A$.  Obviously it suffices to have an
  estimate for $c A$, where $c$ is some small constant. But then we can write
  \[
  cA=\{c(A+\|A\|_{L^\infty})+1\}-\{c\|A\|_{L^\infty}+1\}
  \]
  which, for $c\leq (C-1)(2\|A(x, \cdot)\|_{\cC^3})^{-1}$, is the difference of two
  functions both satisfying the hypotheses above.  Thus, for all $A\in\cC^3$ and
  $B\in \cC^\holexp$, we have
  \begin{equation}\label{eq:smoothmain}
    \left| \Leb_i(A\cdot B\circ F_\ve^n) - \Leb_i(A)\mu_\ve(B)\right|\leq
    C_1\|A\|_{\theta, 3}\|B\|_{x,\holexp} e^{-\holexp c_\ve n}.
  \end{equation}
  To conclude, let us consider the case $A\in \cC^\holexpa$, $\holexpa<3$; for arbitrary
  $\varrho > 0$ let $A_\varrho\in\cC^3$ such that
  $\|A-A_\varrho\|_{\theta, 0}\leq \varrho^\holexpa \|A\|_{\theta,\holexpa}$, and
  $\|A_\varrho\|_{\theta,3}\leq
  \Const\varrho^{-3+\holexpa}\|A\|_{\theta,\holexpa}$.\footnote{
    Such approximate functions can be obtained by standard mollification.}  Then, by
  equations~\eqref{eq:startingpoint0},~\eqref{e_protoDecay} and~\eqref{eq:smoothmain}, we
  have
  \[
  \begin{split}
    \left| \Leb_i(A\cdot B\circ F_\ve^n) - \Leb_i(A)\mu_\ve(B)\right|\leq& \left|
      \Leb_i(A_\varrho\cdot B_\varrho\circ F_\ve^n) - \Leb_i(A_\varrho)\mu_\ve(B_\varrho)\right|\\
    &+\Const \varrho^{\holexpa} \|A\|_{\theta,\holexpa}\|B\|_{x,\holexp}\\
    \leq& (\Const\varrho^{-3+\holexpa} e^{-\beta c_\ve n}+\Const
    \varrho^\holexpa)\|A\|_{\theta,\holexp}\|B\|_{x,\holexp}.
  \end{split}
  \]
  Optimizing $\varrho$, as a function of $n$, we obtain
  $\varrho = e^{-\holexpb c_\ve n/3}$, which yields the wanted result (absorbing the
  factor $3$ in the constants $C_2$ and $C_3$).
\end{proof}
\begin{rem}
  Once again (see Remark~\ref{rem:lebtd}) Theorem~\ref{t_mainTheoremForTraps} is stated
  for Lebesgue measure just for simplicity.  In fact it holds for any initial measure that
  can be obtained as weak limit of standard families \pinnedto{} $\trap_i$.  In
  particular, we have exponential decay of correlations for initial conditions distributed
  according to the SRB measures $\mu_{\ve,i}$ themselves.  Only, in this case our
  estimate~\eqref{e_lowerBoundRateForTrap} for the decay rate is quite possibly not
  optimal when $\nz_i>1$ (e.g. see discussion in Subsection~\ref{subsec:onesink}).
\end{rem}
\section{Coupling: Proofs}\label{sec:coup-proof}
This is the most probabilistic part of the paper: it is then natural to adopt a more
probabilistic notation.  As we have painstakingly explained on which spaces the various
relevant random variables live and how their laws are defined, from now on we will simply
use $\bP$ and $\bE$ for designating, respectively, their probability and expectation,
unless some ambiguity might arise.

We start with an easy corollary of Lemma~\ref{l_goodSet} and Lemma~\ref{l_couplingStep}
with $\NCoup = \NCn$ which ensures that a $\Delta\ve$-matched coupling which is supported
on $\uhappy$ will geometrically decrease its Wasserstein distance after time $\NCn$ except in
an event of exponentially small probability.
\begin{cor}\label{c_fundamentalBound}
  For any $\bar\Delta>0$ there exists $\bar\ve>0$ so that the following holds.  For any
  $\ve\in(0,\bar\ve)$, $\Delta\in(0,\bar\Delta)$ and $\ellC$ a $\Delta\ve$-matched
  standard couple so that $\avgtheta{\ellCa}\in\uhappy$; let $\coupledStepSeq\stdfC\NCn$
  be the family obtained by applying Lemma~\ref{l_couplingStep} with $\NCoup = \NCn$ to
  the couple $\ellC$. Then:
  \begin{align*}
    \bP\big(\avgtheta{\coupledStepSeq\fellCa\NCn(\cdot)}\in\uhappy,
  \coupledStepSeq\fellC\NCn(\cdot)\text{ is
    $\wDist(\ellCa,\ellCb)\expo{-\TCn/2}$-matched}\big)\ge 1-\Const\expo{-\const\vei}.
  \end{align*}
\end{cor}
\begin{proof}
  Let us apply Lemma~\ref{l_goodSet} to $\coupledStepSeq\ellCa0$; by
  Lemma~\ref{l_couplingStep}\ref{i_subcurve} with $\NCoup = \NCn$ we then obtain:
  \begin{align*}
  \bP(\{\theta_{\NCn}\in\uhhappy, \zeta_\NCn\le -9\TCn/16\})\ge 1-\Const\expo{-\const\vei}.
 \end{align*}
  Since $\coupledStepSeq\stdfCa\NCn$ is a $\NCn$-pushforward of $\coupledStepSeq\ellCa0$,
  we can define the subset
  \begin{align*}
    \coupledStepSeq{\alphaset'}\NCn=\alphaMap{\NCn}(\{\theta_\NCn\in\hhappy_k,\zeta_\NCn<-9\TCn/16\}).
  \end{align*}
  Standard distortion estimates then imply that for any
  $\alpha\in\coupledStepSeq{\alphaset'}\NCn$ and $p,q\in\gap_\alpha$, we have
  $|\theta_{\NCn}(q)-\theta_{\NCn}(p)|\le\Const\ve$ and
  $\zeta_\NCn(q)\le\zeta_\NCn(p)+\Const\ve$.
  Lemma~\ref{l_couplingStep}\ref{i_matchedCloseness}, with $\NCoup = \NCn$,~\eqref{eq:wasserstein} and remembering that $\TCn$ is assumed to be large (in particular
  we can assume $\TCn>1$), concludes the proof of the corollary.
\end{proof}

\subsection{{Proof of Lemma~\ref{l_couplingProcedure}}}
\label{ss_proofCouplingProcedure}
We will define the sequence $\auxFamily{\stdfC}k$ and the random variables
$\auxFamily\butime k$ by an inductive construction in which we also introduce an auxiliary
sequence of random variables $\auxFamily\counter k:\auxFamily{\alphaset}k\to\bR$.  In
particular, such random variables will satisfy the following assumptions: let
$\Delta_k=\expo{-k\TCn/4}$
\begin{enumerate}[label=({\bf\roman*}), ref=(\bf\roman*)]
\item if $\auxFamily \butime{k}(\alpha)=\infty$, the couple $\auxFamily{\fellCf{\alpha}}k$
  is $\Delta_k\expo{-\auxFamily \counter
    k(\alpha)\TCn}\ve^{1+\expoCloseness/2}$-matched;\label{i_iaGettingCloser}
\item \label{i_iaUncoupledMatch} for any $l<k$, we have
  $\auxFamily\fm{k}(\auxFamily\butime k=l)=\auxFamily\fm{k-1}(\auxFamily\butime {k-1}=l)$
  and $\auxFamily{\stdfC}{k}|\{\auxFamily\butime
  {k}=l\}\in\eqc{\pFve^{\NCn}\auxFamily{\stdfC}{k-1}|\{\auxFamily\butime{k-1}=l\}}$;
  finally, $\auxFamily{\stdfC}{l}|\{\auxFamily\butime {l}=l-1\}$ is a coupling of
  $l\NCn$-prestandard families.
\item \label{iii-H-def} $\auxFamily \counter k\geq -2\expb$, for all $k\in\bN$, where
  $\expb$ is defined in~\eqref{e_trivialBound}. In addition, $\auxFamily \counter k\geq 0$
  for all $k\leq \Const \expoCloseness\ln\vei $.
\end{enumerate}
Properties~\ref{i_iaGettingCloser} and~\ref{iii-H-def} above trivially imply
item~\ref{i_gettingCloser} of our statement, provided $\ve$ is small enough, while
\ref{i_iaUncoupledMatch} corresponds exactly to item~\ref{i_compatibility}.  Intuitively,
the variable $\auxFamily \counter k$ is a measure of the closeness of a couple of standard
pairs, at iterate $k$, compared with our minimal expectation expressed by $\Delta_k$.  If
$\auxFamily \counter k$ becomes negative, then it means that the couple has failed to get
as close as we like in such a drastic manner that we give up on it and break it up.  Let
us specify our inductive construction.

For the base step, we define $\auxFamily{\stdfC}{0}=\ellC$,
$\auxFamily\counter 0=\frac\largestart{2\TCn}\log\vei$,
$\largestart\leq \expoCloseness/2 +\frac{\TCn}{\ln\vei}$, and
$\auxFamily\butime0=\infty$.

Next, we assume that $\auxFamily\stdfC l$, $\auxFamily\butime l$ and
$\auxFamily\counter l$ are already defined for $0\le l\le k$ and proceed to define
$\auxFamily\stdfC {k+1}$, $\auxFamily\butime {k+1}$ and $\auxFamily\counter {k+1}$. For
each $\alpha\in\auxFamily{\alphaset}{k}$ we will define a family
$\stdfC_\NCn(\alpha)\in\eqc{\pFve^{\NCn}\auxFamily{\fellCf{\alpha}}k}$ and for each
$\alpha'\in\alphaset(\alpha)$ we will define $\auxFamily \counter{k+1}(\alpha')$ and
$\auxFamily \butime {k+1}(\alpha')$.  We then define
$\auxFamily{\stdfC}{k+1}\in \eqc{\pFve^{\NCn}\auxFamily\stdfC k}$ by considering the
convex combination
\begin{align*}
  \auxFamily{\stdfC}{k+1}=\sum_{\alpha\in\auxFamily\alphaset{k}}\auxFamily{\fm}k(\{\alpha\})\stdfC_\NCn(\alpha).
\end{align*}
The random variables $\auxFamily\counter{k+1}$ and $\auxFamily\butime{k+1}$ are thus
naturally defined on $\auxFamily{\alphaset}{k+1}$.\footnote{ Note that there exists a
  natural measure-preserving immersion
  $\bf{i}:\auxFamily{\alphaset}{k+1}\to\auxFamily{\alphaset}{k}$, thus one can
  always see $\auxFamily\butime{k}$ as a random variable on $\auxFamily{\alphaset}{k+1}$
  and similarly for the other random variables. It is thus possible to view all the
  relevant random variables on the same natural probability space (given by the last time
  at which we are interested). We will use this implicitly in the following.}

We proceed to define $\stdfC_{\NCn}(\alpha)$ for $\alpha\in\auxFamily{\alphaset}k$.  There
are several possibilities:
\begin{itemize}
\item $\auxFamily \butime{k}(\alpha)=\infty$ and $\auxFamily \counter k(\alpha)\ge 0$: by
  inductive assumption~\ref{i_iaGettingCloser}, the couple $\auxFamily{\fellCf{\alpha}}k$
  is $\Delta_k\ve^{1+\expoCloseness/2}$-matched; we can thus apply
  Lemma~\ref{l_couplingStep} with $\NCoup = \NCn$ to $\auxFamily{\fellCf{\alpha}}k$ with
  $\Delta=\Delta_k\ve^{\expoCloseness/2}$ and define
  $\stdfC_\NCn(\alpha)=\mC\coupledStepSeq{\stdfC}{\NCn}(\alpha)+(1-\mC)\uncoupledStepSeq{\stdfC}{\NCn}(\alpha)$.

  If $\alpha'\in\uncoupledStepSeq{\alphaset}{}(\alpha)$, we let
  $\auxFamily\butime{k+1}(\alpha')=k$ and
  $\auxFamily\counter{k+1}(\alpha')=\auxFamily\counter{k}(\alpha)+1/4$.  Observe, \emph{en
    passant}, that by Lemma~\ref{l_couplingStep}\ref{i_leftoversPrestandard} with $\NCoup
  = \NCn$ the couple
  $\fellCf{\alpha'}$ is $\NCn+\Const (k+\log\vei)$-prestandard. Choosing $\ve$ sufficiently
  small we can ensure that $\fellCf{\alpha'}$ is indeed $(k+1)\NCn$-prestandard.

  If, on the other hand $\alpha'\in \coupledStepSeq{\alphaset}{}(\alpha)$, we let $\auxFamily
  \butime{k+1}(\alpha')=\infty$ and define $\auxFamily\counter{k+1}$ as:
  \begin{align*}
    \phantom{}\hskip.7cm \auxFamily \counter{k+1}(\alpha') =%
    \begin{cases}
      \auxFamily\counter{k}(\alpha) + \frac14 &\textrm{ if } \auxFamily{\fellCf{\alpha'}}{k+1} \textrm{ is
        $\Delta_k(\alpha,\expoCloseness)$-matched}\\
      \auxFamily\counter{k}(\alpha)-2\expb &\textrm{ otherwise; }
    \end{cases}
  \end{align*}
  where
  $\Delta_k(\alpha,\expoCloseness)=\Delta_k\expo{-(\auxFamily
    \counter{k}(\alpha)+\frac12)\TCn}\ve^{1+\frac\expoCloseness2}$.
\end{itemize}
\begin{rem}
  \label{r_good-start} Note that, by our assumptions,
  $\Delta_0(\alpha,\expoCloseness)\geq \ve^{1+\frac{3\expoCloseness}{4}}$. Thus the second
  option above can only occur if $k\geq \Const \expoCloseness\ln\vei$.
\end{rem}
\begin{itemize}
\item $\auxFamily \butime{k}(\alpha)=\infty$ and $\auxFamily\counter{k}(\alpha)<0$: we
  declare the couple to break up and let $\stdfC_\NCn(\alpha)$ be an arbitrary
  $\NCn$-pushforwards of $\auxFamily{\fellCf{\alpha}}k$. Also, for any
  $\alpha'\in\alphaset(\alpha)$ we let $\auxFamily \butime{k+1}(\alpha')=k$ and
  $\auxFamily \counter{k+1}(\alpha')=\auxFamily\counter{k}(\alpha)+1/4$.
\item if $\auxFamily \butime k(\alpha)<\infty$, we let $\stdfC_\NCn(\alpha)$ be an arbitrary
  $\NCn$-pushforward of $\auxFamily{\fellCf{\alpha}}k$. Also, for any
  $\alpha'\in\alphaset(\alpha)$ we let
  $\auxFamily \butime{k+1}(\alpha')=\auxFamily \butime{k}(\alpha)$ and
  $\auxFamily \counter{k+1}(\alpha')=\auxFamily \counter{k}(\alpha)+1/4$.
\end{itemize}
Inductive assumptions~\ref{i_iaGettingCloser},~\ref{i_iaUncoupledMatch} and
\ref{iii-H-def} then immediately follow from the above definitions and by
Remark~\ref{r_good-start} using~\eqref{e_trivialBound}.  As noticed earlier, they imply
items~\ref{i_gettingCloser} and~\ref{i_compatibility}.  We are now left to show
item~\ref{i_tailBoundM}: in order to do so, first observe that by definition and
Corollary~\ref{c_fundamentalBound} we have
\begin{equation}\label{e_crucialBound}
  \bP(\auxFamily\counter{k+1}-\auxFamily\counter{k}=-2\expb
  \cond\avgtheta{\auxFamily{\fellCa}k}\in\uhappy)\leq  \Const\expo{-\const\vei}.
\end{equation}
We now use the above inequality to prove a preliminary result:
\begin{sublem}\label{sublem:counter-control}
  For any $\smallnesst>0$, there exists $\vt\in(0,1)$ such that
  \begin{align*}
    \bP\left(\inf_{0\leq j\leq \TSl}\auxFamily \counter {k+j} < 0\right) &\le \smallnesst\vt^{k/\log\vei},
  \end{align*}
  where recall $\TSl=\pint{\RSl\log\vei}$ with $\RSl$ defined in Lemma~\ref{l_escapeFromAlcatraz}.
\end{sublem}
\begin{proof}
  Let us fix $p\in\bN$ sufficiently large to be specified later; for $j\ge0$, we define
  auxiliary random variables:
  \begin{align*}
    \auxFamily X j =%
    \begin{cases}%
      1  &\textrm{ if }\auxFamily\counter{(j+1)p\TSl}\geq \auxFamily\counter{jp\TSl} + \TSl\\
      -1 &\textrm{ otherwise.}
    \end{cases}
  \end{align*}
  Then we claim that if $p$ is sufficiently large, there exists $\beta'<\beta$ (where
  $\beta$ is defined in Lemma~\ref{l_escapeFromAlcatraz}) so that:\footnote{ The
    conditioning means simply that we specify the standard pair to which the process
    belongs at the iteration step
    $(j+1)p\TSl\NCn$.} \begin{equation}\label{e_youComeHereYou} \bP(\auxFamily
    X{j+1}=-1\cond\auxFamily{\alphaMapPrivate}{(j+1)p\TSl})\leq \Const \ve^{\beta'},
  \end{equation}
  provided $\ve$ is small enough.
  \begin{rem}
    Observe that, provided that $p > 4$, if $\auxFamily \butime{jp\TSl} < \infty$ (\ie a
    breakup already happened earlier than step $jp\TSl$), then we automatically have
    $\auxFamily X j = 1$ and thus~\eqref{e_youComeHereYou} trivially holds.  This is
    indeed the reason to define
    $\auxFamily\counter{jp\TSl+1} = \auxFamily\counter{jp\TSl} + 1/4$ after a breakup.
  \end{rem}
  Estimate~\eqref{e_youComeHereYou} suffices to conclude the proof of our sub-lemma: in
  fact observe that conditioning on the random variable
  $\auxFamily{\alphaMapPrivate}{(j+1)p\TSl}$ (defined in Remark~\ref{rem:lofp}) is finer
  than conditioning on $\auxFamily X0\cdots\auxFamily X{j}$.  By definition of conditional
  probability it follows
  \begin{align*}
    \bP(\auxFamily X{j+1}=-1\cond\auxFamily X{0}\cdots\auxFamily X{j})\leq \Const \ve^{\beta'}.
  \end{align*}%
  Observe that, by construction, for any $0\le s\le p\TSl$,
  $\auxFamily\counter{jp\TSl+s}-\auxFamily\counter{jp\TSl}\ge -2s\expb$.  Hence, we
  conclude that
  \begin{align}\label{e_estimateRWX}
    \auxFamily\counter{kp\TSl}-\auxFamily\counter0\ge  (1/2-p\expb)\TSl k +(1/2+p\expb)\TSl
    \sum_{l=0}^{k-1}\auxFamily X l.
  \end{align}
  Choose $c\in(0,1-\Const \ve^{\beta'})$ so that $(1/2-p\expb)+c(1/2+p\expb)>1/2$.  Thus
  Lemmata~\ref{l_simpleRandomWalkGame} and~\ref{l_comparison} imply that there exists
  $\vt\in(0,1)$ and $a>0$ such that
  \begin{align*}
    \bP\left(\sum_{l=0}^{k-1} \auxFamily X l \leq ck - a \right)\le \smallnesst\vt^{k}.
  \end{align*}
  Thus, provided that we choose $\largestart$ sufficiently large (relative to
  $a$),~\eqref{e_estimateRWX} implies that
  \begin{align*}
    \bP(\auxFamily\counter{kp\TSl}<k\TSl/2)\le \smallnesst\vt^k
  \end{align*}
  which, by Remark~\ref{r_good-start}, would conclude the proof of our sub-lemma.

  To really conclude, we are left with the proof of~\eqref{e_youComeHereYou}.  Notice that
  \begin{align*}
    \bP(\auxFamily X{j+1}=1\cond\auxFamily{\alphaMapPrivate}{(j+1)p\TSl}) \ge%
    \bP(\auxFamily X{j+1}=1\cond\auxFamily{\alphaMapPrivate}{(j+1)p\TSl}; A_\uhappy)\bP(A_\uhappy),
  \end{align*}
  where we have introduced the event
  $A_\uhappy=\{\avgthetaa{\auxFamily{}{(j+1)p\TSl+r}}\in\uhappy\;\forall r: \TSl\leq r\leq
  p\TSl\}$,
  where $\avgthetaa{n}$ denotes the average $\theta$ with respect to the marginal of the
  the first component of the standard coupling.  By Lemma~\ref{l_escapeFromAlcatraz} we
  have
  \[
  \bP(A_\uhappy)\geq1-(p-1)\TSl\ve^\beta.
  \]
  On the other hand, by iterating $(p-1)\TSl$ times~\eqref{e_crucialBound} we obtain that
  \begin{align*}
    \bP\bigg(\auxFamily{\counter}{(j+1)p\TSl}-\auxFamily{\counter}{(jp+1)\TSl}\ge\frac{(p-1)}4 \TSl
    \bigg|\auxFamily{\alphaMapPrivate}{(j+1)p\TSl};  A_\uhappy\bigg)\ge 1-\frac{(p-1)\TSl}{\expo{\const\vei}}.
  \end{align*}
  Thus, with overwhelming probability,
  \[
  \auxFamily{\counter}{(j+1)p\TSl}-\auxFamily{\counter}{jp\TSl}\geq \frac{(p-1)}4 \TSl-2\expb\TSl\geq \TSl,
  \]
  provided $p>4(1+2\expb)+1$. That is to say that $\auxFamily{X}{j+1}=1$, which
  proves~\eqref{e_youComeHereYou}.
\end{proof}
We can now prove item~\ref{i_tailBoundM}: by our inductive construction, Lemma
\ref{l_couplingStep}\ref{i_subcurve}, with $\NCoup = \NCn$, and
Sub-lemma~\ref{sublem:counter-control} we have:
\begin{align*}
  \bP(\auxFamily\butime {k+1}=\infty) &=
                                        \bP(\auxFamily\butime k=\infty,\auxFamily{\counter}k\ge0)\mC(\Delta_k\ve^{\expoCloseness/2})\\
                                      &\ge\bP(\auxFamily\butime {k-\TSl}=\infty, \inf_{j\leq\TSl}\auxFamily{\counter}{k-j}\geq 0)\prod_{j=0}^{\TSl-1}\mC(\Delta_{k-j}\ve^{\expoCloseness/2})\\
                                      &\ge\bP(\auxFamily\butime {k-\TSl}=\infty)\prod_{j=0}^{\TSl-1}\mC(\Delta_{k-j}\ve^{\expoCloseness/2})-\smallnesst\vt^{k/\log\vei}.
\end{align*}
The above inequality implies, for $\ve$ small enough,
\begin{align*}
  \bP(\auxFamily \butime {k}=\infty) &\ge 
  \prod_{j=0}^{k-1}\mC(\Delta_j\ve^{\expoCloseness/2})-\Const\smallnesst \ge\expo{-\const\smallnesst}.
\end{align*}
Finally, for $j>k$, again by our construction, Lemma~\ref{l_couplingStep} with
$\NCoup = \NCn$ and Sub-Lemma~\ref {sublem:counter-control},
\begin{align*}
  \bP(\auxFamily \butime k=\infty)-\bP(\auxFamily \butime j=\infty) &\le
  \left(1-\prod_{l=k}^{j-1}\mC(\Delta_l\ve^{\expoCloseness/2})\right)
  +\Const\smallnesst\vt^{k/\log\vei}\\&\le \Const\expo{-\const 
  k}\ve^{\expoCloseness}+\Const\smallnesst\vt^{k/\log\vei}
\end{align*}
provided we choose $\ve$ to be small enough.  The two inequalities above
prove~\eqref{e_firstMass} and~\eqref{e_tailboundM} and conclude the proof of our Lemma.
\qed
\subsection{Proof of Lemma~\ref{l_bootstrap}}\label{sec:coup-prooftwo}
\newcommand{\interval}{I} \newcommand{\lowbd}{p_\Bs''}%
First, we prove the following
\begin{sublem}\label{l_homeSink}
  Let $\ell$ be a standard pair \pinnedto{}  $\trap_i$; there exists $p_\Bs'>0$ so that:
  \begin{align}\label{e_homeSink}
    \mu_\ell(\theta_{\TSl\NCn}\in\hhappy_i)>p_\Bs',
  \end{align}
  where, recall, $\TSl=\pint{\RSl\log\vei}$ and $\RSl$ is the constant obtained in
  Lemma~\ref{l_escapeFromAlcatraz}.  Moreover, if $\nz_{i}=1$, $p_\Bs'$ can be
  chosen to be uniform in $\ve$; otherwise $p_\Bs'=\Const\expo{-\const\vei}$.
\end{sublem}
\begin{proof}
  If $\nz_{i}=1$, then $\uhhappy\cap\trap_i=\hhappy_i$ and the statement immediately
  follows by Lemma~\ref{l_escapeFromAlcatraz} and forward invariance of trapping sets~\eqref{e_protoInv}, which
  proves~\eqref{e_homeSink} for any $p_\Bs' < 1-\ve^{\beta}$.

  Assume now that $\nz>1$: Lemma~\ref{l_propertiesTrappingSet}\ref{pp_trivialInclusion}
  guarantees the existence of an $\aeps$-\admissiblep{\avgtheta\ell}{\sink i} of length bounded by
  $\TTrap$; Theorem~\ref{l_largeDevzLowerBound} then implies that
  \begin{align*}
    \mu_{\ell}({\theta_{\pint{\TTrap\vei}}\in\hhappy_i}) > \expo{-\const\vei}.
  \end{align*}
  We can then conclude by using Corollary~\ref{c_iterateGoodSet}, which
  proves~\eqref{e_homeSink} for $p_\Bs'=e^{-\const\vei}$.
\end{proof}
Let us now conclude the proof of Lemma~\ref{l_bootstrap}: let $C>0$ be the constant given
by Lemma~\ref{l_deepPurple} and let $J\subset\bT^{1}$ be the interval
$B(\theta_{i,-},C\sqrt\ve)$.  Subdivide $J$ into $\pint{\ve^{-1/2}}$ subintervals
$\{\interval_j\}$ of equal length $\Const\ve$.  By Theorem~\ref{thm:lclt} we can choose
$T>0$ sufficiently large such that for any standard pair $\ell$ \pinnedto{} $J$ and for
any $j$:
\begin{align*}
  \mu_{\ell}(\theta_{\pint{T\vei}}\in\interval_j) > \lowbd\ve^{1/2}.
\end{align*}
where $\lowbd>0$ is uniform in $\ve$ and independent of $\ell$.  Thus, combining the above
observation with Lemma~\ref{l_deepPurple}, we conclude that if $\ell$ is a standard pair
with $\avgtheta{\ell}\in\happy_i$, and we let $\Tgen=\pint{(\Rdp+1)\log\vei}$, where
$\Rdp$ is the constant found in Lemma~\ref{l_deepPurple}; then, for all $j$,
\begin{align}\label{e_uniformDistro}
  \mu_{\ell}(\theta_{\Tgen\NCn}\in\interval_j) > \frac12\lowbd\ve^{1/2}.
\end{align}
Hence, together with Sub-Lemma~\ref{l_homeSink}, we proved that if $\ella$ and $\ellb$ are
any two standard pairs \pinnedto{} the same trapping set $\trap_i$, the probability that
their $(\TSl+\Tgen)\NCn$-image have $\theta$-coordinates which are $\Const\ve$-close is at
least $\frac12p_\Bs'\lowbd$.

We now need to find pairs that are actually $\Delta\ve$-matched for some $\Delta>0$; this
task can be accomplished by the following argument.  Let $I\subset\bT^{1}$ be a fixed
interval of length $\delta$.  Since our maps are uniformly expanding in the $x$ direction,
there exist $M>0$ and $p\in(0,1)$ so that, given any standard pair $\ell$, we can
construct an $M$-pushforward of $\ell$ so that one of the standard pairs lies above the
interval $I$ and this standard pair has probability larger than $p$.  Moreover, by
Remark~\ref{r_flatDensity}, we can assume that this $\ell$ has a flat density, decreasing
$p$ by a factor $2/3$; the leftover pairs will be $\cO(1)$-prestandard.  We can then
construct the canonical coupling of all pairs which lie above $I$ and the independent
coupling of all other pairs.

We thus proved that if $\Rgen>\RSl+\Rdp+1$, then there exists a coupling
\begin{align*}
  \eqc{\pFve^{\pint{\Rgen\log\vei}\NCn}\ellC}\ni \mB'\bootstrapped\tstdfC{N}+(1-\mB')\notBootstrapped\tstdfC{N},
\end{align*}
where $\mB'=\frac12 p p_\Bs' p_\Bs''$ and $\bootstrapped\tstdfC{N}$ is a
$\Delta\ve$-matched standard coupling whose components are supported on
$\bT\times\happy_i$ and $\notBootstrapped\tstdfC{N}$ is a $\cO(1)$-prestandard coupling.
In order to conclude the proof of our statement we need to obtain couplings which are
$\ve^{1+\expoCloseness}$-matched: to do so it suffices to apply iteratively
Lemma~\ref{l_couplingStep} with $\NCoup = \NCn$ to pairs in
$\bootstrapped\tstdfC{N}$. Using Corollary~\ref{c_fundamentalBound} (as we did in the
proof of Sub-lemma~\ref{sublem:counter-control}) we conclude that there exists $C'$ so that
a substantial portion of the mass of a $(C'\expoCloseness\log\vei)\NCn$-pushforward of
$\bootstrapped\tstdfC{N}$ will be $\ve^{1+\expoCloseness}$-matched and the leftover pairs
will be $\cO(\expoCloseness\log\vei)$-prestandard, which concludes our proof choosing
$\RBs = \RSl + \Rdp + 1 + C'\expoCloseness$.  \qed

\section{Conclusions and open problems}\label{sec:conclusion}
In this work we have discussed the case in which the dynamics of the fast variable is
given by a one dimensional expanding map.  In this setting we proved exponential decay of
correlation for an open set of partially hyperbolic endomorphisms of the two-torus
$\bT^2$. To keep the exposition as terse as possible, in particular we did not
investigated in detail the adiabatic, metastable, regime. This can be done similarly
to~\cite{WeFr, Kifer09} and is postponed to future work.

Another natural issue, already pointed out in Section~\ref{sec:results}, is the necessity
of hypothesis~\ref{a_almostTrivial}.  In our scheme of proof it is certainly
needed. Nevertheless, we provided an example in Section~\ref{sec:no-ex} that does not
satisfy~\ref{a_almostTrivial} and yet numerical computations seems to show that it behaves
similarly to the examples for which~\ref{a_almostTrivial} is satisfied~\cite{Volk}.  This
suggests that our understanding of the possible mechanisms of convergence to equilibrium
is partial at best, and that further thought is much needed.

Next, observe that assumption~\ref{a_discreteZeros} is substantial: the set of $\omega$
such that $\{\theta\st\bar\omega(\theta)=0\}=\emptyset$ is open.  If $\bar\omega$ has no
zeros, then the averaged motion is a rotation, with no sinks or sources; the main
mechanism to establish a coupling argument would then be the diffusion centered on the
rotation. Note however that this would require a time scale $\ve^{-2}$ to bring any two
standard pairs close enough to couple them,~\cite{DimaAveraging}.  This situation is of
considerable interest in non-Equilibrium Statistical Mechanics when the dynamics is
Hamiltonian and the slow variables are the energies of nearby, weakly interacting,
systems, see~\cite{Dimaliverani}. In this case we conjecture that, generically, the system
should be mixing and the correlations should decay exponentially with rate which would be,
at best, $\ve^2$.  However, to prove such a result stands as a substantial challenge in
the field.

Finally, it would be very interesting to prove analogous results for the case in which the
fast variable evolves according to a more general hyperbolic system and when the slow
variable is higher dimensional. The first generalization could prove rather difficult when
trying to extend, \eg, the needed results of our paper~\cite{DeL1} to the case of flows or
systems with discontinuities. The second does not pose any particular problem as far as
the results in~\cite{DeL1} are concerned.  The difficulties come instead from the fact
that in higher dimension a generic dynamics has many different types of $\omega$ limit
sets (not just sinks or the whole space, as it is in one dimension) and these
possibilities give rise to situations to which the ideas put forward in the present paper
may not easily apply.
\appendix
\section{Random walks}\label{a_rw}
We start by recalling a well known fact about one dimensional random walks (it can be
obtained, e.g., from Cramer's Theorem).
\begin{lem}\label{l_simpleRandomWalkGame}
  Let $\iid_k\in\{-1,1\}$ be a sequence of i.i.d.\ random variables with distribution
  $\prob{\iid_i=1}=p$ for $p\in(0,1)$.  Let $\siid_0=0$ and for $n>0$, define:
  $\siid_n=\sum_{j=1}^n\iid_j$.  For any $c<2p-1$ there exist $\vt,\lt\in(0,1)$ such that,
  for any $k\in\bN$ and $a\in\bR$:
  \[
  \prob{\siid_k \leq kc - a}\leq\lt^a\vt^k.
  \]
\end{lem}

Next, we introduce an useful comparison argument:
\begin{lem}\label{l_comparison}
  Let $\iid_k\in\{-1,1\}$ be a sequence of independent random variables and let
  $\iidb_k\in\{-1,0,1\}$ be a random process such that
  \[
  \prob{\iidb_{k+1}=1|\iidb_1\cdots\iidb_{k}}\geq \prob{\iid_{k+1}=1}.
  \]
  For $n>0$ define the random variables
  \begin{align*}
    \siid_n&=\sum_{j=1}^n\iid_j&
    \siidb_n&=\sum_{j=1}^n\iidb_j
  \end{align*}
  where $N>0$ is some fixed natural number (if $n=0$ we let them all equal to $0$); then
  for each $n\in\bN$ and $L\in\bZ$:
  \begin{align}\label{e_comparisonRW}
    \prob{\siidb_{k}\leq L}&\leq\prob{\siid_{k}\leq L}.
  \end{align}
  In particular, if $\tau_\siid$ is the hitting time $\tau=\inf\{k\st\siid_k \ge L\}$ and
  $\tau_\siidb=\inf\{k\st\siidb_k \ge L\}$ we have, for any $s>0$:
  \begin{align*}
    \prob{\tau_\siidb > s} \le \prob{\tau_\siid > s}
  \end{align*}
\end{lem}
\begin{proof}[{Proof (see~{\cite[Proposition~2.4]{Dima}})}]
  The proof amounts to design a suitable coupling $(\iid^*_k,\iidb^*_k)$ of the random
  variables $\iid_k$ and $\iidb_k$.  Let us introduce an auxiliary sequence $U_k$ of
  independent random variables uniformly distributed on $[0,1]$ and define the random
  variables
  \begin{align*}
    \iid^*_k&=
    \begin{cases}
      +1  &\text{if}\ U_k<\prob{\iid_k\geq 1}\\
      -1&\text{otherwise}
    \end{cases}
    \intertext{and}
    \iidb^*_k&=
    \begin{cases}
      +1 &\text{if}\ U_k<\prob{\iidb_k=
        1|\iidb_1=\iidb^*_1,\cdots,\iidb_{k-1}=\iidb^*_{k-1}}\\
      -1 &\text{if}\ U_k\geq1-\prob{\iidb_k =
        -1|\iidb_1=\iidb^*_1,\cdots,\iidb_{k-1}=\iidb^*_{k-1}}\\
      0&\text{otherwise}.
    \end{cases}
  \end{align*}
  We then define
  \begin{align*}
    \siid^*_n&=\sum_{j=1}^n\iid^*_j&
    \siidb^*_n&=\sum_{j=1}^n\iidb^*_j
  \end{align*}
  Clearly $\iid^*_k$ (\resp  $\iidb^*_k$) has the same distribution of $\iid_k$
  (\resp $\iidb_k$) and consequently $\siid^*_k$ (\resp $\siidb^*_k$) has the same
  distribution of $\siid_k$ (\resp $\siidb_k$).  Moreover, $\iid^*_k\leq\iidb^*_k$ by
  design which in turn implies that $\siid^*_k\leq\siidb^*_k$.  This concludes the proof
  of our lemma.
\end{proof}

\bibliographystyle{abbrv} \bibliography{rw}
\end{document}